\documentclass{article}   	

\usepackage{microtype}
\usepackage{graphicx}
\usepackage{booktabs} 

\usepackage[accepted]{icml2019}

\usepackage{amssymb,amsmath,amsthm}
\usepackage{enumitem}
\usepackage{color}

\usepackage{caption}
\usepackage{subcaption}

\newtheorem{Ass}{Assumption}

\newtheorem{Thm}{Theorem}
\newtheorem{Lem}{Lemma}
\newtheorem{Cor}{Corollary}

\newtheorem{Fact}{Fact}
\newtheorem{Rem}{Remark}

\newcommand{\norm}[1]{\Vert{#1}\Vert}
\newcommand{\abs}[1]{\vert{#1}\vert}

\newcommand{\mb}[1]{{\mathbf{#1} }}
\newcommand{\mbb}[1]{{\mathbb{#1} }}
\newcommand{\mc}[1]{{\mathcal{#1} }}

\newcommand{\bs}[1]{{\boldsymbol{#1}}}

\newcommand*{\tran}{^{\mkern-1.5mu\mathsf{T}}}

\newcommand*{\defeq}{\overset{\Delta}{=}}

\icmltitlerunning{On the Linear Speedup Analysis of Communication Efficient Momentum SGD for Distributed Non-Convex Optimization}

\begin{document}
\twocolumn[
\icmltitle{On the Linear Speedup Analysis of Communication Efficient Momentum SGD for Distributed Non-Convex Optimization}
 
\icmlsetsymbol{equal}{*}           
\begin{icmlauthorlist}
\icmlauthor{Hao Yu}{alibaba}
\icmlauthor{Rong Jin}{alibaba}
\icmlauthor{Sen Yang}{alibaba}
\end{icmlauthorlist}

\icmlaffiliation{alibaba}{Machine Intelligence Technology Lab, Alibaba Group (U.S.) Inc., Bellevue, WA}
\icmlcorrespondingauthor{Hao Yu}{eeyuhao@gmail.com}
\icmlkeywords{Distributed Stochastic Optimization}

\vskip 0.3in
]

\printAffiliationsAndNotice{}

\begin{abstract}
Recent developments on large-scale distributed machine learning applications, e.g., deep neural networks, benefit enormously from the advances in distributed non-convex optimization techniques, e.g., distributed Stochastic Gradient Descent (SGD).  A series of recent works study the linear speedup property of distributed SGD variants with reduced communication. The linear speedup property enable us to scale out the computing capability by adding more computing nodes into our system. The reduced communication complexity is desirable since communication overhead is often the performance bottleneck in distributed systems. Recently, momentum methods are more and more widely adopted in training machine learning models and can often converge faster and generalize better.  For example, many practitioners use distributed SGD with momentum to train deep neural networks with big data. However,  it remains unclear whether any distributed momentum SGD possesses the same linear speedup property as distributed SGD and has reduced communication complexity.  This paper fills the gap by considering a distributed communication efficient momentum SGD method and proving its linear speedup property. 
\end{abstract}

\section{Introduction}
Consider distributed non-convex optimization scenarios where $N$ workers jointly 
solve the following consensus optimization problem:
{\small
\begin{align}
\min_{ \mb{x}\in \mbb{R}^{m}}  \quad f(\mb{x}) =\frac{1}{N}\sum_{i=1}^N \underbrace{\mathbb{E}_{\xi_i \sim \mc{D}_i} [F_i(\mb{x}; \xi_i)]}_{\defeq f_i(\mb{x})} \label{eq:sto-opt}
\end{align}
}%
where $f_i(\mb{x}) \overset{\Delta}{=} \mathbb{E}_{\xi_i \sim \mc{D}_i} [F_i(\mb{x}; \xi_i)]$ are smooth non-convex functions with possibly different distributions $\mc{D}_i$.   Problem \eqref{eq:sto-opt} is particularly important in deep learning since it captures data parallelism for training deep neural networks. In deep learning with data parallelism, each $F_i(\cdot; \cdot) = F(\cdot; \cdot) $ represents the common deep neural network to be jointly trained and each $\mc{D}_i$ represents the distribution of the local data set accessed by worker $i$.  The scenario where each $\mc{D}_i$ are different also captures the {\bf federated learning} setting recently proposed in \cite{McMahan17AISTATS} where mobile clients with private local training data and slow intermittent network connections cooperatively train a high-quality centralized model.

Stochastic Gradient Descent (SGD) (and its momentum variants) have been the dominating methodology for solving machine learning problem. For large-scale distributed machine learning problems, such as training deep neural networks, a parallel version of SGD, known as parallel mini-batch SGD, is widely adopted \cite{Dean12NIPS, Dekel12JMLR, Li14NIPS}. With $N$ parallel workers, parallel mini-batch SGD can solve problem \eqref{eq:sto-opt} with $O(1/\sqrt{NT})$ convergence, which is $N$ times faster\footnote{If an algorithm has $O(1/\sqrt{T})$ convergence, then it takes $1/\epsilon^2$ iterations to attain an $O(\epsilon)$ accurate solution. Similarly, if another algorithm has $O(1/\sqrt{NT})$ convergence, then it takes $1/(N\epsilon^2)$ iterations, which is $N$ times smaller than $1/\epsilon^2$, to attain an $O(\epsilon)$ accurate solution. In this sense, the second algorithm is $N$ times faster than the first one.} than the $O(1/\sqrt{T})$ convergence attained by SGD over a single node \cite{GhadimiLan13SIOPT, Lian15NIPS}. Obviously, such linear speedup with respect to (w.r.t.) number of workers is desired in distributed training as it implies perfect computation scalability by increasing the number of used workers. However, such linear speedup is difficult to harvest in practice because the classical parallel mini-batch SGD requires all workers to synchronize local gradients or models at {\bf every} iteration such that inter-node communication cost becomes the performance bottleneck. To eliminate potential communication bottlenecks, such as high latency or low bandwidths, in computing infrastructures, many distributed SGD variants have been recently proposed. For example, \cite{Lian17NIPS, Jiang17NIPS,Assran18ArXiv} consider decentralized parallel SGD where global gradient aggregations used in the classical parallel mini-batch SGD are replaced with local aggregations between neighbors. To reduce the communication cost at each iteration, compression or sparsification based parallel SGD are considered in \cite{Seide14Interspeech,Strom15Interspeech,Aji17EMNLP, Wen17NIPS, Alistarh17NIPS}. Recently,  \cite{Yu18ArXivAAAI, WangJoshi18ArXiv, Jiang18NIPS} prove that certain parallel SGD variants that strategically skip communication rounds can achieve the fast $O(1/\sqrt{NT})$ convergence with significantly less communication rounds.  See Table \ref{table:sgd-communication} for a summary on recent works studying distributed SGD with $O(1/\sqrt{NT})$ convergence and reduced communication complexity.

It is worth noting that existing convergence analyses on distributed methods for solving problem \eqref{eq:sto-opt} focus on parallel SGD {\bf without momentum}.  In practice, momentum SGD is more commonly used to train deep neural networks since it often can converge faster and generalize better \cite{Krizhevsky12NIPS, Sutskever13ICML,Yan18IJCAI}.  For example, momentum SGD is suggested for training ResNet for image classification tasks to obtain the best test accuracy \cite{He16CVPR}. See Figure \ref{fig:1} for the comparison of test accuracy between training with momentum and training without momentum.\footnote{As observed in Figure \ref{fig:1}, the final test accuracy of momentum SGD is roughly $2.5\%$ better than that of vanilla SGD (without momentum) over CIFAR10. In practice, people usually use momentum SGD to train ResNet in both single GPU and multiple GPU scenarios. Some practitioners conjecture that it is possible to avoid the performance degradation of vanilla SGD if its hyper-parameters are well tuned. However, even if this conjecture is true, the hyper-parameter tuning can be extremely time-consuming. } In fact, while previous works \cite{Lian17NIPS, Stich18ArXiv, Yu18ArXivAAAI, Jiang18NIPS} prove that distributed vanilla SGD (without momentum) can train deep neural networks with $O(1/\sqrt{NT})$ convergence using significantly fewer communications rounds, most of their experiments use momentum SGD rather than vanilla SGD (to achieve the targeted test accuracy). In addition, previous empirical works \cite{Chen16ICASSP, Lin18ArXiv} on distributed training for deep neural networks explicitly observe that momentum is necessary to improve the convergence and test accuracy. In this perspective, there is a huge gap between the current practices, i.e., using momentum SGD rather than vanilla SGD in distributed training for deep neural networks, and existing theoretical analyses such as \cite{Yu18ArXivAAAI,  WangJoshi18ArXiv, Jiang18NIPS} studying the convergence rate and communication complexity of SGD  without momentum. Momentum methods are originally proposed by Polyak in \cite{Polyak64} to minimize deterministic strongly convex quadratic functions.  Its convergence rate for deterministic convex optimization, which is not necessarily strongly convex, is established in \cite{Ghadimi14ArXiv}. For non-convex optimization, the convergence for deterministic non-convex optimization is proven in \cite{Zavriev93} and the $O(1/\sqrt{T})$ convergence rate for stochastic non-convex optimization is recently shown in \cite{Yan18IJCAI}. However, to our best knowledge, it remains as an open question: ``{\it Whether any distributed momentum SGD can achieve the same  $O(1/\sqrt{NT})$ convergence, i.e., linear speedup w.r.t. number of workers, with reduced communication complexity as SGD (without momentum) methods in \cite{Yu18ArXivAAAI,  WangJoshi18ArXiv, Jiang18NIPS} ?}"

\vspace{-1.em}
\begin{figure}[h!] 
\centering
\includegraphics[width=0.42\textwidth]{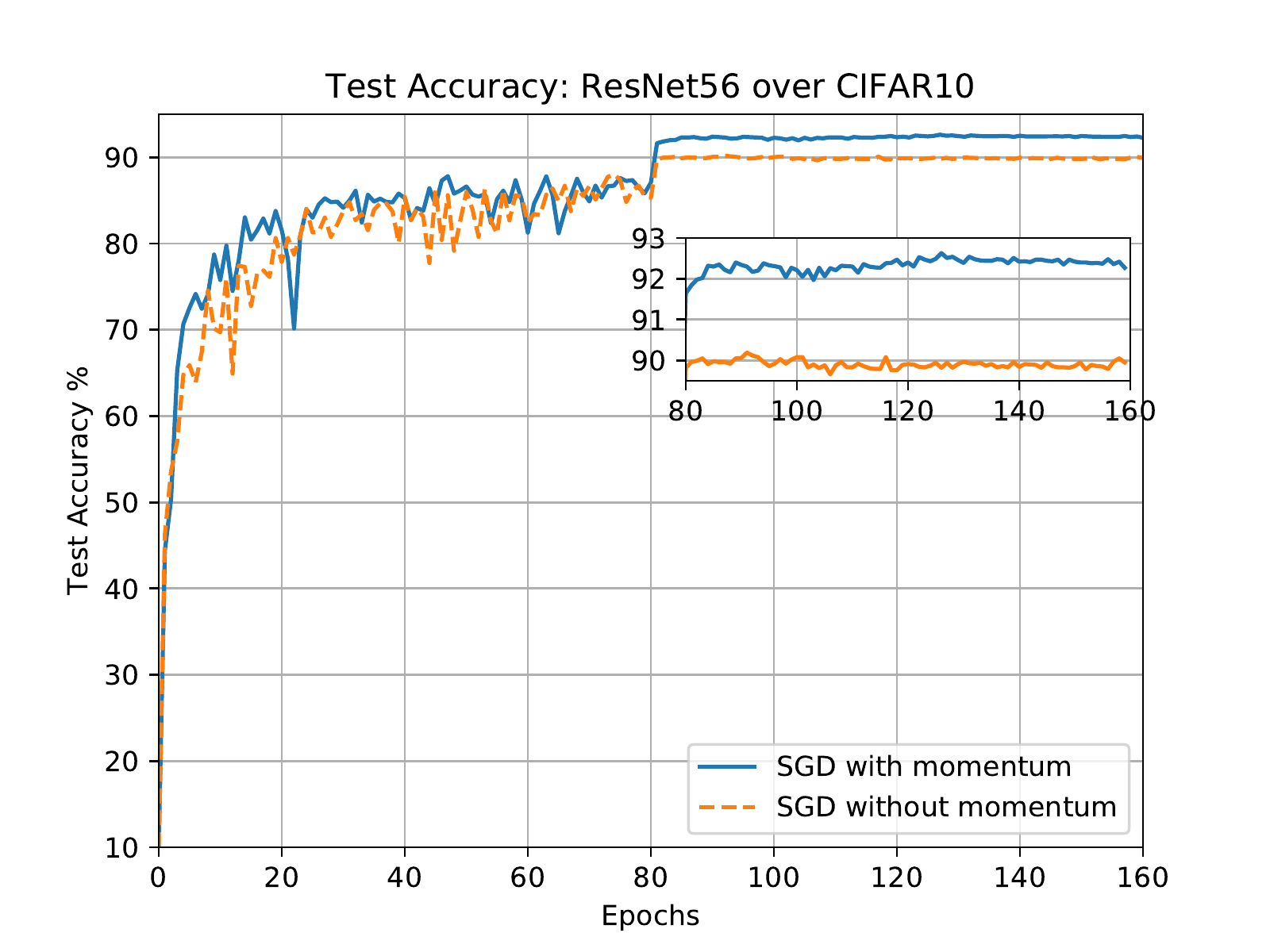}
\caption{Test accuracy performance of ResNet56 over CIFAR10 when trained with SGD with and without momentum on a single GPU using hyper-parameters suggested in \cite{He16CVPR}.}\label{fig:1}
\end{figure}

\begin{table*}[th]
	\caption{Number of communication rounds involved in $T$ iterations of existing distributed SGD {\bf without momentum} with $O(1/\sqrt{NT})$ convergence for {\bf non-convex} stochastic optimization \eqref{eq:sto-opt}. }
	\label{table:sgd-communication}
	\begin{center}
		\begin{small}
			\begin{sc}
				\begin{tabular}{lccc}
					\toprule
					Reference & {\small Identical $f_i(\mb{x})$, i.e.,$\kappa=0$} &  {\small Non-identical $f_i(\mb{x})$, i.e., $\kappa\neq0$}  & {\small Extra Assumptions}\\
					\midrule
					\cite{Lian17NIPS}    & $O(T)$ & $O(T)$& No \\
					\cite{Yu18ArXivAAAI} & $O(N^{3/4}T^{3/4})$ & $O(N^{3/4}T^{3/4})$ & Bounded Gradients\\
					\cite{Jiang18NIPS}    & $O(N^{5/2}T^{1/2})$ &  $O(N^{5/4}T^{3/4})$ & No \\
					\cite{WangJoshi18ArXiv}    &  $O(N^{3/2} T^{1/2})$& Not Applicable&  No       \\
					This Paper &  $O(N^{3/2} T^{1/2})$&$O(N^{3/4}T^{3/4})$ &  No       \\
					\bottomrule
				\end{tabular}
			\end{sc}
		\end{small}
	\end{center}
\end{table*}

{\bf Our Contributions}: In this paper, we considers parallel restarted SGD with momentum (described in Algorithm \ref{alg:prsgd-momentum}), which can be viewed as the momentum extension of parallel restarted SGD, also referred as local SGD, considered in \cite{Stich18ArXiv, Yu18ArXivAAAI}. Such a method is also faithful to the common practice ``model averaging with momentum" used by practitioners for training deep neural networks in \cite{Chen16ICASSP, McMahan17AISTATS, Su18ArXiv, Lin18ArXiv}.  We show that parallel restarted SGD with momentum can solve problem \eqref{eq:sto-opt} with $O(1/\sqrt{NT})$ convergence, i.e., achieving linear speedup w.r.t. number of workers. Moreover,  to achieve the fast $O(1/\sqrt{NT})$ convergence,  $T$ iterations of our algorithm  only requires $O(N^{3/2}T^{1/2})$ communication rounds when all workers can access identical data sets; or $O(N^{3/4}T^{3/4})$ communication rounds when workers access non-identical data sets. To our knowledge, this is the first time that a distributed momentum SGD method for non-convex stochastic optimization is proven to possess the same linear speedup property (with communication reduction) as distributed SGD (without momentum) in \cite{Lian17NIPS, Yu18ArXivAAAI,  WangJoshi18ArXiv, Jiang18NIPS}.

Recall that momentum SGD degrades to vanilla SGD when its momentum coefficient is set $0$. Algorithm \ref{alg:prsgd-momentum} with $\beta = 0$ degrades to the parallel SGD methods (without momentum) considered in \cite{Yu18ArXivAAAI,  WangJoshi18ArXiv, Jiang18NIPS}. Our communication complexity results for parallel SGD without momentum cases also improve the state-of-the-art.  As shown in Table \ref{table:sgd-communication},  the number of required communication rounds shown in this paper is the fewest in both identical training data set case and non-identical data set case.  In particular, this paper relaxes the bounded gradient moment assumption used in \cite{Yu18ArXivAAAI} to a milder bounded variance assumption and reduces the communication complexity for the case with identical training data. Our analysis applies to the distributed training scenarios where workers access non-identical training sets, e.g., training with sharded data or federated learning, that can not be handled in \cite{WangJoshi18ArXiv}.

This paper further considers momentum SGD with decentralized communication and proves that it possess the same linear speedup property as its vanilla SGD (without momentum) counterpart considered in \cite{Lian17NIPS}.

\section{Parallel Restarted SGD with Momentum}

\subsection{Preliminary}

Following the convention in (distributed) stochastic optimization, we assume each worker can independently evaluate unbiased stochastic gradients $\nabla F_i(\mb{x}_i; \xi_i), \xi_i \sim \mathcal{D}_i$ using its own local variable $\mb{x}_i$. Throughout this paper, we have the following standard assumption:  
\begin{Ass}\label{ass:basic} Problem \eqref{eq:sto-opt} satisfies the following:
\vspace{-0.5em}
\begin{enumerate}
\item {\bf Smoothness:} Each function $f_i(\mb{x})$ in problem \eqref{eq:sto-opt} is smooth with modulus $L$.
\vspace{-0.5em}
\item {\bf Bounded variances:} There exist $\sigma >0$ and $\kappa > 0$ such that 
\begin{align}
\mathbb{E}_{\xi \sim \mc{D}_i} [\norm{\nabla F_i(\mb{x};\xi) - \nabla f_i(\mb{x})}^2 ] &\leq \sigma^2, \forall i, \forall \mb{x}\\
\frac{1}{N} \sum_{i=1}^N \norm{\nabla f_i(\mb{x}) - \nabla f(\mb{x})}^2 &\leq \kappa^2, \forall \mb{x}
\end{align}
\end{enumerate}
\end{Ass}
\vspace{-0.3em}
Note that $\sigma^2$ in Assumption \ref{ass:basic} quantifies the variance of stochastic gradients at local worker, and $\kappa^2$ quantifies the deviations between each worker's local objective function $f_i(\mb{x})$. For distributed training of neural networks, if all workers can access the same data set, then $\kappa =0$. Assumption \ref{ass:basic} was previously used in \cite{Lian17NIPS, Wen17NIPS, Alistarh17NIPS, Jiang18NIPS}. The bounded variance assumption in Assumption \ref{ass:basic} is milder than the bounded second moment assumption used in \cite{Stich18ArXiv, Yu18ArXivAAAI}.

Recall that if function $f(\cdot)$ is smooth with modulus $L$, then we have the key property that $f(\mb{x}) \leq f(\mb{y}) + \langle \nabla f(\mb{y}), \mb{x} - \mb{y}\rangle + \frac{L}{2} \norm{\mb{y}-\mb{x}}^2, \forall \mb{x}, \mb{y}$ where $\langle\cdot, \cdot \rangle$ represents the inner product of two vectors. Another two useful properties implied by Assumption \ref{ass:basic} are summarized in the next two lemmas.

\begin{Lem}\label{lm:expected-gradient}
	Consider problem \eqref{eq:sto-opt} under Assumption \ref{ass:basic}.  Let $\mb{g}_i$ be mutually independent unbiased stochastic gradients sampled at points $\mb{x}_i$ such that $\mbb{E}[\mb{g}_i] = \nabla f_i(\mb{x}_i), \forall i\in\{1,2,\ldots,N\}$. Then, we have 
	\begin{align}
	\mbb{E}[\norm{\frac{1}{N}\sum_{i=1}^N \mb{g}_i}^2] \leq  \frac{1}{N}\sigma^{2}  +  \mathbb{E} [ \Vert \frac{1}{N} \sum_{i=1}^{N} \nabla f_{i}(\mathbf{x}_{i})\Vert^{2}] \nonumber 
	\end{align}	
\end{Lem}
\begin{proof}
	See Supplement \ref{sec:pf-lm-expected-gradient}
\end{proof}
Since $\mbb{E}[\mb{g}_i] = \nabla f_i(\mb{x}_i)$, we have $\mbb{E}[\frac{1}{N}\sum_{i=1}^N \mb{g}_i] = \frac{1}{N} \sum_{i=1}^{N} \nabla f_{i}(\mathbf{x}_{i})$. The implication of Lemma \ref{lm:expected-gradient} is:  while $\mb{g}_i$ are sampled at different points $\mb{x}_i$, the variance of $\frac{1}{N}\sum_{i=1}^N \mb{g}_i$, which is equal to $\mbb{E}[\norm{\frac{1}{N}\sum_{i=1}^N \mb{g}_i - \frac{1}{N} \sum_{i=1}^{N} \nabla f_{i}(\mathbf{x}_{i})}^2] = \mbb{E}[\norm{\frac{1}{N}\sum_{i=1}^N \mb{g}_i}^2] -  \mathbb{E} [ \Vert \frac{1}{N} \sum_{i=1}^{N} \nabla f_{i}(\mathbf{x}_{i})\Vert^{2}]$, scales down by $N$ times when $N$ workers sample gradients independently (even though at different $\mb{x}_i$). Note that if each $\mb{g}_i$ are independently sampled by $N$ workers at the same point $\mb{x}_i \equiv \mb{x}, \forall i$ (as in the classical parallel mini-batch SGD), then  Lemma \ref{lm:expected-gradient} reduces to $\mbb{E}[\norm{\frac{1}{N}\sum_{i=1}^N \mb{g}_i - \nabla f(\mb{x})}^2] \leq \frac{\sigma^2}{N}$, which is the fundamental reason why the classical parallel mini-batch SGD can converges $N$ times faster with $N$ workers. However, to reduce the communication overhead associated with synchronization between workers, we shall consider distributed algorithms with fewer synchronization rounds such that $\mb{x}_i$ at different workers can possibly deviate from each other.  In this case, huge deviations across $\mb{x}_i$ can cause that $\mb{g}_i$ sampled at $\mb{x}_i$ provides unreliable or even contradicting gradient knowledge and hence slow down the convergence. Thus, to let $N$ workers jointly solve problem \eqref{eq:sto-opt} with fast convergence, the distributed algorithm needs to enforce certain concentration of $\mb{g}_i$. The following useful lemma relates the concentration of $\nabla f_i(\mb{x})$ with deviations across all $\mb{x}_i$.

\begin{Lem}\label{lm:from-smooth-f}
Consider problem \eqref{eq:sto-opt} under Assumption \ref{ass:basic}.  For any $N$ points $\mb{x}_i, i\in\{1,2,\ldots,N\}$, if we define $\bar{\mb{x}} = \frac{1}{N}\sum_{i=1}^{N}\mb{x}_i$ ,then we have 
	\begin{align}
	&\frac{1}{N}\sum_{i=1}^N\norm{\nabla f_i(\mb{x}_i) - \frac{1}{N}\sum_{j=1}^N \nabla f_j(\mb{x}_j)}^2 \nonumber\\
	\leq&  \frac{6L^2}{N}\sum_{i=1}^{N}\norm{\mb{x}_i - \bar{\mb{x}}}^2 + \frac{3}{N}\sum_{i=1}^{N}\norm{\nabla f_i(\bar{\mb{x}}) - \nabla f(\bar{\mb{x}})}^2 \nonumber
	\end{align} 
\end{Lem}
\begin{proof}
	See Supplement \ref{sec:pf-lm-from-smooth-f}
\end{proof}

\subsection{Parallel Restarted SGD with Momentum}

\subsubsection{Algorithm \ref{alg:prsgd-momentum} with Polyak's Momentum}

\begin{algorithm}[tb]
	\caption{Parallel Restarted SGD with momentum (PR-SGD-Momentum)}\label{alg:prsgd-momentum}
	\begin{algorithmic}[1]
		\STATE {\bf Input:}  Initialize local solutions $\mb{x}_i^{(0)} = \hat{\mb{x}} \in \mathbb{R}^m$ and momentum buffers $\mb{u}_i^{(0)} = \mb{0}, \forall i\in\{1,2,\ldots,N\}$. Set learning rate $\gamma > 0$, momentum coefficient $\beta\in[0,1)$,  node synchronization interval $I>0$ and number of iterations $T$
		\FOR{$t=1, 2\ldots,T-1$}
		\STATE  Each worker $i$ samples its stochastic gradient $\mb{g}_i^{(t-1)} = \nabla F_i(\mb{x}_i^{(t-1)}; \xi_i^{(t-1)})$ with $\xi_i^{(t-1)}\sim \mc{D}_i$. 
		\STATE Each worker $i$ in parallel updates its local momentum buffer and solution via 
		\begin{align}
		\text{Option I:}\begin{cases}\mb{u}_{i}^{(t)} = \beta \mb{u}_{i}^{(t-1)} + \mb{g}_i^{(t-1)} \\
		\mb{x}_{i}^{(t)} =  \mathbf{x}_{i}^{(t-1)} - \gamma \mb{u}_{i}^{(t)} 
		\end{cases}\forall i \label{eq:prsgd-polyak}.
		\end{align} 
		Or
		\vskip-15pt
		\begin{align}
		\text{Option II:}\begin{cases}\mb{u}_{i}^{(t)} = \beta \mb{u}_{i}^{(t-1)} + \mb{g}_i^{(t-1)}\\
		\mb{v}_{i}^{(t)} = \beta \mb{u}_{i}^{(t)} + \mb{g}_i^{(t-1)}\\
		\mb{x}_{i}^{(t)} =  \mathbf{x}_{i}^{(t-1)} - \gamma \mb{v}_{i}^{(t)}
		\end{cases}\forall i \label{eq:prsgd-nesterov}.
		\end{align} 
		
		\IF {$t$\text{~is a multiple of~}$I$, i.e., $t~\text{mod}~I = 0$,} 
		\STATE Each worker resets its momentum buffer and local solution as the node averages via
		\begin{align}
		\begin{cases}
		\mb{u}_i^{(t)} = \hat{\mb{u}}\defeq \frac{1}{N}\sum_{j=1}^{N} \mb{u}_{j}^{(t)}\\
		\mb{x}_{i}^{(t)} = \hat{\mb{x}} \defeq \frac{1}{N}\sum_{j=1}^{N} \mb{x}_{j}^{(t)}\end{cases}\forall i \label{eq:prsgd-restart}
		\end{align}
		\ENDIF
		\ENDFOR
	\end{algorithmic}
\end{algorithm}

Consider applying the distributed algorithm described in Algorithm \ref{alg:prsgd-momentum} to solve problem \eqref{eq:sto-opt} with $N$ workers.  In the literature, there are two momentum methods for SGD, i.e.,  Polyak's momentum method (also known as the heavy ball method) and Nesterov's momentum method.  Algorithm \ref{alg:prsgd-momentum} can use either of them for local worker updates by providing two options: ``Option I" is Polyak's momentum; ``Option II" is  Nesterov's momentum.\footnote{The current description in \eqref{eq:prsgd-polyak} and \eqref{eq:prsgd-nesterov} is identical to the default implementations of momentum methods in PyTorch's SGD optimizer. Polyak's and Nesterov's momentum have many other equivalent representations. These equivalent representations will be further discussed later.}  We also note that the update steps described in \eqref{eq:prsgd-polyak} or \eqref{eq:prsgd-nesterov} can be locally performed at each worker in parallel.

The synchronization step \eqref{eq:prsgd-restart} can be interpreted as restarting momentum SGD with {\bf new} initial point $\hat{\mb{x}}$ and momentum buffer $\hat{\mb{u}}$, i.e., resetting  $\mb{u}_i^{(t)}=\hat{\mb{u}}$ and $\mb{x}_i^{(t)} = \hat{\mb{x}}$.  In this perspective, we call Algorithm \ref{alg:prsgd-momentum} ``parallel restarted SGD with momentum".  Note that if we choose $\beta = 0$ in Algorithm \ref{alg:prsgd-momentum}, then it degrades to the ``parallel restarted SGD", also known as ``local SGD" or ``SGD with periodic averaging", in \cite{Stich18ArXiv, Yu18ArXivAAAI, Lin18ArXiv, WangJoshi18ArXiv, Zhou18IJCAI, Jiang18NIPS}.  

Since inter-node communication is only needed by Algorithm \ref{alg:prsgd-momentum} to calculate global averages $\hat{\mb{x}}$ and $\hat{\mb{u}}$ in \eqref{eq:prsgd-restart} and happens only once every $I$ iterations, the total number of inter-communication rounds involved in $T$ iterations of Algorithm \ref{alg:prsgd-momentum}  is given by $T/I$. If we use {\bf all-reduce} operations to compute the global averages, the per-round communication is relatively cheap and does not involve a centralized parameter server \cite{Goyal17ArXiv}.  

Algorithm \ref{alg:prsgd-momentum} uses $\{\mb{x}_i^{(t)}\}_{t\geq 0}$ to denote local solution sequences generated at each worker $i$. For each iteration $t$, we define 
\begin{align}
\bar{\mb{x}}^{(t)} \defeq \frac{1}{N}\sum_{i=1}^N \mb{x}_{i}^{(t)} \label{eq:def-x-bar}
\end{align}
as the averages of local solutions $\mb{x}_i^{(t)}$ across all $N$ nodes. Our performance analysis will be performed regarding the aggregated version $\bar{\mb{x}}^{(t)}$ defined in \eqref{eq:def-x-bar} so that ``consensus" is no longer our concern for problem \eqref{eq:sto-opt}.  Performing convergence analysis for the node-average version $\bar{\mb{x}}^{(t)}$ has been an important technique used in previous works on distributed consensus optimization \cite{Lian17NIPS, Mania17SIOPT, Stich18ArXiv, Yu18ArXivAAAI, Jiang18NIPS}.

\subsection{Performance Analysis}
This subsection analyzes the convergence rates of Algorithm \ref{alg:prsgd-momentum} with Polyak's and Nesterov's momentum, respectively.

We first consider Algorithm \ref{alg:prsgd-momentum} with Option I given by \eqref{eq:prsgd-polyak}, i.e., Polyak's momentum.  Polyak's momentum SGD, also known as the heavy ball method, is the classical momentum SGD used for training deep neural networks and often provides fast convergence and good generalization \cite{Krizhevsky12NIPS, Sutskever13ICML,Yan18IJCAI}. If we let $\mb{x}_i^{(-1)} = \mb{x}_i^{(0)}$, then \eqref{eq:prsgd-polyak} in Algorithm \ref{alg:prsgd-momentum} can be equivalently written as the following single variable version
\begin{align}
\mb{x}_i^{(t)} = \mb{x}_i^{(t-1)} - \gamma \mb{g}_i^{(t-1)} + \beta [\mb{x}_i^{(t-1)} - \mb{x}_i^{(t-2)}] \label{eq:polyak-1-variable-version}
\end{align}
where the last term $\beta [\mb{x}_i^{(t-1)} - \mb{x}_i^{(t-2)}]$ is often called Polyak's momentum term. 

It's interesting to note that if we define an auxiliary sequence 
\begin{align}
\bar{\mb{u}}^{(t)} \defeq \frac{1}{N} \sum_{i=1}^{N} \mb{u}_{i}^{(t)} \label{eq:def-u-bar}
\end{align}
which is the node average sequence of local buffer variables $\mb{u}_i^{(t)}$ from Algorithm \ref{alg:prsgd-momentum} with Option I, then we have 
\begin{align}
\begin{cases}
\bar{\mb{u}}^{(t)} &= \beta \bar{\mb{u}}^{(t-1)} + \frac{1}{N}\sum_{i=1}^N \mb{g}_i^{(t-1)} \\
\bar{\mb{x}}^{(t)}  &= \bar{\mb{x}}^{(t-1)}  - \gamma  \bar{\mb{u}}^{(t)}
\end{cases} \label{eq:bar-polyak}, \forall t\geq 1
\end{align}
where $\bar{\mb{x}}^{(t)}$ is defined in  \eqref{eq:def-x-bar}.

An important observation is that if $\frac{1}{N}\sum_{i=1}^N \mb{g}_i^{(t-1)}$ in \eqref{eq:bar-polyak} is replaced with $\nabla f(\bar{\mb{x}}^{(t-1)})$, then \eqref{eq:bar-polyak} is exactly a standard {\bf single-node} SGD momentum method with momentum buffer $\bar{\mb{u}}^{(t)}$ and solution $\bar{\mb{x}}^{(t)}$. That is, under Algorithm \ref{alg:prsgd-momentum},  $N$ workers essentially jointly update $\bar{\mb{u}}^{(t)}$ and $\bar{\mb{x}}^{(t)}$ with momentum SGD using an {\bf inaccurate} stochastic gradient $\frac{1}{N}\sum_{i=1}^N \mb{g}_i^{(t-1)}$.  However, since \eqref{eq:prsgd-restart} periodically (every $I$ iterations) synchronizes all local variables $\mb{u}_i^{(t)}$ and $\mb{x}_i^{(t)}$, our intuition is by synchronizing frequently enough the inaccuracy of the used gradient counterpart in \eqref{eq:bar-polyak} shall not damage the convergence too much. The next theorem summarizes the convergence rate of Algorithm \ref{alg:prsgd-momentum} and characterizes the effect of synchronization interval $I$ in its convergence.

\begin{Rem}
Our Algorithm \ref{alg:prsgd-momentum} is different from a common heuristic model averaging strategy for momentum SGD suggested in \cite{WangJoshi18ArXiv2} and in Microsoft's CNTK framework \cite{CNTK}. The strategy in \cite{CNTK,WangJoshi18ArXiv2} let each worker run local momentum SGD in parallel, and periodically reset momentum buffer variables to zero and average local models.  However, there is no convergence analysis for this strategy. In contrast, this paper shall provide rigorous convergence analysis for our Algorithm \ref{alg:prsgd-momentum}. Our experiment in Supplement \ref{sec:naive-ma} furthers shows that Algorithm \ref{alg:prsgd-momentum} has faster convergence than the strategy in \cite{CNTK,WangJoshi18ArXiv2} and yields better test accuracy when used to train ResNet for CIFAR10.
\end{Rem} 
 
\begin{Thm} \label{thm:polyak-rate}
	Consider problem \eqref{eq:sto-opt} under Assumption \ref{ass:basic}. If we choose $\gamma \leq \frac{(1-\beta)^2}{(1+\beta)L}$ and $I \leq \frac{1-\beta}{6L\gamma}$ in Algorithm \ref{alg:prsgd-momentum} with {\bf Option I}, then for all $T \geq 1$, we have
	{\scriptsize
	\begin{align}
	&\frac{1}{T} \sum_{t=0}^{T-1}\mbb{E}[\norm{\nabla f(\bar{\mb{x}}^{(t)})}^2] \nonumber\\
	\leq &\frac{2(1-\beta)}{\gamma T} \big(f(\bar{\mb{x}}^{(0)}) - f^\ast \big) +  \frac{L\gamma }{(1-\beta)^2}  \frac{\sigma^2}{N}+  \frac{4 L^2\gamma^2 I \sigma^2}{(1-\beta)^2} +  \frac{9L^2\gamma^2I^2 \kappa^2}{(1-\beta)^2} \nonumber\\
	=& O(\frac{1}{\gamma T})  + O(\frac{\gamma}{N} \sigma^2) + O(\gamma^2 I \sigma^2)+ O(\gamma^2 I^2 \kappa^2) \nonumber
	\end{align}
	}%
	where $\{\bar{\mathbf{x}}^{(t)}\}_{t\geq 0}$ is the sequence defined in \eqref{eq:def-x-bar}; and $\sigma$ and $\kappa$ are constants defined in Assumption \ref{ass:basic}; and $f^\ast$ is the minimum value of problem \eqref{eq:sto-opt}.
\end{Thm}
\vspace{-0.5em}
\begin{proof}
See Supplement \ref{sec:pf-thm-polyak-rate}.
\end{proof}
\vspace{-0.5em}
The next corollary summarizes that Algorithm \ref{alg:prsgd-momentum} with Option I using $N$ workers can solve problem \eqref{eq:sto-opt} with the fast $O(\frac{1}{\sqrt{NT}})$ convergence, i.e., achieving the linear speedup (w.r.t. number of workers).  Note that $O(\frac{1}{\sqrt{NT}})$ dominates $O(\frac{N}{T})$ in Corollary \ref{cor:prsgd-momentum-I-equal-1} when $T$ is sufficiently large.

\begin{Cor}[Linear Speedup]\label{cor:prsgd-momentum-I-equal-1}
Consider problem \eqref{eq:sto-opt} under Assumption \ref{ass:basic}. If we choose $\gamma = \frac{\sqrt{N}}{\sqrt{T}}$ and $I = 1$ in Algorithm \ref{alg:prsgd-momentum}  with {\bf Option I}, then for any $T \geq \frac{36L^2N}{(1-\beta)^2}$, we have 
\begin{align}
\frac{1}{T} \sum_{t=0}^{T-1}\mbb{E}[\norm{\nabla f(\bar{\mb{x}}^{(t)})}^2]  = O(\frac{1}{\sqrt{NT}}) + O(\frac{N}{T}).
\end{align}
\end{Cor}
\begin{proof}
This corollary follows because the selection of $\gamma$ and $I$ satisfies the conditions in Theorem \ref{alg:prsgd-momentum}.   
\end{proof}

\vspace{-1em}
The next corollary summarizes that the desired linear speedup can be attained with communication skipping, i.e., using $I > 1$ in Algorithm \ref{alg:prsgd-momentum}. In particular, to achieve the linear speedup, $T$ iterations in Algorithm \ref{alg:prsgd-momentum} with Option I only require $O(N^{3/2}T^{1/2})$ communication rounds when $\kappa = 0$, i.e., workers access the common training data set; or only require $O(N^{3/4}T^{3/4})$ communication rounds when $\kappa \neq 0$, i.e., workers access non-identical training data sets.

\begin{Cor}[Linear Speedup with Reduced Communication]\label{cor:prsgd-momentum-rate}
Consider problem \eqref{eq:sto-opt} under Assumption \ref{ass:basic}. 
\vspace{-1em}
\begin{itemize}[leftmargin=10pt]
\item {\bf Case $\kappa = 0$}:  If we choose $\gamma = \frac{\sqrt{N}}{\sqrt{T}}$ and $I \leq \frac{1-\beta}{6L}\frac{\sqrt{T}}{N^{3/2}}  = O(\frac{T^{1/2}}{N^{3/2}})$ in Algorithm \ref{alg:prsgd-momentum}  with {\bf Option I}, then for any $T \geq \frac{(1+\beta)^2L^2}{(1-\beta)^4} N$, we have 
\begin{align}
\frac{1}{T} \sum_{t=0}^{T-1}\mbb{E}[\norm{\nabla f(\bar{\mb{x}}^{(t)})}^2]  = O(\frac{1}{\sqrt{NT}}).
\end{align}
By using $I = O(\frac{T^{1/2}}{N^{3/2}})$, the total number of inter-node communication rounds involved in Algorithm \ref{alg:prsgd-momentum} is $O(N^{3/2}T^{1/2})$. 
\item {\bf Case $\kappa \neq 0$}:  If we choose $\gamma = \frac{\sqrt{N}}{\sqrt{T}}$ and $I \leq \frac{1-\beta}{6L} \frac{T^{1/4}}{N^{3/4}} = O(\frac{T^{1/4}}{N^{3/4}})$ in Algorithm \ref{alg:prsgd-momentum} with {\bf Option I}, then for any $T \geq \frac{(1+\beta)^2L^2}{(1-\beta)^4} N$, we have
\begin{align}
\frac{1}{T} \sum_{t=0}^{T-1}\mbb{E}[\norm{\nabla f(\bar{\mb{x}}^{(t)})}^2]  = O(\frac{1}{\sqrt{NT}}).
\end{align}
By using $I = O(\frac{T^{1/4}}{N^{3/4}})$, the total number of inter-node communication rounds involved in Algorithm \ref{alg:prsgd-momentum} is $O(N^{3/4}T^{3/4})$. 
\end{itemize}
\end{Cor}
\vspace{-1em}
\begin{proof}
This corollary follows  simply because the selection of $\gamma$ and $I$ satisfies the conditions in Theorem \ref{alg:prsgd-momentum}.  
\end{proof}

\begin{Rem}
Following the convention in non-convex optimization,  the convergence rate in Corollary \ref{cor:prsgd-momentum-rate} is measured by the expected squared gradient norm used in \cite{GhadimiLan13SIOPT, Lian17NIPS, Yu18ArXivAAAI, Jiang18NIPS}.  To attain the average $\frac{1}{T} \sum_{t=0}^{T-1}\mbb{E}[\norm{\nabla f(\bar{\mb{x}}^{(t)})}^2]$ in expectation, one neat strategy suggested in \cite{GhadimiLan13SIOPT} is to generate a random iteration count $R$ uniformly from $\{0,1,\ldots,T-1\}$, terminate Algorithm \ref{alg:prsgd-momentum} at iteration $R$ and output $\bar{\mb{x}}^{(R)}$ as the solution.
\end{Rem}
\begin{Rem}
Note that Theorem \ref{thm:polyak-rate} and Corollary \ref{cor:prsgd-momentum-rate} hold for any $0\leq \beta < 1$. Recall that Algorithm \ref{alg:prsgd-momentum} with $\beta=0$ degrades to parallel restarted SGD (without momentum). For parallel SGD (without momentum), the communication complexity implied by Corollary \ref{cor:prsgd-momentum-rate} improves the state-of-the-art as summarized in Table \ref{table:sgd-communication}.
\end{Rem}

\subsubsection{Algorithm \ref{alg:prsgd-momentum} with Nesterov's Momentum}

Now consider using Nesterov's momentum described in Option II in Algorithm \ref{alg:prsgd-momentum}.  Option II given by \eqref{eq:prsgd-nesterov} introduces extra auxiliary variables $\bar{\mb{v}}_i^{(t)}$ and uses them in the update of local solutions $\mb{x}_i^{(t)}$. It is not difficult\footnote{The derivation on the equivalence is in Supplement \ref{sec:two-equivalent-nesterov}. } to show that \eqref{eq:prsgd-nesterov} yields the same solution sequences $\{\mb{x}_{i}^{(t)}\}_{t\geq0}$ as 
\begin{align}
\begin{cases}
\mb{y}_{i}^{(t)} = \mb{x}_{i}^{(t-1)} - \gamma \mb{g}_i^{(t-1)}\\
\mb{x}_{i}^{(t)} =  \mathbf{y}_{i}^{(t)} +\beta [\mb{y}_{i}^{(t)} - \mb{y}_{i}^{(t-1)}]
\end{cases}\forall i \label{eq:nesterov-common-version}
\end{align}
with $\mathbf{y}_{i}^{(0)} \defeq \mb{0}, \forall i$. The only difference between \eqref{eq:prsgd-nesterov} and \eqref{eq:nesterov-common-version} is the adoption of different cache variables.

Equation \eqref{eq:nesterov-common-version} is more widely used to describe SGD with Nesterov's momentum in the literature \cite{book_ConvexOpt_Nesterov}.  However, by writing Nesterov's momentum SGD as \eqref{eq:prsgd-nesterov}, an important observation is that momentum buffer variables $\mb{u}_i^{(t)}$ in Polyak's and Nesterov's momentum methods evolve according to the same dynamic (with stochastic gradients sampled at different points).  By adapting the convergence rate analysis for Polyak's momentum, Theorem \ref{thm:nesterov-rate} (formally proven in Supplement \ref{sec:pf-thm-nesterov-rate}) summarizes the convergence for Nesterov's momentum.

\begin{Thm} \label{thm:nesterov-rate}
	Consider problem \eqref{eq:sto-opt} under Assumption \ref{ass:basic}. If we choose $\gamma \leq \frac{(1-\beta)^2}{L(1+\beta^3)}$ and $I \leq \frac{1-\beta}{6L\gamma}$ in Algorithm \ref{alg:prsgd-momentum} with {\bf Option II}, then for all $T > 0$, we have 
	{\scriptsize
	\begin{align}
	&\frac{1}{T} \sum_{t=0}^{T-1}\mbb{E}[\norm{\nabla f(\bar{\mb{x}}^{(t)})}^2] \nonumber\\
	\leq &\frac{2(1-\beta)}{\gamma T} \big(f(\bar{\mb{x}}^{(0)}) - f^\ast \big) +  \frac{L\gamma }{(1-\beta)^2}  \frac{\sigma^2}{N}+  \frac{4 L^2\gamma^2 I \sigma^2}{(1-\beta)^2} +  \frac{9L^2\gamma^2I^2 \kappa^2}{(1-\beta)^2} \nonumber\\
	=& O(\frac{1}{\gamma T})  + O(\frac{\gamma}{N} \sigma^2) + O(\gamma^2 I \sigma^2)+ O(\gamma^2 I^2 \kappa^2) \nonumber
	\end{align}
	}%
	where $\{\bar{\mathbf{x}}^{(t)}\}_{t\geq 0}$ is the sequence defined in \eqref{eq:def-x-bar}; and $\sigma$ and $\kappa$ are constants defined in Assumption \ref{ass:basic}; and $f^\ast$ is the minimum value of problem \eqref{eq:sto-opt}.
\end{Thm}

\begin{Rem}
By comparing Theorem \ref{thm:polyak-rate} and Theorem \ref{thm:nesterov-rate}, we conclude that the order of convergence rate for both options in Algorithm \ref{alg:prsgd-momentum} (with slightly different choices of learning rate $\gamma$) have the same dependence on $\gamma, N$ and $I$.  It is then immediate that Algorithm \ref{alg:prsgd-momentum} with Option II can achieve the linear speedup with the same (reduced) communication complexity as summarized in Corollary \ref{cor:prsgd-momentum-rate} for Option I. 
\end{Rem}

\vspace{-1em}
\section{Extension: Momentum SGD with Decentralized Communication}

Note that Algorithm \ref{alg:prsgd-momentum} requires to compute global averages of all $N$ nodes at each communication step \eqref{eq:prsgd-restart}. Such global averages are possible without a central parameter center by using distributive average protocols such as all-reduce primitives \cite{Goyal17ArXiv}. However, the feasibility and performance of exiting all-reduce protocols are restricted to the underlying network topology. For example, the performance of a ring based all-reduce protocol can be poor in a line network, which does not have a physical ring.  In this subsection, we consider distributed momentum SGD with decentralized communication (described in Algorithm \ref{alg:prsgd-decentralized}) where the communication pattern can be fully customized.  Let $\mb{W}\in \mbb{R}^{N\times N}$ be a symmetric doubly stochastic matrix. By definition, a symmetric doubly stochastic matrix satisfies  (1) $W_{ij}\in[0,1], \forall i, j\in \{1,2,\ldots, N\}$; (2) $\mb{W}\tran = \mb{W}$; (3) $\sum_{j=1}^N W_{ij}= 1, \forall i\in \{1,2,\ldots, N\}$. For a computer network with $N$ nodes, we can use $\mb{W}$ to encode the communication between nodes, i.e., if node $i$ and node $j$ are disconnected, we let $W_{ij} = 0$. This section further assumes the mixing matrix $\mb{W}$ that obeys the network connection topology is selected to satisfy the following assumption. The same assumption is used in \cite{Lian17NIPS,Jiang17NIPS,WangJoshi18ArXiv}.

\begin{Ass}\label{ass:mix-matrix}
The mixing matrix $\mb{W}\in \mbb{R}^{N\times N}$ is a symmetric doubly stochastic matrix satisfying $\lambda_{1}(\mb{W}) = 1$ and 
\begin{align}
\max\{\abs{\lambda_2(\mb{W})}, \abs{\lambda_{N}(\mb{W})}\} \leq \sqrt{\rho} <1
\end{align}
where $\lambda_i(\mb{W})$ is the $i$-th largest eigenvalue of $\mb{W}$.
\end{Ass}

\begin{algorithm}[tb]
	\caption{Distributed Momentum SGD with Decentralized Communication}\label{alg:prsgd-decentralized}
	\begin{algorithmic}[1]
		\STATE {\bf Input:}  Initialize local solutions $\mb{x}_i^{(0)} =\hat{\mb{x}} \in \mathbb{R}^m$ and momentum buffers $ \mb{u}_i^{(0)} = \mb{0}, \forall i\in\{1,2,\ldots,N\}$. Set mixing matrix $\mb{W}$, learning rate $\gamma > 0$, momentum coefficient $\beta\in[0,1)$ and number of iterations $T$
		\FOR{$t=1, 2\ldots,T-1$}
		\STATE  Each worker $i$ samples its stochastic gradient $\mb{g}_i^{(t-1)} = \nabla F_i(\mb{x}_i^{(t-1)}; \xi_i^{(t-1)})$ with $\xi_i^{(t-1)}\sim \mc{D}_i$. 
		\STATE Each worker $i$ in parallel updates its local auxiliary momentum buffer $\tilde{\mb{u}}_i^{(t)}$ and solution  $\tilde{\mb{x}}_i^{(t)}$ via 		 
		\begin{align}
		\text{Option I:}\begin{cases}\tilde{\mb{u}}_{i}^{(t)} = \beta \mb{u}_{i}^{(t-1)} + \mb{g}_i^{(t-1)} \\
		\tilde{\mb{x}}_{i}^{(t)} =  \mathbf{x}_{i}^{(t-1)} - \gamma \tilde{\mb{u}}_{i}^{(t)} 
		\end{cases}\forall i \label{eq:decentralized-polyak}.
		\end{align} 
		Or
		\vskip-15pt
		\begin{align}
		\text{Option II:}\begin{cases}\tilde{\mb{u}}_{i}^{(t)} = \beta \mb{u}_{i}^{(t-1)} + \mb{g}_i^{(t-1)}\\
		\tilde{\mb{v}}_{i}^{(t)} = \beta \tilde{\mb{u}}_{i}^{(t)} + \mb{g}_i^{(t-1)}\\
		\tilde{\mb{x}}_{i}^{(t)} =  \mathbf{x}_{i}^{(t-1)} - \gamma \tilde{\mb{v}}_{i}^{(t)}
		\end{cases}\forall i \label{eq:decentralized-nesterov}.
		\end{align}
		
		\STATE Each worker $i$ updates its local momentum buffer $\mb{u}_i^{(t)}$ and solution  $\mb{x}_i^{(t)}$ based on the neighborhood weighted averages given by  
		\begin{align}
		\begin{cases}
		\mb{u}_i^{(t)} =  \sum_{j=1}^N \tilde{\mb{u}}_j^{(t)}  W_{ji} \\
		\mb{x}_{i}^{(t)} =  \sum_{j=1}^N \tilde{\mb{x}}_j^{(t)}  W_{ji} \end{cases}\forall i \label{eq:prsgd-avg-decentralized}
		\end{align}
		\ENDFOR
	\end{algorithmic}
\end{algorithm}

Compared with PR-SGD-Momentum described in Algorithm \ref{alg:prsgd-momentum}, Algorithm \ref{alg:prsgd-decentralized}  uses a {\bf local} aggregation step \eqref{eq:prsgd-avg-decentralized} at every iteration.  While Algorithm \ref{alg:prsgd-decentralized} can not skip the aggregation steps as Algorithm \ref{alg:prsgd-momentum} does and hence involves $I$ times more communication rounds, the per-step communication is more flexible since each node only computes neighborhood averages rather than global averages. As a consequence, Algorithm \ref{alg:prsgd-decentralized} is more suitable to distributed machine learning with mobiles where the network topology is heterogeneous and diverse. By extending our analysis for Algorithm \ref{alg:prsgd-momentum}, the next theorem proves the linear speedup of Algorithm \ref{alg:prsgd-decentralized}. 

\begin{Thm} \label{thm:decentralized-rate}
	Consider problem \eqref{eq:sto-opt} under Assumptions \ref{ass:basic} and \ref{ass:mix-matrix}. If we choose $\gamma \leq \min\{\frac{(1-\beta)^2(1-\sqrt{\rho})^2}{6L},  \frac{(1-\beta)(1-\sqrt{\rho})}{4L}\}$ in Algorithm \ref{alg:prsgd-decentralized}, then for all $T \geq 1$, we have
	{\scriptsize
	\begin{align}
	&\frac{1}{T} \sum_{t=0}^{T-1}\mbb{E}[\norm{\nabla f(\bar{\mb{x}}^{(t)})}^2] = O(\frac{1}{\gamma T})  + O(\frac{\gamma}{N} \sigma^2) + O(\gamma^2 \sigma^2)+ O(\gamma^2 \kappa^2) \nonumber
	\end{align}
	}%
	where $\{\bar{\mathbf{x}}^{(t)}\}_{t\geq 0}$ defined in \eqref{eq:def-x-bar} are the averages of local solutions across $N$ workers; and $\sigma, \kappa$ and $\rho$ are constants defined in Assumptions \ref{ass:basic}-\ref{ass:mix-matrix}. 
\end{Thm}
\vspace{-1em}
\begin{proof}
See Supplement \ref{sec:pf-thm-decentralized-rate}.
\end{proof}
\vspace{-0.8em}
This next corollary further summarizes that using $\gamma = O(\frac{\sqrt{N}}{\sqrt{T}})$ in Algorithm \ref{alg:prsgd-decentralized} can achieve $O(\frac{1}{\sqrt{NT}})$ convergence with the linear speedup. 
\begin{Cor}
Consider problem \eqref{eq:sto-opt} under Assumptions \ref{ass:basic} and \ref{ass:mix-matrix}. If we choose $\gamma =\frac{\sqrt{N}}{\sqrt{T}}$ in Algorithm \ref{alg:prsgd-decentralized}, then for all $T\geq \max\{\frac{36NL^2}{(1-\beta)^4(1-\sqrt{\rho})^4}, \frac{32NL^2}{(1-\beta)^2(1-\sqrt{\rho})^2}\}$, we have
	\begin{align}
	\frac{1}{T} \sum_{t=0}^{T-1}\mbb{E}[\norm{\nabla f(\bar{\mb{x}}^{(t)})}^2] = O(\frac{1}{\sqrt{NT}})  + O(\frac{N}{T}) \nonumber
	\end{align}
where $\{\bar{\mathbf{x}}^{(t)}\}_{t\geq 0}$ defined in \eqref{eq:def-x-bar} are the averages of local solutions across $N$ workers.
\end{Cor}

Note that Algorithm \ref{alg:prsgd-decentralized} with $\beta=0$ degrades to the decentralized SGD without momentum considered in \cite{Lian17NIPS}, where the linear speedup of decentralized SGD is first shown.  Recently, many variants of decentralized SGD without momentum with linear speedup have been developed for distributed deep learning \cite{Tang18ArXiv,WangJoshi18ArXiv, Assran18ArXiv}. To our knowledge, this is the first time that decentralized momentum SGD is shown to possess the same linear speedup.

\vspace{-0.8em}
\section{Experiments}
\begin{figure*}[t!]
\begin{subfigure}[t]{0.5\textwidth}
\includegraphics[height=2.35in]{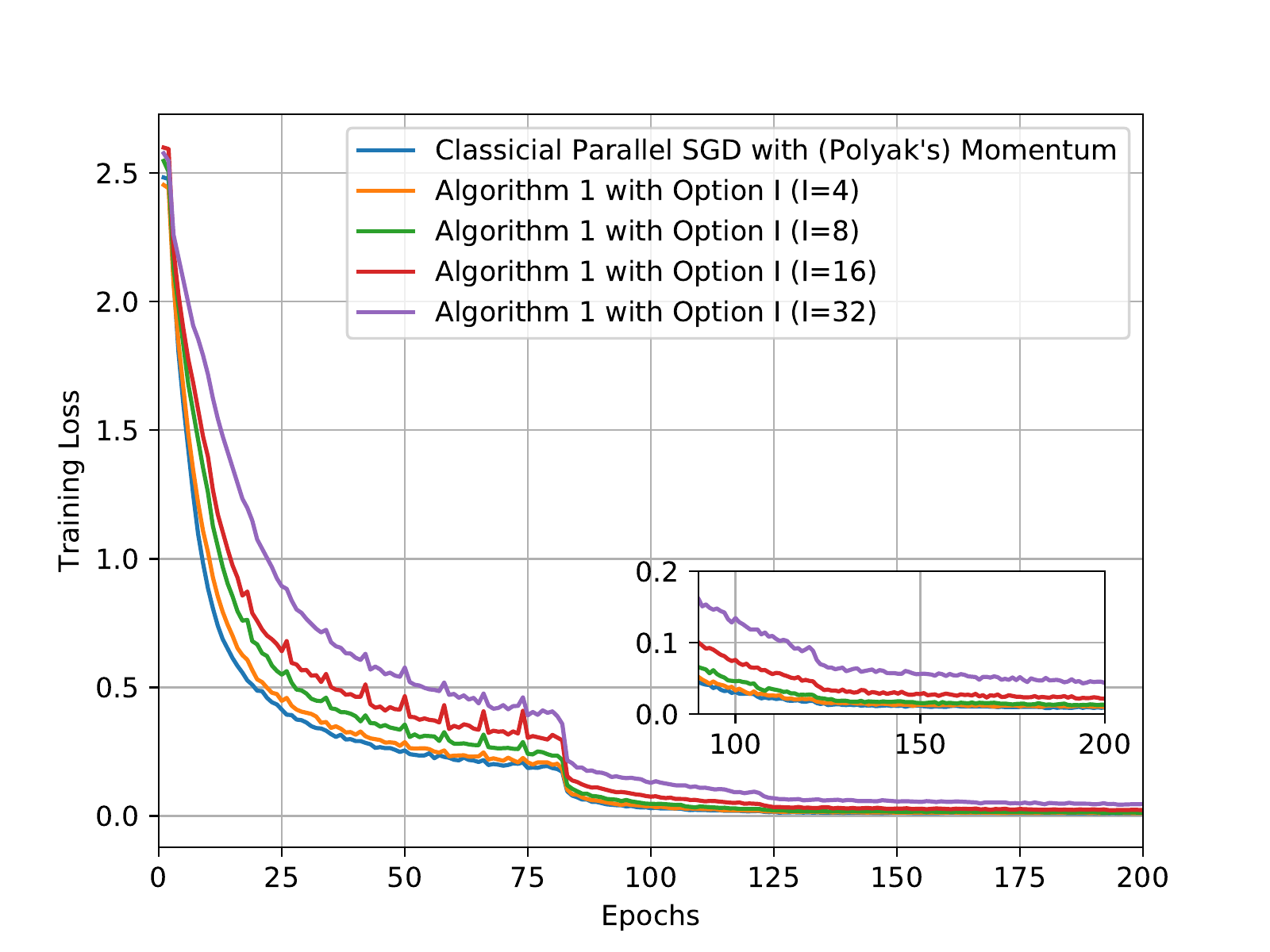}
\caption{Training loss v.s. epochs. }
 \end{subfigure}
 ~
 \begin{subfigure}[t]{0.5\textwidth}
\includegraphics[height=2.35in]{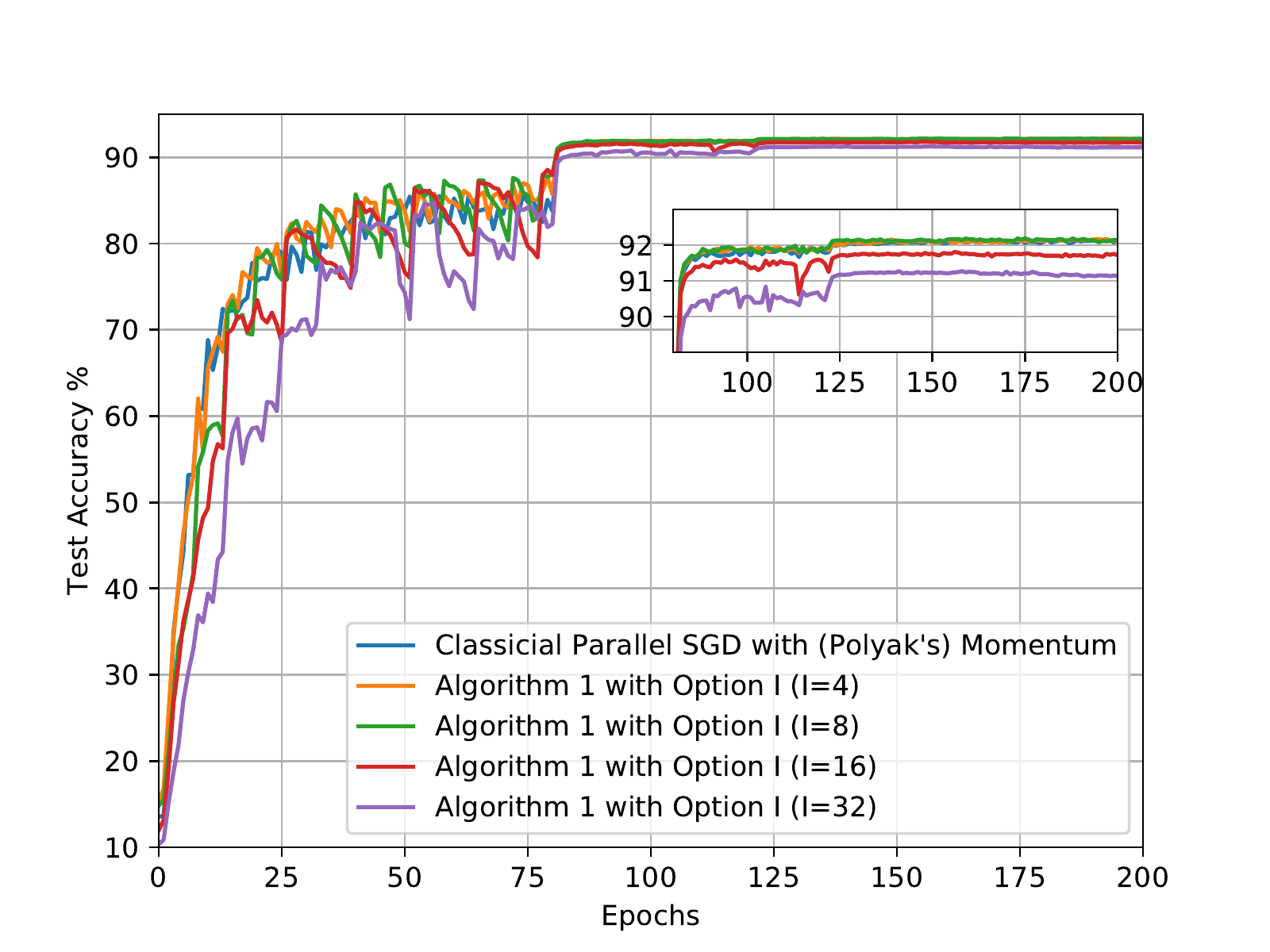}
\caption{Test accuracy v.s. epochs. }
 \end{subfigure}
  \caption{Algorithm \ref{alg:prsgd-momentum} with Option I: convergence v.s. epochs for ResNet56 over CIFAR10. }
   \label{fig:polyak_epoch}
\end{figure*} 
\vspace{-0.8em}
\begin{figure*}[t!]
\begin{subfigure}[t]{0.5\textwidth}
\includegraphics[height=2.4in]{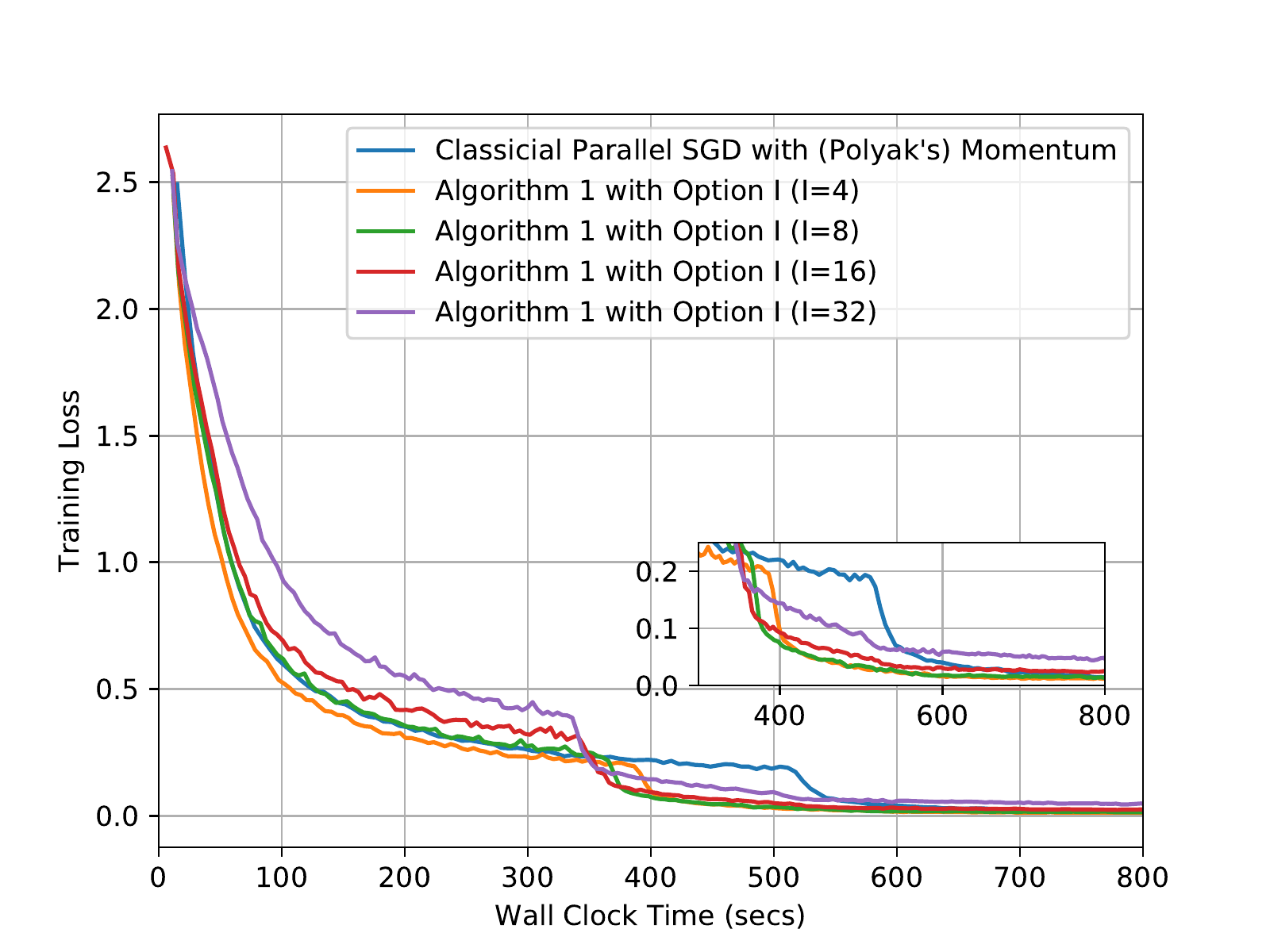}
\caption{Training loss v.s. wall clock time.}
 \end{subfigure}
 ~
 \begin{subfigure}[t]{0.5\textwidth}
\includegraphics[height=2.4in]{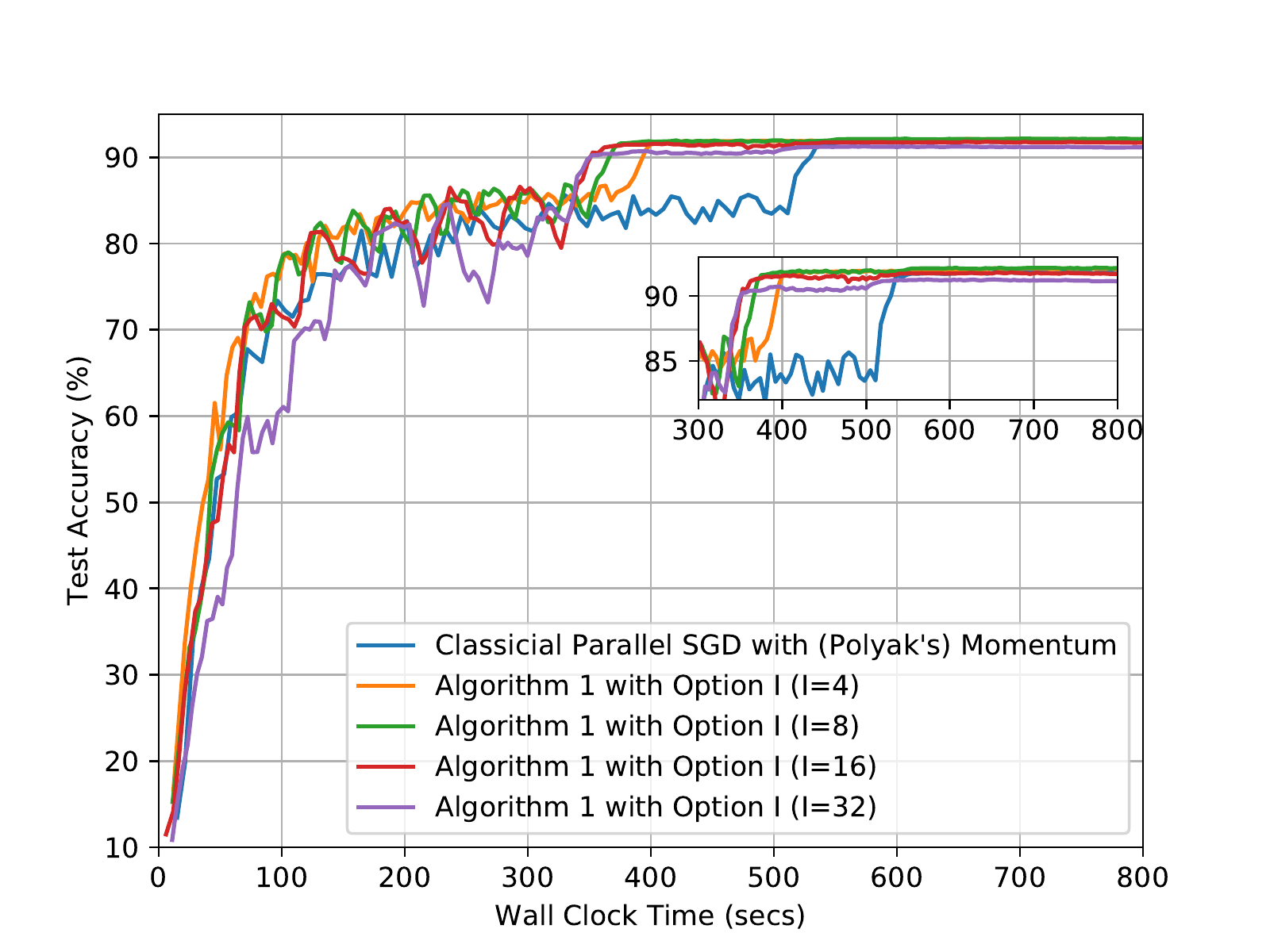}
\caption{Test accuracy v.s. wall clock time. }
 \end{subfigure}
  \caption{Algorithm \ref{alg:prsgd-momentum} with Option I: convergence v.s. wall clock time for ResNet56 over CIFAR10. }
   \label{fig:polyak_wall}
\end{figure*}

In this section, we validate our theory with experiments on training ResNet \cite{He16CVPR} for the image classification tasks over CIFAR-10 \cite{Krizhevsky09} and ImageNet \cite{ImageNet}. Our experiments are conducted on a machine with $8$ NVIDIA P100 GPUs. The local batch size at each GPU is $64$.  The learning rate is initialized to $0.3$ and is divided by $10$ when all GPUs jointly access $80$ and $120$ epochs.\footnote{Such decaying learning rates are used to achieve good test accuracy by practitioners \cite{He16CVPR}. This deviates from the constant learning rates used in the theory to establish the linear speedup. On one hand, it is possible follow the analysis techniques in \cite{Yu18ArXivAAAI} to extend this paper's theory to establish a similar linear speedup. On the other hand, our supplement \ref{sec:exp-clr} reports extra experiments to verify the linear speedup of Algorithm \ref{alg:prsgd-momentum} using constant learning rates faithful to the theory.} The momentum coefficient , i.e., $\beta$ in Algorithm \ref{alg:prsgd-momentum}, is set to $0.9$ for both Polyak's and Nesterov's momentum. All algorithms are implemented using PyTorch $0.4$.  To study how communication skipping affects the convergence of Algorithm \ref{alg:prsgd-momentum}, we run Algorithm \ref{alg:prsgd-momentum} with $I\in\{4,8,16,32\}$ and compare the test accuracy convergence between Algorithm \ref{alg:prsgd-momentum} and the classical parallel mini-batch SGD with momentum, which uses an inter-node communication step at every iteration and can be interpreted as Algorithm \ref{alg:prsgd-momentum} with $I=1$.  Figure \ref{fig:polyak_epoch} plots the convergence of training loss and test accuracy in terms of the number of epochs that are jointly accessed by all used GPUs. That is, if the $x$-axis value is $8$, then each GPU access $1$ epoch of training data. The same convention is followed by other figures for multiple GPU training in this paper. By plotting the convergence in terms of the number of epochs that are jointly accessed by all GPUs, we can verify the $O(\frac{1}{\sqrt{NT}})$ convergence for Algorithm \ref{alg:prsgd-momentum} proven in this paper. To verify the benefit of skipping communication in our Algorithm \ref{alg:prsgd-momentum}, Figure \ref{fig:polyak_wall}  plots the convergence of training loss and test accuracy in terms of the wall clock time. Since Algorithm \ref{alg:prsgd-momentum} with $I>0$ skip communication steps, it is much faster than the classical parallel momentum SGD when measured by the wall clock time.  More numerical experiments on training ResNet over ImageNet, comparisons with a model averaging strategy suggested in \cite{CNTK,WangJoshi18ArXiv2}, and Algorithm \ref{alg:prsgd-momentum} with Nesterov's momentum are available in Supplement \ref{sec:more-exp}.

\vspace{-1em}
\section{Conclusion}
This paper considers parallel restarted SGD with momentum and prove that it can achieve $O(1/\sqrt{NT})$ convergence with $O(N^{3/2}T^{1/2})$ or $O(N^{3/4}T^{3/4})$ communication rounds depending whether each node accesses identical objective functions $f_i(\mb{x})$ or not.  We further show that distributed momentum SGD with decentralized communication can achieve $O(1/\sqrt{NT})$ convergence.

\bibliographystyle{icml2019}
\bibliography{./mybibfile}
 
\onecolumn

\section{Supplement}

\subsection{Proof of Lemma \ref{lm:expected-gradient}} \label{sec:pf-lm-expected-gradient}
The  lemma follows simply  because unbiased stochastic gradients $\mb{g}_i$ are independently sampled. The formal proof is as follows: 
\begin{align}
& \mbb{E}[\norm{\frac{1}{N}\sum_{i=1}^N \mb{g}_i}^2] \nonumber\\ 
\overset{(a)}{=}& \mathbb{E}[ \Vert  \frac{1}{N}\sum_{i=1}^{N} ( \mb{g}_i - \nabla f_{i}(\mathbf{x}_{i})) \Vert^{2}]  +\mathbb{E} [ \Vert \frac{1}{N} \sum_{i=1}^{N} \nabla f_{i}(\mathbf{x}_{i})\Vert^{2}] \nonumber\\
\overset{(b)}{=}& \frac{1}{N^{2}}\sum_{i=1}^{N} \mathbb{E}[\Vert  \mb{g}_i - \nabla f_{i}(\mathbf{x}_{i})\Vert^{2}] +  \mathbb{E} [ \Vert \frac{1}{N} \sum_{i=1}^{N} \nabla f_{i}(\mathbf{x}_{i})\Vert^{2}] \nonumber\\
\overset{(c)}{\leq }& \frac{1}{N}\sigma^{2}  +  \mathbb{E} [ \Vert \frac{1}{N} \sum_{i=1}^{N} \nabla f_{i}(\mathbf{x}_{i})\Vert^{2}] \nonumber
\end{align}
where (a) follows from the facts that $\mathbb{E}[\mb{g}_i] = \nabla f_{i}(\mathbf{x}_{i}), \forall i$ and identity $\mathbb{E}[\Vert \mathbf{Z} \Vert^{2}] = \mathbb{E} [ \Vert \mathbf{\mathbf{Z}} - \mathbb{E}[\mathbf{Z}]\Vert^{2}] + \Vert\mathbb{E}[\mathbf{Z}] \Vert^{2}$ holds for any random vector $\mathbf{Z}$; (b) holds from the facts that $\mb{g}_i - \nabla f_{i}(\mathbf{x}_{i})$ are independent random vectors with $\mathbf{0}$ means and identity $\mbb{E}[\Vert \sum_{i=1}^N \mb{Z}_i \Vert^2] = \sum_{i=1}^N \mbb{E}[\norm{\mb{Z}_i}^2]$ holds when $\mb{Z}_i$ are independent with $\mb{0}$ means; and (c) follows from Assumption \ref{ass:basic}.

\subsection{Proof of Lemma \ref{lm:from-smooth-f}} \label{sec:pf-lm-from-smooth-f}

This lemma follows from simple algebraic manipulations as follows:
	\begin{align}
	&\frac{1}{N}\sum_{i=1}^N \Vert \nabla f_i(\mb{x}_i)- \frac{1}{N}\sum_{j=1}^{N} \nabla f_j(\mb{x}_j)  \Vert^2\nonumber\\
	=&\frac{1}{N}\sum_{i=1}^N  \big\Vert \nabla f_i(\mb{x}_i) - \nabla f_i(\bar{\mb{x}}) + \nabla f_i(\bar{\mb{x}}) - \nabla f(\bar{\mb{x}}) + \nabla f(\bar{\mb{x}}) - \frac{1}{N}\sum_{j=1}^N \nabla f_j (\mb{x}_j)\big\Vert^2 \nonumber \\
	\overset{(a)}{\leq}& \frac{1}{N}\sum_{i=1}^N \Big[3\Vert \nabla f_i(\mb{x}_i) - \nabla f_i(\bar{\mb{x}})\Vert^2 +3\Vert \nabla f_i(\bar{\mb{x}}) - \nabla f(\bar{\mb{x}})\Vert^2 + 3\Vert\nabla f(\bar{\mb{x}}) - \frac{1}{N}\sum_{j=1}^N \nabla f_j (\mb{x}_j)\Vert^2 \Big] \nonumber\\
	\overset{(b)}{\leq}& \frac{1}{N}\sum_{i=1}^N \Big[3L^2\Vert \mb{x}_i - \bar{\mb{x}}\Vert^2 +3\Vert \nabla f_i(\bar{\mb{x}}) - \nabla f(\bar{\mb{x}})\Vert^2 + 3 \frac{1}{N}\sum_{j=1}^N L^2 \Vert \bar{\mb{x}} - \mb{x}_j\Vert^2 \Big] \nonumber\\
	=& \frac{6L^2}{N}\sum_{i=1}^N \Vert \mb{x}_i - \bar{\mb{x}}\Vert^2 +\frac{3}{N}\sum_{i=1}^N\Vert \nabla f_i(\bar{\mb{x}}) - \nabla f(\bar{\mb{x}})\Vert^2  \nonumber
	\end{align}
where (a) follows from the basic inequality $\norm{\mb{a}_1 + \mb{a}_2 + \mb{a}_3}^2 \leq 3 \norm{\mb{a}_1}^2 + 3\norm{\mb{a}_2}^2 + 3\norm{\mb{a}_3}^2$ for any vectors $\mb{a}_1,\mb{a}_2$ and $\mb{a}_3$;  (b) follows because $\Vert \nabla f_i(\mb{x}_i) - \nabla f_i(\bar{\mb{x}})\Vert \leq L \Vert \mb{x}_i -\bar{\mb{x}} \Vert$ by the smoothness of $f_i(\cdot)$ in Assumption \ref{ass:basic} and $\Vert\nabla f(\bar{\mb{x}}) - \frac{1}{N}\sum_{j=1}^N \nabla f_j (\mb{x}_j)\Vert^2 = \Vert \frac{1}{N} \sum_{j=1}^N \nabla f_j(\bar{\mb{x}}) - \frac{1}{N}\sum_{j=1}^N \nabla f_j (\mb{x}_j)\Vert^2 \leq \frac{1}{N} \sum_{j=1}^{N} \Vert \nabla f_j(\bar{\mb{x}})  - \nabla f_j (\mb{x}_j)\Vert^2 \leq \frac{1}{N} \sum_{j=1}^N L^2 \Vert \bar{\mb{x}} - \mb{x}_j\Vert^2$, where the first inequality follows from the convexity of $\Vert \cdot \Vert^2$ and Jensen's inequality and the second inequality follows from the smoothness of each $f_j(\cdot)$.

\subsection{Proof of Theorem \ref{thm:polyak-rate}} \label{sec:pf-thm-polyak-rate}

To prove this theorem, we first introduce an auxiliary sequence $\{\bar{\mb{z}}^{(t)}\}_{t\geq 0}$ given by  
\begin{align}
\bar{\mb{z}}^{(t)} \defeq \left \{  \begin{array}{ll} \bar{\mb{x}}^{(t)}, & \quad t=0 \\
\frac{1}{1-\beta} \bar{\mb{x}}^{(t)} - \frac{\beta}{1-\beta}\bar{\mb{x}}^{(t-1)}, &\quad t\geq 1 \end{array} \right. \label{eq:def-z}
\end{align}
where $\bar{\mb{x}}^{(t)}$ defined in \eqref{eq:def-x-bar} are the node averages of local solutions from Algorithm \ref{alg:prsgd-momentum} with Option I.  A similar auxiliary sequence $\{\bar{\mb{z}}^{(t)}\}_{t\geq 0}$ has been used for the convergence analysis of standard momentum methods (single node without restarting)  in \cite{Ghadimi14ArXiv} and \cite{Yan18IJCAI}.

\subsubsection{Useful Lemmas}
Before the main proof of Theorem \ref{thm:polyak-rate}, we further introduce several useful lemmas.

\begin{Lem}\label{lm:z-diff}
Consider the sequence $\{\bar{\mb{z}}^{(t)}\}_{t\geq 0}$ defined in \eqref{eq:def-z}. Algorithm \ref{alg:prsgd-momentum} with Option I ensures that for all $t\geq 0$, we have 
\begin{align}
\bar{\mb{z}}^{(t+1)} - \bar{\mb{z}}^{(t)} = -\frac{\gamma}{1-\beta}\frac{1}{N}\sum_{i=1}^N \mb{g}_i^{(t)}.
\end{align}
\end{Lem}
\begin{proof} To prove $\bar{\mb{z}}^{(t+1)} - \bar{\mb{z}}^{(t)} = -\frac{\gamma}{1-\beta}\frac{1}{N}\sum_{i=1}^N \mb{g}_i^{(t)}$, we consider $t= 0$ and $t\geq 1$ separately.
		\begin{itemize}
			\item {\bf Case $t= 0$}: We have
			\begin{align*}
			\bar{\mb{z}}^{(t+1)} - \bar{\mb{z}}^{(t)} &= \bar{\mb{z}}^{(1)} - \bar{\mb{z}}^{(0)}\\
			&\overset{(a)}{=} \frac{1}{1-\beta} \bar{\mb{x}}^{(1)}-\frac{\beta}{1-\beta}\bar{\mb{x}}^{(0)} - \bar{\mb{x}}^{(0)} \\
			&=\frac{1}{1-\beta} [\bar{\mb{x}}^{(1)} - \bar{\mb{x}}^{(0)}]\\
			&\overset{(b)}{=} -\frac{\gamma}{1-\beta} \frac{1}{N} \sum_{i=1}^{N} \mb{g}_i^{(0)}
			\end{align*}
			where (a) follows from \eqref{eq:def-z}; (b) follows from \eqref{eq:bar-polyak} by noting that $\bar{\mb{u}}^{(0)} =  \mb{0}$.
			\item  {\bf Case $t\geq 1$}: We have
			\begin{align*}
			\bar{\mb{z}}^{(t+1)} - \bar{\mb{z}}^{(t)} \overset{(a)}{=}& \frac{1}{1-\beta} [ \bar{\mb{x}}^{(t+1)} -\bar{\mb{x}}^{(t)}] - \frac{\beta}{1-\beta}[\bar{\mb{x}}^{(t)} - \bar{\mb{x}}^{(t-1)}] \\
			\overset{(b)}{=}&- \frac{\gamma}{1-\beta} \bar{\mb{u}}^{(t+1)} + \frac{\gamma \beta}{1-\beta} \bar{\mb{u}}^{(t)} \\
			\overset{(c)}{=}& -\frac{\gamma}{1-\beta}[ \beta \bar{\mb{u}}^{(t)} + \frac{1}{N}\sum_{i=1}^N \mb{g}_i^{(t)}] + \frac{\gamma \beta}{1-\beta} \bar{\mb{u}}^{(t)}\\
			=& -\frac{\gamma}{1-\beta}\frac{1}{N}\sum_{i=1}^N \mb{g}_i^{(t)}
			\end{align*}
			where (a) follows from \eqref{eq:def-z}; and (b) and (c) follow from \eqref{eq:bar-polyak}.
		\end{itemize} 
\end{proof}

\begin{Lem} \label{lm:z-x-diff}
	Let $\{\bar{\mb{x}}^{(t)}\}_{t\geq 0}$ defined in \eqref{eq:def-x-bar} be the node averages of local solutions from Algorithm \ref{alg:prsgd-momentum} with Option I.  Let $\{\bar{\mb{z}}^{(t)}\}_{t\geq 0}$ be defined in \eqref{eq:def-z}. For all $T\geq 1$, Algorithm \ref{alg:prsgd-momentum} with Option I ensures that 
	\begin{align}
		\sum_{t=0}^{T-1} \norm{ \bar{\mb{z}}^{(t)} - \bar{\mb{x}}^{(t)} }^2 \leq \frac{\gamma^2 \beta^2}{(1-\beta)^4}  \sum_{t=0}^{T-1} \Big\Vert \Big[\frac{1}{N} \sum_{i=1}^N \mb{g}_i^{(t)} \Big] \Big\Vert^2
	\end{align}
\end{Lem}
\begin{proof}
	Recall that $\bar{\mb{u}}^{(0)} = \mb{0}$. Recursively applying the first equation in \eqref{eq:bar-polyak} for $t$ times yields:
	\begin{align}
		\bar{\mb{u}}^{(t)} =  \sum_{\tau=0}^{t-1} \beta^{t-1-\tau} \Big[\frac{1}{N} \sum_{i=1}^N \mb{g}_i^{(\tau)} \Big], \quad \forall t\geq 1 \label{eq:pf-lm-z-x-diff-eq1}
	\end{align}
	Note that for all $t\geq 1$, we have
	\begin{align}
	\bar{\mb{z}}^{(t)} - \bar{\mb{x}}^{(t)} \overset{(a)}{=}  \frac{\beta}{1-\beta} [\bar{\mb{x}}^{(t)} - \bar{\mb{x}}^{t-1}] \overset{(b)}{=}-\frac{\beta \gamma }{1-\beta} \bar{\mb{u}}^{(t)} \label{eq:pf-lm-z-x-diff-eq2}
	\end{align}
	where (a) follows from \eqref{eq:def-z} and (b) follows from the second equation in \eqref{eq:bar-polyak}.

	Combining \eqref{eq:pf-lm-z-x-diff-eq1} and \eqref{eq:pf-lm-z-x-diff-eq2} yields
	\begin{align}
		\bar{\mb{z}}^{(t)} - \bar{\mb{x}}^{(t)} =  -\frac{\beta \gamma }{1-\beta} \sum_{\tau=0}^{t-1} \beta^{t-1-\tau} \Big[\frac{1}{N} \sum_{i=1}^N \mb{g}_i^{(\tau)} \Big], \quad \forall t\geq 1
	\end{align} 
	Define $s_t \defeq \sum_{\tau=0}^{t-1} \beta^{t-1-\tau} = \frac{1-\beta^{t}}{1-\beta}$. For all $t\geq 1$, we have 
	\begin{align}
		\norm{\bar{\mb{z}}^{(t)} - \bar{\mb{x}}^{(t)}}^2 &= \frac{\gamma^2 \beta^2}{(1-\beta)^2}s_t^2 \Big\Vert \sum_{\tau=0}^{t-1} \frac{\beta^{t-1-\tau}}{s_t}  \Big[\frac{1}{N} \sum_{i=1}^N \mb{g}_i^{(\tau)} \Big] \Big\Vert^2 \nonumber\\
		&\overset{(a)}{\leq} \frac{\gamma^2 \beta^2}{(1-\beta)^2} s_t^2 \sum_{\tau=0}^{t-1} \frac{\beta^{t-1-\tau}}{s_t} \Big\Vert \Big[\frac{1}{N} \sum_{i=1}^N \mb{g}_i^{(\tau)} \Big] \Big\Vert^2 \nonumber\\
		&= \frac{\gamma^2 \beta^2(1-\beta^t)}{(1-\beta)^3} \sum_{\tau=0}^{t-1} \beta^{t-1-\tau}\Big\Vert \Big[\frac{1}{N} \sum_{i=1}^N \mb{g}_i^{(\tau)} \Big] \Big\Vert^2 \nonumber\\
		&\leq \frac{\gamma^2 \beta^2}{(1-\beta)^3} \sum_{\tau=0}^{t-1} \beta^{t-1-\tau}\Big\Vert \Big[\frac{1}{N} \sum_{i=1}^N \mb{g}_i^{(\tau)} \Big] \Big\Vert^2 \label{eq:pf-sum-u-bound-eq1}
	\end{align}
	where (a) follows from the convexity of $\Vert \cdot \Vert^2$ and Jensen's inequality.
	
	Fix $T\geq 1$. Note that $\bar{\mb{z}}^{(0)} - \bar{\mb{x}}^{(0)}  = \mb{0}$. Summing \eqref{eq:pf-sum-u-bound-eq1} over $t\in\{1,2,\ldots, T-1\}$ yields
	\begin{align}
	\sum_{t=0}^{T-1}\norm{\bar{\mb{z}}^{(t)} - \bar{\mb{x}}^{(t)}}^2 &\leq \frac{\gamma^2 \beta^2}{(1-\beta)^3}\sum_{t=1}^{T-1} \sum_{\tau=0}^{t-1} \beta^{t-1-\tau}\Big\Vert \Big[\frac{1}{N} \sum_{i=1}^N \mb{g}_i^{(\tau)} \Big] \Big\Vert^2  \nonumber \\
	&=\frac{\gamma^2 \beta^2}{(1-\beta)^3} \sum_{\tau=0}^{T-2} \Big(\Big\Vert \Big[\frac{1}{N} \sum_{i=1}^N \mb{g}_i^{(\tau)} \Big] \Big\Vert^2 \sum_{l=\tau+1}^{T-1} \beta^{l-1-\tau} \Big) \nonumber\\
	&\overset{(a)}{\leq} \frac{\gamma^2 \beta^2}{(1-\beta)^4}  \sum_{\tau=0}^{T-2} \Big\Vert \Big[\frac{1}{N} \sum_{i=1}^N \mb{g}_i^{(\tau)} \Big] \Big\Vert^2\nonumber\\
	&\leq  \frac{\gamma^2 \beta^2}{(1-\beta)^4}  \sum_{\tau=0}^{T-1} \Big\Vert \Big[\frac{1}{N} \sum_{i=1}^N \mb{g}_i^{(\tau)} \Big] \Big\Vert^2
	\end{align}
	where (a) follows by noting that $\sum_{l=\tau+1}^{T-1} \beta^{l-1-\tau} = \frac{1-\beta^{T-\tau-1}}{1-\beta} \leq \frac{1}{1-\beta}$.
\end{proof}

\begin{Lem} \label{lm:diff-avg-per-node}
	Consider problem \eqref{eq:sto-opt} under Assumption \ref{ass:basic}.  Let $\{\bar{\mb{x}}^{(t)}\}_{t\geq 0}$ defined in \eqref{eq:def-x-bar} be the node averages of local solutions from Algorithm \ref{alg:prsgd-momentum} with Option I.  If $\gamma$ and $I$  are chosen to satisfy $\frac{12L^2\gamma^2I^2}{(1-\beta)^2} < 1$, then for all $T\geq 1$, we have 
	\begin{align*}
	\sum_{t=0}^{T-1}\frac{1}{N}\sum_{i=1}^{N}\mbb{E}[\norm{\bar{\mb{x}}^{(t)}  - \mb{x}_i^{(t)}}^2] \leq  \frac{1}{1-\frac{12L^2\gamma^2I^2}{(1-\beta)^2}} \frac{2 \gamma^2 I \sigma^2}{(1-\beta)^2} T+ \frac{1}{1- \frac{12L^2\gamma^2I^2}{(1-\beta)^2}} \frac{6\gamma^2I^2 \kappa^2}{(1-\beta)^2} T
	\end{align*}
\end{Lem}
\begin{proof}
	Since $\bar{\mb{x}}^{(t)}  - \mb{x}_i^{(t)}=\mb{0}$ when $t$ is a multiple of $I$, our focus is to develop upper bounds for $\frac{1}{N}\sum_{i=1}^{N}\mbb{E}[\norm{\bar{\mb{x}}^{(t)}  - \mb{x}_i^{(t)}}^2] $ when $t$ is not a multiple of $I$. Consider $t \geq 1$ that is not a multiple of $I$. Note that Algorithm \ref{alg:prsgd-momentum} restarts ``SGD with momentum" every $I$ iterations (by resetting $\mb{u}_i^{(t)}  =  \hat{\mb{u}}\defeq \frac{1}{N} \sum_{j=1}^{N} \mb{u}_j^{(t)}$ and $\mb{x}_i^{(t)} = \hat{\mathbf{x}} \defeq \frac{1}{N} \sum_{j=1}^{N} \mathbf{x}_{j}^{(t)}$).  Let $t_{0} < t$ be the largest iteration index that is a multiple of $I$, i.e., $t_0~\text{mod}~I = 0$. Note that we must have $t - t_0 < I$ and $\mb{u}_i^{(t_0)} = \hat{\mb{u}}, \mb{x}_i^{(t_0)} = \hat{\mb{x}}$.  For any $\tau\in\{t_0+1, \ldots, t\}$, by recursively applying the first equation in \eqref{eq:prsgd-polyak}, we have 
	\begin{align}
	\mb{u}_i^{(\tau)} = \beta^{\tau-t_0} \hat{\mb{u}} + \sum_{k=t_0}^{\tau-1} \beta^{\tau-1-k} \mb{g}_i^{(k)} \label{eq:pf-diff-avg-per-node-eq0}
	\end{align}
	For any  $\tau\in\{t_0+1, \ldots, t\}$, by the second equation in \eqref{eq:prsgd-polyak}, we have
	\begin{align}
	\mb{x}_i^{(\tau)} -\mb{x}_i^{(\tau-1)} = -\gamma \mb{u}_i^{(\tau)}
	\end{align}	 
	Summing over $\tau\in\{t_0+1, \ldots, t\}$ and noting that $\mb{x}_i^{(t_0)} = \hat{\mb{x}}$ yields
	\begin{align}
	\mb{x}_i^{(t)} &= \hat{\mb{x}} -\gamma \sum_{\tau=t_0+1}^{t} \mb{u}_i^{(\tau)} \nonumber\\
	&\overset{(a)}{=} \hat{\mb{x}} -\gamma \big(\sum_{\tau=t_0+1}^{t} \beta^{\tau-t_0} \big)\hat{\mb{u}} - \gamma \sum_{\tau = t_0+1}^{t} \sum_{k=t_0}^{\tau-1} \beta^{\tau-1-k} \mb{g}_i^{(k)} \nonumber \\
	&= \hat{\mb{x}} -\gamma\big(\sum_{j=1}^{t-t_0} \beta^{j} \big) \hat{\mb{u}} - \gamma \sum_{k = t_0}^{t-1} \sum_{j=0}^{t-1-k} \beta^{j} \mb{g}_i^{(k)} \nonumber\\
	&= \hat{\mb{x}} -\gamma\big(\sum_{j=1}^{t-t_0} \beta^{j} \big)  \hat{\mb{u}} - \gamma \sum_{\tau = t_0}^{t-1} \frac{1-\beta^{t-\tau}}{1-\beta} \mb{g}_i^{(\tau)} \label{eq:pf-diff-avg-per-node-eq1}
	\end{align}
	where (a) follows by substituting \eqref{eq:pf-diff-avg-per-node-eq0}.

	Using an argument similar to the above, we can further show 
	\begin{align}
	\bar{\mb{x}}^{(t)} = \hat{\mb{x}} -\gamma \big(\sum_{j=1}^{t-t_0} \beta^{j} \big) \hat{\mb{u}} - \gamma \sum_{\tau=t_0}^{t-1}\frac{1-\beta^{t-\tau}}{1-\beta} \frac{1}{N}\sum_{i=1}^N\mb{g}_i^{(\tau)} \label{eq:pf-diff-avg-per-node-eq2} 
	\end{align}
	Combining \eqref{eq:pf-diff-avg-per-node-eq1} and \eqref{eq:pf-diff-avg-per-node-eq2} yields
	\begin{align}
	&\frac{1}{N}\sum_{i=1}^{N}\mbb{E}[\norm{\bar{\mb{x}}^{(t)}  - \mb{x}_i^{(t)}}^2] \nonumber\\
	=& \gamma^2 \frac{1}{N} \sum_{i=1}^N \mbb{E}\Big[\Big\Vert \sum_{\tau = t_0}^{t-1} \big[ \mb{g}_i^{(\tau)} - \frac{1}{N} \sum_{j=1}^N \mb{g}_j^{(\tau)}\big]\frac{1-\beta^{t-\tau}}{1-\beta} \Big\Vert^2\Big] \nonumber\\
	\overset{(a)}{\leq}& 2\gamma^2\frac{1}{N}\sum_{i=1}^N\mbb{E}\Big[\Big\Vert \sum_{\tau = t_0}^{t-1} \big[ [\mb{g}_i^{(\tau)} - \nabla f_i(\mb{x}_i^{(\tau)}) ] - \frac{1}{N} \sum_{j=1}^N [\mb{g}_j^{(\tau)} - \nabla f_j(\mb{x}_j^{(\tau)})] \big]\frac{1-\beta^{t-\tau}}{1-\beta} \Big\Vert^2\Big]  \nonumber \\&+ 2\gamma^2 \frac{1}{N}\sum_{i=1}^N \mbb{E}\Big[\Big\Vert   \sum_{\tau=t_0}^{t-1} \big[\nabla f_i(\mb{x}_i^{(\tau)})- \frac{1}{N}\sum_{j=1}^{N} \nabla f_j(\mb{x}_j^{(\tau)}) \big]\frac{1-\beta^{t-\tau}}{1-\beta} \Big\Vert^2\Big] \label{eq:diff-avg-per-node-eq3}
	\end{align}
	where (a) follows from the basic inequality $\norm{\mb{a}_1 + \mb{a}_2}^2 \leq 2 \norm{\mb{a}_1}^2 + 2\norm{\mb{a}_2}^2$.

	Now we develop the respective upper bounds of the two terms on the right side of \eqref{eq:diff-avg-per-node-eq3}. We note that 
	\begin{align}
	&\frac{1}{N}\sum_{i=1}^N\mbb{E}\Big[\Big\Vert \sum_{\tau = t_0}^{t-1} \big[ [\mb{g}_i^{(\tau)} - \nabla f_i(\mb{x}_i^{(\tau)}) ] - \frac{1}{N} \sum_{j=1}^N [\mb{g}_j^{(\tau)} - \nabla f_j(\mb{x}_j^{(\tau)})] \big]\frac{1-\beta^{t-\tau}}{1-\beta} \Big\Vert^2\Big]  \nonumber\\
	\overset{(a)}{\leq}& \frac{1}{N}\sum_{i=1}^N\mbb{E}\Big[\Big\Vert \sum_{\tau = t_0}^{t-1} \big[ \mb{g}_i^{(\tau)} - \nabla f_i(\mb{x}_i^{(\tau)})  \big]\frac{1-\beta^{t-\tau}}{1-\beta} \Big\Vert^2\Big]  \nonumber\\
	\overset{(b)}{=} & \frac{1}{N}\sum_{i=1}^N \sum_{\tau = t_0}^{t-1} \mbb{E}\Big[\Big\Vert  \big[ \mb{g}_i^{(\tau)} - \nabla f_i(\mb{x}_i^{(\tau)})  \big]\frac{1-\beta^{t-\tau}}{1-\beta} \Big\Vert^2\Big]  \nonumber \\
	\overset{(c)}{\leq} &  \frac{I\sigma^2}{(1-\beta)^2}  \label{eq:diff-avg-per-node-eq4}
	\end{align}
	where (a) follows from the inequality $\frac{1}{N}\sum_{i=1}^N \norm{\mb{a}_i - [\frac{1}{N}\sum_{j=1}^N \mb{a}_j] }^2 = \frac{1}{N}\sum_{i=1}^{N} \norm{\mb{a}_i}^2 - \norm{\frac{1}{N}\sum_{i=1}^{N} \mb{a}_i}^2 \leq  \frac{1}{N}\sum_{i=1}^{N} \norm{\mb{a}_i}^2$ with $\mb{a}_i = \sum_{\tau = t_0}^{t-1}[ \mb{g}_i^{(\tau)} - \nabla f_i(\mb{x}_i^{(\tau)})]\frac{1-\beta^{t-\tau}}{1-\beta}$; (b) follows because $\mbb{E}[ \mb{g}_i^{(\tau_2)} - \nabla f_i(\mb{x}_i^{(\tau_2)}) \big\vert  \mb{g}_i^{(\tau_1)} - \nabla f_i(\mb{x}_i^{(\tau_1)}) ] = \mb{0}$ for any $\tau_2 > \tau_1$; (c) follows because $\abs{\frac{1-\beta^{t-\tau}}{1-\beta} } \leq \frac{1}{1-\beta}$, $t-t_0 < I$ and $\mbb{E}[\Vert \mb{g}_i^{(\tau)} - \nabla f_i(\mb{x}_i^{(\tau)})\Vert^2] \leq \sigma^2$ (by Assumption \ref{ass:basic}).
	
	We further note that 
	\begin{align}
	& \frac{1}{N}\sum_{i=1}^N \mbb{E}\Big[\Big\Vert   \sum_{\tau=t_0}^{t-1} \big[\nabla f_i(\mb{x}_i^{(\tau)})- \frac{1}{N}\sum_{j=1}^{N} \nabla f_j(\mb{x}_j^{(\tau)}) \big]\frac{1-\beta^{t-\tau}}{1-\beta} \Big\Vert^2\Big]\nonumber\\
	\overset{(a)}{\leq} & \frac{1}{N}\sum_{i=1}^N (t-t_0)\frac{1}{(1-\beta)^2}\sum_{\tau=t_0}^{t-1} \mbb{E}\Big[\Big\Vert  \nabla f_i(\mb{x}_i^{(\tau)})- \frac{1}{N}\sum_{j=1}^{N} \nabla f_j(\mb{x}_j^{(\tau)})  \Big\Vert^2\Big]\nonumber\\
	\overset{(b)}{\leq}&\frac{I}{(1-\beta)^2}\sum_{\tau=t_0}^{t-1}\frac{1}{N}\sum_{i=1}^N  \mbb{E}\Big[\Big\Vert    \big[\nabla f_i(\mb{x}_i^{(\tau)})- \frac{1}{N}\sum_{j=1}^{N} \nabla f_j(\mb{x}_j^{(\tau)}) \big] \Big\Vert^2\Big] \nonumber\\
	\overset{(c)}{\leq}& \frac{6L^2I}{(1-\beta)^2}\sum_{\tau=t_0}^{t-1} \frac{1}{N}\sum_{i=1}^N \mbb{E}\big[\Vert \mb{x}_i^{(\tau)} - \bar{\mb{x}}^{(\tau)}\Vert^2\big] +  \frac{3I}{(1-\beta)^2}\sum_{\tau=t_0}^{t-1} \frac{1}{N}\sum_{i=1}^N\mbb{E}\big[\Vert \nabla f_i(\bar{\mb{x}}^{(\tau)}) - \nabla f(\bar{\mb{x}}^{(\tau)})\Vert^2 \big] \nonumber\\
	\overset{(d)}{\leq}& \frac{6L^2I}{(1-\beta)^2}\sum_{\tau=t_0}^{t-1} \frac{1}{N}\sum_{i=1}^N \mbb{E}\big[\Vert \mb{x}_i^{(\tau)} - \bar{\mb{x}}^{(\tau)}\Vert^2\big] +  \frac{3I^2 \kappa^2}{(1-\beta)^2} \label{eq:diff-avg-per-node-eq5}
	\end{align} 
	where (a) follows by applying the basic inequality $\norm{\sum_{i=1}^{n} \mb{a}_i}^2 \leq n \sum_{i=1}^n \norm{\mb{a}_i}^2$ for any vectors $\mb{a}_i$ and any integer $n$ and noting that $\abs{\frac{1-\beta^{t-\tau}}{1-\beta}} \leq \frac{1}{1-\beta}$; (b) follows because $0 <  t-t_0 <  I$; (c) follows from Lemma \ref{lm:from-smooth-f};  and (d) follows because $\frac{1}{N}\sum_{i=1}^N\mbb{E}\big[\Vert \nabla f_i(\bar{\mb{x}}^{(\tau)}) - \nabla f(\bar{\mb{x}}^{(\tau)})\Vert^2 \big] \leq \kappa^2$ by Assumption \ref{ass:basic} and $0 < t-t_0 < I$.

	Substituting \eqref{eq:diff-avg-per-node-eq4}-\eqref{eq:diff-avg-per-node-eq5} into \eqref{eq:diff-avg-per-node-eq3} yields
	\begin{align}
	\frac{1}{N}\sum_{i=1}^{N}\mbb{E}[\norm{\bar{\mb{x}}^{(t)}  - \mb{x}_i^{(t)}}^2] \leq   \frac{2\gamma^2 I \sigma^2}{(1-\beta)^2} + \frac{12L^2\gamma^2I}{(1-\beta)^2}\sum_{\tau=t_0}^{t-1} \frac{1}{N}\sum_{i=1}^N \mbb{E}\big[\Vert \mb{x}_i^{(\tau)} - \bar{\mb{x}}^{(\tau)}\Vert^2\big] +  \frac{6\gamma^2I^2 \kappa^2}{(1-\beta)^2} \label{eq:diff-avg-per-node-eq6}
	\end{align}
	Recall that for each $t$, the index $t_0$ in the above equation is the largest integer such that $t_0 < t$,  $t_0~\text{mod}~I = 0$ and $t-t_0 <  I$. Summing \eqref{eq:diff-avg-per-node-eq6} over $t\in \{0,1, \ldots, T\}$ that are not a multiple of $I$ and noting that $\bar{\mb{x}}^{(t)}  - \mb{x}_i^{(t)}=\mb{0}$ if $t$ is a multiple of $I$ yields
	\begin{align}
	\sum_{t=0}^{T-1}\frac{1}{N}\sum_{i=1}^{N}\mbb{E}[\norm{\bar{\mb{x}}^{(t)}  - \mb{x}_i^{(t)}}^2] \leq   \frac{2 \gamma^2 I \sigma^2}{(1-\beta)^2}T+ \frac{12L^2\gamma^2I^2}{(1-\beta)^2}\sum_{t=0}^{T-1}\frac{1}{N}\sum_{i=1}^{N}\mbb{E}[\norm{\bar{\mb{x}}^{(t)}  - \mb{x}_i^{(t)}}^2] +  \frac{6\gamma^2I^2 \kappa^2}{(1-\beta)^2}T\nonumber
	\end{align}
	Collecting common terms and dividing both sides by $1-\frac{12L^2\gamma^2I^2}{(1-\beta)^2}$ yields
	\begin{align}
	\sum_{t=0}^{T-1}\frac{1}{N}\sum_{i=1}^{N}\mbb{E}[\norm{\bar{\mb{x}}^{(t)}  - \mb{x}_i^{(t)}}^2] \leq  \frac{1}{1-\frac{12L^2\gamma^2I^2}{(1-\beta)^2}} \frac{2 \gamma^2 I \sigma^2}{(1-\beta)^2}T + \frac{1}{1- \frac{12L^2\gamma^2I^2}{(1-\beta)^2}} \frac{6\gamma^2I^2 \kappa^2}{(1-\beta)^2}T \label{eq:diff-avg-per-node-eq7}
	\end{align}
\end{proof}

\subsubsection{Main Proof of Theorem \ref{thm:polyak-rate}} \label{sec:main-pf-thm-polyak-rate}

With the above lemmas, we are ready to present the main proof of Theorem \ref{thm:polyak-rate}.

Fix $t\geq 0$. By the smoothness of function $f(\cdot)$ (in Assumption \ref{ass:basic}), we have 
\begin{align}
\mbb{E}[f(\bar{\mb{z}}^{(t+1)})] \leq \mbb{E}[f(\bar{\mb{z}}^{(t)})] + \mbb{E}[\langle \nabla f(\bar{\mb{z}}^{(t)}), \bar{\mb{z}}^{(t+1)} - \bar{\mb{z}}^{(t)}\rangle] + \frac{L}{2} \mbb{E}[\norm{\bar{\mb{z}}^{(t+1)} - \bar{\mb{z}}^{(t)}}^2]  \label{eq:pf-thm-rate-eq1}
\end{align}

By Lemma \ref{lm:z-diff}, we have 
\begin{align}
&\mbb{E}[\langle \nabla f(\bar{\mb{z}}^{(t)}), \bar{\mb{z}}^{(t+1)} - \bar{\mb{z}}^{(t)}\rangle] \nonumber\\
=& -\frac{\gamma}{1-\beta} \mbb{E}[\langle \nabla f(\bar{\mb{z}}^{(t)}), \frac{1}{N}\sum_{i=1}^N\mb{g}_i^{(t)}\rangle]\nonumber\\
\overset{(a)}{=}& -\frac{\gamma}{1-\beta} \mbb{E}[\langle \nabla f(\bar{\mb{z}}^{(t)}), \frac{1}{N}\sum_{i=1}^N\nabla f_i(\mb{x}_i^{(t)})\rangle] \nonumber \\
=& -\frac{\gamma}{1-\beta} \mbb{E}[\langle \nabla f(\bar{\mb{z}}^{(t)}) - \nabla f(\bar{\mb{x}}^{(t)}), \frac{1}{N}\sum_{i=1}^N\nabla f_i(\mb{x}_i^{(t)})\rangle] -\frac{\gamma}{1-\beta} \mbb{E}[\langle \nabla f(\bar{\mb{x}}^{(t)}), \frac{1}{N}\sum_{i=1}^N\nabla f_i(\mb{x}_i^{(t)})\rangle] \label{eq:pf-thm-rate-eq2}
\end{align}
where (a) follows because $\bar{\mb{z}}^{(t)}$ and $\mb{x}_i^{(t)}$ are determined by $\bs{\xi}^{[t-1]} = [\bs{\xi}^{(0)}, \ldots, \bs{\xi}^{(t-1)}]$, which is independent of $\bs{\xi}^{(t)}$, and $\mbb{E}[\mb{g}_i^{(t)} | \bs{\xi}^{[t-1]}] = \mbb{E}[\mb{g}_i^{(t)}] = \nabla f_i (\mb{x}_i^{(t)})$.

We note that 
\begin{align}
&-\frac{\gamma}{1-\beta} \langle \nabla f(\bar{\mb{z}}^{(t)})-\nabla f(\bar{\mb{x}}^{(t)}), \frac{1}{N}\sum_{i=1}^N\nabla f_i(\mb{x}_i^{(t)})\rangle  \nonumber\\
\overset{(a)}{\leq}& \frac{1-\beta}{2\beta L} \norm{\nabla f(\bar{\mb{z}}^{(t)})-\nabla f(\bar{\mb{x}}^{(t)})}^2 + \frac{\beta L \gamma^2}{2(1-\beta)^3} \norm{\frac{1}{N}\sum_{i=1}^N\nabla f_i (\mb{x}_i^{(t)})}^2 \nonumber \\
\overset{(b)}{\leq}& \frac{(1-\beta)L}{2\beta} \norm{ \bar{\mb{z}}^{(t)}- \bar{\mb{x}}^{(t)}}^2 +\frac{\beta L \gamma^2}{2(1-\beta)^3} \norm{\frac{1}{N}\sum_{i=1}^N\nabla f_i (\mb{x}_i^{(t)})}^2 \label{eq:pf-thm-rate-eq3}
\end{align}
where (a) follows by applying the basic inequality $\langle \mb{a}, \mb{b}\rangle \leq \frac{1}{2}\norm{\mb{a}}^2 + \frac{1}{2}\norm{\mb{b}}^2$ with $\mb{a} =- \frac{\sqrt{1-\beta}}{\sqrt{\beta L}} [\nabla f(\bar{\mb{z}}^{(t)})-\nabla f(\bar{\mb{x}}^{(t)})]$ and $\mb{b} = \frac{\gamma \sqrt{\beta L}}{(1-\beta)^{3/2}} \frac{1}{N}\sum_{i=1}^N\nabla f_i(\mb{x}_i^{(t)})$; and (b) follows from the smoothness of function $f(\cdot)$.

Applying the basic identity $\langle \mb{a}, \mb{b}\rangle = \frac{1}{2}[\norm{\mb{a}}^2 + \norm{\mb{b}}^2 - \norm{\mb{a}-\mb{b}}^2]$ with $\mb{a} = \nabla f(\bar{\mb{x}}^{(t)})$ and $\mb{b} =\frac{1}{N}\sum_{i=1}^N\nabla f_i (\mb{x}_i^{(t)})$ yields
\begin{align}
&\langle \nabla f(\bar{\mb{x}}^{(t)}), \frac{1}{N}\sum_{i=1}^N\nabla f_i(\mb{x}_i^{(t)})\rangle \nonumber\\
=&  \frac{1}{2} \Big(\norm{\nabla f(\bar{\mb{x}}^{(t)})}^2 + \norm{\frac{1}{N}\sum_{i=1}^N\nabla f_i(\mb{x}_i^{(t)})}^2 - \norm{\nabla f(\bar{\mb{x}}^{(t)})-\frac{1}{N}\sum_{i=1}^N\nabla f_i(\mb{x}_i^{(t)})}^2 \Big) \nonumber \\
\overset{(a)}{\geq} & \frac{1}{2} \Big(\norm{\nabla f(\bar{\mb{x}}^{(t)})}^2 + \norm{\frac{1}{N}\sum_{i=1}^N\nabla f_i(\mb{x}_i^{(t)})}^2 -  L^2 \frac{1}{N} \sum_{i=1}^N \Vert \bar{\mb{x}}^{(t)} - \mb{x}_i^{(t)}\Vert^2 \Big) \label{eq:pf-thm-rate-eq4}
\end{align}
where (a) follows because $\Vert\nabla f(\bar{\mb{x}}^{(t)}) - \frac{1}{N}\sum_{i=1}^N \nabla f_i (\mb{x}_i^{(t)})\Vert^2 = \Vert \frac{1}{N} \sum_{i=1}^N \nabla f_i(\bar{\mb{x}}^{(t)}) - \frac{1}{N}\sum_{i=1}^N \nabla f_i (\mb{x}_i^{(t)})\Vert^2 \leq \frac{1}{N} \sum_{i=1}^{N} \Vert \nabla f_i(\bar{\mb{x}}^{(t)})  - \nabla f_i (\mb{x}_i^{(t)})\Vert^2 \leq \frac{1}{N} \sum_{i=1}^N L^2 \Vert \bar{\mb{x}}^{(t)} - \mb{x}_i^{(t)}\Vert^2$, where the first inequality follows from the convexity of $\Vert \cdot \Vert^2$ and Jensen's inequality and the second inequality follows from the smoothness of each $f_i(\cdot)$.

Substituting \eqref{eq:pf-thm-rate-eq3}-\eqref{eq:pf-thm-rate-eq4} into \eqref{eq:pf-thm-rate-eq2} yields
\begin{align}
&\mbb{E}[\langle \nabla f(\bar{\mb{z}}^{(t)}), \bar{\mb{z}}^{(t+1)} - \bar{\mb{z}}^{(t)}\rangle] \nonumber\\
\leq &\frac{(1-\beta)L}{2\beta} \mbb{E}[\norm{\bar{\mb{z}}^{(t)} - \bar{\mb{x}}^{(t)} }^2] + \big( \frac{\beta L \gamma^2}{2(1-\beta)^3} - \frac{\gamma}{2(1-\beta)}\big) \mbb{E}[\norm{\frac{1}{N}\sum_{i=1}^N\nabla f_i (\mb{x}_i^{(t)})}^2]  - \frac{\gamma}{2(1-\beta)}\mbb{E}[\norm{\nabla f(\bar{\mb{x}}^{(t)})}^2] \nonumber \\ &+ \frac{\gamma L^2}{2(1-\beta)} \frac{1}{N}\sum_{i=1}^N\mbb{E}[ \Vert \bar{\mb{x}}^{(t)} - \mb{x}_i^{(t)}\Vert^2]  \label{eq:pf-thm-rate-eq5}
\end{align}

By Lemma \ref{lm:z-diff}, we have 
\begin{align}
\mbb{E}[\norm{\bar{\mb{z}}^{(t+1)} - \bar{\mb{z}}^{(t)}}^2] = \frac{\gamma^2}{(1-\beta)^2} \mbb{E}[\norm{\frac{1}{N}\sum_{i=1}^N \mb{g}_i^{(t)}}^2]  \label{eq:pf-thm-rate-eq6}
\end{align}

Substituting \eqref{eq:pf-thm-rate-eq5}-\eqref{eq:pf-thm-rate-eq6} into \eqref{eq:pf-thm-rate-eq1} yields

\begin{align}
\mbb{E}[f(\bar{\mb{z}}^{(t+1)})] \leq &\mbb{E}[f(\bar{\mb{z}}^{(t)})] + \frac{(1-\beta)L}{2\beta}  \mbb{E}[\norm{\bar{\mb{z}}^{(t)} - \bar{\mb{x}}^{(t)} }^2] + \big( \frac{\beta L \gamma^2}{2(1-\beta)^3} - \frac{\gamma}{2(1-\beta)}\big) \mbb{E}[\norm{\frac{1}{N}\sum_{i=1}^N\nabla f_i (\mb{x}_i^{(t)})}^2] \nonumber\\ &- \frac{\gamma}{2(1-\beta)}\mbb{E}[\norm{\nabla f(\bar{\mb{x}}^{(t)})}^2] + \frac{\gamma L^2}{2(1-\beta)} \frac{1}{N}\sum_{i=1}^N\mbb{E}[ \Vert \bar{\mb{x}}^{(t)} - \mb{x}_i^{(t)}\Vert^2] + \frac{\gamma^2L}{2(1-\beta)^2} \mbb{E}[\norm{\frac{1}{N}\sum_{i=1}^N \mb{g}_i^{(t)}}^2]
\end{align}
Dividing both sides by $\frac{\gamma}{2(1-\beta)}$ and rearranging terms yields
\begin{align}
\mbb{E}[\norm{\nabla f(\bar{\mb{x}}^{(t)})}^2] \leq &\frac{2(1-\beta)}{\gamma} \big(\mbb{E}[f(\bar{\mb{z}}^{(t)})] -\mbb{E}[f(\bar{\mb{z}}^{(t+1)})]  \big) - (1- \frac{L \gamma\beta}{(1-\beta)^2})\mbb{E}[\norm{\frac{1}{N}\sum_{i=1}^N\nabla f_i (\mb{x}_i^{(t)})}^2]  \nonumber\\
&+ \frac{(1-\beta)^2 L}{\beta \gamma} \mbb{E}[\norm{ \bar{\mb{z}}^{(t)} - \bar{\mb{x}}^{(t)} }^2] +  \frac{L\gamma}{(1-\beta)} \mbb{E}[\norm{\frac{1}{N}\sum_{i=1}^N \mb{g}_i^{(t)}}^2] + L^2 \frac{1}{N}\sum_{i=1}^N\mbb{E}[ \Vert \bar{\mb{x}}^{(t)} - \mb{x}_i^{(t)}\Vert^2]  \label{eq:pf-thm-rate-eq7}
\end{align}
Summing over $t\in\{0,1,\ldots, T-1\}$ 
\begin{align}
&\sum_{t=0}^{T-1}\mbb{E}[\norm{\nabla f(\bar{\mb{x}}^{(t)})}^2] \nonumber \\
\leq &\frac{2(1-\beta)}{\gamma } \big(\mbb{E}[f(\bar{\mb{z}}^{(0)})] -\mbb{E}[f(\bar{\mb{z}}^{(T)})]  \big) - \Big(1- \frac{L\gamma \beta}{(1-\beta)^2}\Big) \sum_{t=0}^{T-1}\mbb{E}[\norm{\frac{1}{N}\sum_{i=1}^N\nabla f_i (\mb{x}_i^{(t)})}^2] \nonumber\\
&+ \frac{(1-\beta)^2 L}{\beta \gamma} \sum_{t=0}^{T-1}\mbb{E}[\norm{ \bar{\mb{z}}^{(t)} - \bar{\mb{x}}^{(t)} }^2] + \frac{L\gamma}{(1-\beta)} \sum_{t=0}^{T-1}\mbb{E}[\norm{\frac{1}{N}\sum_{i=1}^N \mb{g}_i^{(t)}}^2] + L^2 \sum_{t=0}^{T-1} \frac{1}{N}\sum_{i=1}^N\mbb{E}[ \Vert \bar{\mb{x}}^{(t)} - \mb{x}_i^{(t)}\Vert^2] \nonumber \\
\overset{(a)}{\leq} &\frac{2(1-\beta)}{\gamma } \big(\mbb{E}[f(\bar{\mb{z}}^{(0)})] -\mbb{E}[f(\bar{\mb{z}}^{(T)})]  \big) - \Big(1- \frac{L\gamma\beta}{(1-\beta)^2}\Big) \sum_{t=0}^{T-1}\mbb{E}[\norm{\frac{1}{N}\sum_{i=1}^N\nabla f_i (\mb{x}_i^{(t)})}^2]  \nonumber \\ & + \frac{L\gamma}{(1-\beta)^2} \sum_{t=0}^{T-1}\mbb{E}[\norm{\frac{1}{N}\sum_{i=1}^N \mb{g}_i^{(t)}}^2] + L^2 \sum_{t=0}^{T-1} \frac{1}{N}\sum_{i=1}^N\mbb{E}[ \Vert \bar{\mb{x}}^{(t)} - \mb{x}_i^{(t)}\Vert^2] \nonumber\\
\overset{(b)}{\leq} &\frac{2(1-\beta)}{\gamma } \big(\mbb{E}[f(\bar{\mb{z}}^{(0)})] -\mbb{E}[f(\bar{\mb{z}}^{(T)})]  \big) -  \Big(1- \frac{ (1+\beta)L\gamma}{(1-\beta)^2}\Big) \sum_{t=0}^{T-1}\mbb{E}[\norm{\frac{1}{N}\sum_{i=1}^N\nabla f_i (\mb{x}_i^{(t)})}^2]  +  \frac{L\gamma }{(1-\beta)^2}  \frac{\sigma^2}{N}T \nonumber \\ &+ \frac{1}{1-\frac{12L^2\gamma^2I^2}{(1-\beta)^2}} \frac{2 L^2\gamma^2 I \sigma^2}{(1-\beta)^2} T+ \frac{1}{1- \frac{12L^2\gamma^2I^2}{(1-\beta)^2}} \frac{6L^2\gamma^2I^2 \kappa^2}{(1-\beta)^2}T  \nonumber
\end{align}
\begin{align}
\overset{(c)}{\leq} &\frac{2(1-\beta)}{\gamma } \big(f(\bar{\mb{x}}^{(0)}) - f^\ast \big) +  \frac{L\gamma }{(1-\beta)^2}  \frac{\sigma^2}{N}T+ \frac{1}{1-\frac{12L^2\gamma^2I^2}{(1-\beta)^2}} \frac{2 L^2\gamma^2 I \sigma^2}{(1-\beta)^2} T+ \frac{1}{1- \frac{12L^2\gamma^2I^2}{(1-\beta)^2}} \frac{6L^2\gamma^2I^2 \kappa^2}{(1-\beta)^2}T
\end{align}
where (a) follow by using Lemma \ref{lm:z-x-diff} and $\frac{\beta L\gamma}{(1-\beta)^2} + \frac{L\gamma}{1-\beta} = \frac{L\gamma}{(1-\beta)^2}$; (b) follows by applying Lemma \ref{lm:diff-avg-per-node} and by noting that $\mbb{E}[\norm{\frac{1}{N}\sum_{i=1}^N \mb{g}_i^{(t)}}^2] \leq  \frac{1}{N}\sigma^{2}  +  \mathbb{E} [ \Vert \frac{1}{N} \sum_{i=1}^{N} \nabla f_{i}(\mathbf{x}_{i}^{(t)})\Vert^{2}]$ by Lemma \ref{lm:expected-gradient} ; and (c) follows because $\gamma $ is chosen to ensure $1- \frac{ (1+\beta)L\gamma}{(1-\beta)^2} \geq 0$, $\bar{\mb{z}}^{(0)} = \bar{\mb{x}}^{(0)}$ by definition in \eqref{eq:def-z}, and $f^\ast$ is the minimum value of problem \eqref{eq:sto-opt}.

Dividing both sides by $T$ yields
\begin{align}
&\frac{1}{T} \sum_{t=0}^{T-1}\mbb{E}[\norm{\nabla f(\bar{\mb{x}}^{(t)})}^2] \nonumber \\
\leq&\frac{2(1-\beta)}{\gamma T} \big(f(\bar{\mb{x}}^{(0)}) - f^\ast \big) +  \frac{L\gamma }{(1-\beta)^2}  \frac{\sigma^2}{N}+ \frac{1}{1-\frac{12L^2\gamma^2I^2}{(1-\beta)^2}} \frac{2 L^2\gamma^2 I \sigma^2}{(1-\beta)^2} + \frac{1}{1- \frac{12L^2\gamma^2I^2}{(1-\beta)^2}} \frac{6L^2\gamma^2I^2 \kappa^2}{(1-\beta)^2} \nonumber\\
\overset{(a)}{\leq}&\frac{2(1-\beta)}{\gamma T} \big(f(\bar{\mb{x}}^{(0)}) - f^\ast \big) +  \frac{L\gamma }{(1-\beta)^2}  \frac{\sigma^2}{N}+  \frac{3 L^2\gamma^2 I \sigma^2}{(1-\beta)^2} +  \frac{9L^2\gamma^2I^2 \kappa^2}{(1-\beta)^2} \nonumber\\
=& O(\frac{1}{\gamma T})  + O(\frac{\gamma}{N} \sigma^2) + O(\gamma^2 I \sigma^2)+ O(\gamma^2 I^2 \kappa^2)
\end{align}
where (a) follows because $I \leq \frac{1-\beta}{6L\gamma}$ is chosen to ensure $\frac{1}{1-\frac{12L^2\gamma^2I^2}{(1-\beta)^2}} \leq \frac{3}{2}$.

\subsection{On the equivalence between \eqref{eq:prsgd-nesterov} and \eqref{eq:nesterov-common-version}} \label{sec:two-equivalent-nesterov}
In this subsection, we show both and yield the same solution sequences $\{\mb{x}_{i}^{(t)}\}_{t\geq 0}$ assume they are initialized at the same $\mb{x}_i^{(0)}$. It is easy to verify that  \eqref{eq:prsgd-nesterov} and \eqref{eq:nesterov-common-version} yield the same $\mb{x}_i^{(1)}$ by noting that $\mb{y}_i^{(0)} = 0$ and $\mb{u}_i^{(0)} = \mb{0}$. Next, we show they yield the same $\{\mb{x}_i^{(0)}\}_{t\geq 2}$.

Substituting the second equation of \eqref{eq:prsgd-nesterov} into the third equation of \eqref{eq:prsgd-nesterov} yields
\begin{align}
\mb{x}_i^{(t)} = \mb{x}_i^{(t-1)} - \gamma \beta \mb{u}_i^{(t)} - \gamma \mb{g}_i^{(t-1)}, \forall t\geq 1 \label{eq:pf-equivalence-nestero-eq1}
\end{align}
Fix $t\geq 2$. Substituting the first equation of \eqref{eq:prsgd-nesterov} into the above equation yields
\begin{align}
\mb{x}_i^{(t)} =& \mb{x}_i^{(t-1)} - \gamma \mb{g}_i^{(t-1)} - \gamma \beta \mb{g}_i^{(t-1)}  - \gamma \beta^2 \mb{u}_i^{(t-1)} \nonumber \\
\overset{(a)}{=}& \mb{x}_i^{(t-1)} - \gamma \mb{g}_i^{(t-1)} - \gamma \beta \mb{g}_i^{(t-1)} + \gamma \beta \mb{g}_i^{(t-2)} + \beta[\mb{x}_i^{(t-1)} - \mb{x}_i^{(t-2)}] \nonumber\\
=&  \mb{x}_i^{(t-1)} - \gamma \mb{g}_i^{(t-1)} + \beta \big[ [\mb{x}_i^{(t-1)}- \gamma  \mb{g}_i^{(t-1)}] - [\mb{x}_i^{(t-2)}- \gamma  \mb{g}_i^{(t-2)}] \big], \forall t\geq 2.
\end{align}
where (a) follows by noting that $-\gamma \beta^2 \mb{u}_i^{(t-1)}  = \gamma \beta \mb{g}_i^{(t-2)} + \beta[\mb{x}_i^{(t-1)} - \mb{x}_i^{(t-2)}], \forall t\geq 2$ from \eqref{eq:pf-equivalence-nestero-eq1}.

Now if we define $\mb{y}_i^{(t)} \defeq  \mb{x}_i^{(t-1)} - \gamma \mb{g}_i^{(t-1)}$, which is the first equation in \eqref{eq:nesterov-common-version}, then the above equation can be written as 
\begin{align}
\mb{x}_i^{(t)} = \mb{y}_i^{(t)} + \beta [\mb{y}_i^{(t)} - \mb{y}_i^{(t-1)}], \quad \forall t\geq 2
\end{align}
which is exactly the second equation in \eqref{eq:nesterov-common-version}.

Thus, we can conclude that  \eqref{eq:prsgd-nesterov} and \eqref{eq:nesterov-common-version} yield the same solution sequences $\{\mb{x}_i^{(0)}\}_{t\geq 0}$ (with different auxiliary/buffer sequences.)

\subsection{Proof of Theorem \ref{thm:nesterov-rate}} \label{sec:pf-thm-nesterov-rate}

To establish the convergence rate for Algorithm \ref{alg:prsgd-momentum} with Nesterov's momentum, we introduce a different auxiliary sequence $\{\bar{\mb{y}}^{(t)}\}_{t\geq 0}$ given by  
\begin{align}
\bar{\mb{y}}^{(t)} \defeq \left \{  \begin{array}{ll} \bar{\mb{x}}^{(t)}, & \quad t=0 \\
\frac{1}{1-\beta} \bar{\mb{x}}^{(t)} - \frac{\beta}{1-\beta}\bar{\mb{x}}^{(t-1)} + \frac{\gamma \beta}{1-\beta} \frac{1}{N} \sum_{i=1}^{N} \mb{g}_i^{(t-1)}, &\quad t\geq 1 \end{array} \right. \label{eq:def-y}
\end{align}
where $\bar{\mb{x}}^{(t)}$ defined in \eqref{eq:def-x-bar} are the node averages of local solutions from Algorithm \ref{alg:prsgd-momentum} with Option II.  

Using the notation $\bar{\mb{u}}^{(t)} \defeq \frac{1}{N} \sum_{i=1}^N \mb{u}_i^{(t)}$ in \eqref{eq:def-u-bar} and $\bar{\mb{v}}^{(t)} \defeq \frac{1}{N} \sum_{i=1}^N \mb{v}_i^{(t)}$ with $\mb{u}_i^{(t)}$ and $\mb{v}_i^{(t)}$ from Algorithm \ref{alg:prsgd-momentum} with Option II, we shall have
\begin{align}
\begin{cases}
\bar{\mb{u}}^{(t)} &= \beta \bar{\mb{u}}^{(t-1)} + \frac{1}{N}\sum_{i=1}^N \mb{g}_i^{(t-1)} \\
\bar{\mb{v}}^{(t)} &= \beta \bar{\mb{u}}^{(t)} + \frac{1}{N}\sum_{i=1}^N \mb{g}_i^{(t-1)} \\
\bar{\mb{x}}^{(t)}  &= \bar{\mb{x}}^{(t-1)}  - \gamma  \bar{\mb{v}}^{(t)}
\end{cases} \label{eq:bar-nesterov}
\end{align}

The main proof shall follow similar steps as in our main proof for Polyak's momentum in Section \ref{sec:main-pf-thm-polyak-rate} as long as we prove the counterpart of Lemmas \ref{lm:z-diff}-\ref{lm:diff-avg-per-node} for Algorithm \ref{alg:prsgd-momentum} with Option II. 

\subsubsection{Useful Lemmas}

While the sequence $\{\bar{\mb{y}}^{(t)}\}_{t\geq 0}$ introduced in \eqref{eq:def-y} is different from $\{\bar{\mb{z}}^{(t)}\}_{t\geq 0}$ defined in \eqref{eq:def-z} for Polyak's momentum, the next lemma shows that $\bar{\mb{y}}^{(t+1)} - \bar{\mb{y}}^{(t)}$ for Algorithm \ref{alg:prsgd-momentum} with Option II is equal to $\bar{\mb{z}}^{(t+1)} - \bar{\mb{z}}^{(t)}$ for Algorithm \ref{alg:prsgd-momentum} with Option I. 

\begin{Lem}\label{lm:y-diff}
Consider the sequence $\{\bar{\mb{y}}^{(t)}\}_{t\geq 0}$ defined in \eqref{eq:def-y}. Algorithm \ref{alg:prsgd-momentum} with Option II ensures that for all $t\geq 0$, we have 
\begin{align}
\bar{\mb{y}}^{(t+1)} - \bar{\mb{y}}^{(t)} = -\frac{\gamma}{1-\beta}\frac{1}{N}\sum_{i=1}^N \mb{g}_i^{(t)}.
\end{align}
\end{Lem}
\begin{proof}  The case $\bar{\mb{y}}^{(1)} - \bar{\mb{y}}^{(0)} = -\frac{\gamma}{1-\beta}\frac{1}{N}\sum_{i=1}^N \mb{g}_i^{(0)}$ can be shown directly by combining \eqref{eq:def-y} and \eqref{eq:bar-nesterov} with $t=0$. To prove the case $t\geq 1$, we note that 
\begin{align*}
			&\bar{\mb{y}}^{(t+1)} - \bar{\mb{y}}^{(t)} \\
			\overset{(a)}{=}& \frac{1}{1-\beta} [ \bar{\mb{x}}^{(t+1)} -\bar{\mb{x}}^{(t)}] - \frac{\beta}{1-\beta}[\bar{\mb{x}}^{(t)} - \bar{\mb{x}}^{(t-1)}] + \frac{\gamma \beta}{1-\beta} [\frac{1}{N}\sum_{i=1}^N \mb{g}_i^{(t)} - \frac{1}{N}\sum_{i=1}^N \mb{g}_i^{(t-1)}] \\
			\overset{(b)}{=}&- \frac{\gamma}{1-\beta} \bar{\mb{v}}^{(t+1)} + \frac{\gamma \beta}{1-\beta} \bar{\mb{v}}^{(t)} + \frac{\gamma \beta}{1-\beta} [\frac{1}{N}\sum_{i=1}^N \mb{g}_i^{(t)} - \frac{1}{N}\sum_{i=1}^N \mb{g}_i^{(t-1)}]\\
			\overset{(c)}{=}& - \frac{\gamma \beta^2}{1-\beta} \bar{\mb{u}}^{(t)} - \frac{\gamma(1+\beta)}{(1-\beta)} \frac{1}{N}\sum_{i=1}^N \mb{g}_i^{(t)} + \frac{\gamma \beta^2}{1-\beta} \bar{\mb{u}}^{(t)} +  \frac{\gamma \beta}{1-\beta}  \frac{1}{N}\sum_{i=1}^N \mb{g}_i^{(t-1)} +  \frac{\gamma \beta}{1-\beta} [\frac{1}{N}\sum_{i=1}^N \mb{g}_i^{(t)} - \frac{1}{N}\sum_{i=1}^N \mb{g}_i^{(t-1)}]\\
			=& -\frac{\gamma}{1-\beta}\frac{1}{N}\sum_{i=1}^N \mb{g}_i^{(t)}
			\end{align*}
			where (a) follows from \eqref{eq:def-y}; (b) follows from the third equation in \eqref{eq:bar-nesterov}; and (c) follows by substituting  $\bar{\mb{v}}^{(t)}  = \beta \bar{\mb{u}}^{(t)} + \frac{1}{N}\sum_{i=1}^N \mb{g}_i^{(t-1)}$ and $\bar{\mb{v}}^{(t+1)}  = \beta^2 \bar{\mb{u}}^{(t)} + (1+\beta) \frac{1}{N}\sum_{i=1}^N \mb{g}_i^{(t)}$, which is implied by combining the first and second equations in \eqref{eq:bar-nesterov}.
\end{proof}

\begin{Lem} \label{lm:y-x-diff}
	Let $\{\bar{\mb{x}}^{(t)}\}_{t\geq 0}$ defined in \eqref{eq:def-x-bar} be the node averages of local solutions from Algorithm \ref{alg:prsgd-momentum} with Option II.  Let $\{\bar{\mb{y}}^{(t)}\}_{t\geq 0}$ be defined in \eqref{eq:def-y}. For all $T\geq 1$, Algorithm \ref{alg:prsgd-momentum} with Option II ensures that 
	\begin{align}
		\sum_{t=0}^{T-1} \norm{ \bar{\mb{y}}^{(t)} - \bar{\mb{x}}^{(t)} }^2 \leq \frac{\gamma^2 \beta^4}{(1-\beta)^4}  \sum_{t=0}^{T-1} \Big\Vert \Big[\frac{1}{N} \sum_{i=1}^N \mb{g}_i^{(t)} \Big] \Big\Vert^2
	\end{align}
\end{Lem}
\begin{proof}
	Recall that $\bar{\mb{u}}^{(0)} = \mb{0}$.  Recursively applying the first equation in \eqref{eq:bar-nesterov} for $t$ times yields:
	\begin{align}
		\bar{\mb{u}}^{(t)} =  \sum_{\tau=0}^{t-1} \beta^{t-1-\tau} \Big[\frac{1}{N} \sum_{i=1}^N \mb{g}_i^{(\tau)} \Big], \quad \forall t\geq 1 \label{eq:pf-lm-y-x-diff-eq1}
	\end{align}
	Note that for all $t\geq 1$, we have 
	\begin{align}
	\bar{\mb{y}}^{(t)} - \bar{\mb{x}}^{(t)} \overset{(a)}{=}&  \frac{\beta}{1-\beta} [\bar{\mb{x}}^{(t)} - \bar{\mb{x}}^{t-1}] + \frac{\gamma \beta}{1-\beta} \frac{1}{N} \sum_{i=1}^{N} \mb{g}_i^{(t-1)} \nonumber\\
	\overset{(b)}{=}&-\frac{\gamma \beta  }{1-\beta} \bar{\mb{v}}^{(t)} + \frac{\gamma \beta}{1-\beta} \frac{1}{N} \sum_{i=1}^{N} \mb{g}_i^{(t-1)} \nonumber \\
	\overset{(c)}{=}& -\frac{\gamma\beta^2  }{1-\beta} \bar{\mb{u}}^{(t)}  \label{eq:pf-lm-y-x-diff-eq2}
	\end{align}
	where (a) follows from \eqref{eq:def-y}; (b) follows from the third equation in \eqref{eq:bar-nesterov}; and (c) follows from the second equation in \eqref{eq:bar-nesterov}.
	
	Combining \eqref{eq:pf-lm-y-x-diff-eq1} and \eqref{eq:pf-lm-y-x-diff-eq2} yields
	\begin{align}
		\bar{\mb{y}}^{(t)} - \bar{\mb{x}}^{(t)} =  -\frac{\gamma \beta^2 }{1-\beta} \sum_{\tau=0}^{t-1} \beta^{t-1-\tau} \Big[\frac{1}{N} \sum_{i=1}^N \mb{g}_i^{(\tau)} \Big], \quad \forall t\geq 1
	\end{align} 
	Define $s_t \defeq \sum_{\tau=0}^{t-1} \beta^{t-1-\tau} = \frac{1-\beta^{t}}{1-\beta}$. For all $t\geq 1$, we have 
	\begin{align}
		\norm{\bar{\mb{y}}^{(t)} - \bar{\mb{x}}^{(t)}}^2 &= \frac{\gamma^2 \beta^4}{(1-\beta)^2}s_t^2 \Big\Vert \sum_{\tau=0}^{t-1} \frac{\beta^{t-1-\tau}}{s_t}  \Big[\frac{1}{N} \sum_{i=1}^N \mb{g}_i^{(\tau)} \Big] \Big\Vert^2 \nonumber\\
		&\overset{(a)}{\leq} \frac{\gamma^2 \beta^4}{(1-\beta)^2} s_t^2 \sum_{\tau=0}^{t-1} \frac{\beta^{t-1-\tau}}{s_t} \Big\Vert \Big[\frac{1}{N} \sum_{i=1}^N \mb{g}_i^{(\tau)} \Big] \Big\Vert^2 \nonumber\\
		&= \frac{\gamma^2 \beta^4(1-\beta^t)}{(1-\beta)^3} \sum_{\tau=0}^{t-1} \beta^{t-1-\tau}\Big\Vert \Big[\frac{1}{N} \sum_{i=1}^N \mb{g}_i^{(\tau)} \Big] \Big\Vert^2 \nonumber\\
		&\leq \frac{\gamma^2 \beta^4}{(1-\beta)^3} \sum_{\tau=0}^{t-1} \beta^{t-1-\tau}\Big\Vert \Big[\frac{1}{N} \sum_{i=1}^N \mb{g}_i^{(\tau)} \Big] \Big\Vert^2 \label{eq:pf-lm-y-x-diff-eq3}
	\end{align}
	where (a) follows from the convexity of $\Vert \cdot \Vert^2$ and Jensen's inequality.
	
	Fix $T\geq 1$. Recall that $\bar{\mb{y}}^{(0)} - \bar{\mb{x}}^{(0)}  = \mb{0}$. Summing \eqref{eq:pf-lm-y-x-diff-eq3} over $t\in\{1,2,\ldots, T-1\}$ yields
	\begin{align}
	\sum_{t=0}^{T-1}\norm{\bar{\mb{y}}^{(t)} - \bar{\mb{x}}^{(t)}}^2 &\leq \frac{\gamma^2 \beta^4}{(1-\beta)^3}\sum_{t=1}^{T-1} \sum_{\tau=0}^{t-1} \beta^{t-1-\tau}\Big\Vert \Big[\frac{1}{N} \sum_{i=1}^N \mb{g}_i^{(\tau)} \Big] \Big\Vert^2  \nonumber \\
	&=\frac{\gamma^2 \beta^4}{(1-\beta)^3} \sum_{\tau=0}^{T-2} \Big(\Big\Vert \Big[\frac{1}{N} \sum_{i=1}^N \mb{g}_i^{(\tau)} \Big] \Big\Vert^2 \sum_{l=\tau+1}^{T-1} \beta^{l-1-\tau} \Big) \nonumber\\
	&\overset{(a)}{\leq} \frac{\gamma^2 \beta^4}{(1-\beta)^4}  \sum_{\tau=0}^{T-2} \Big\Vert \Big[\frac{1}{N} \sum_{i=1}^N \mb{g}_i^{(\tau)} \Big] \Big\Vert^2 \nonumber\\
	&\leq \frac{\gamma^2 \beta^4}{(1-\beta)^4}  \sum_{\tau=0}^{T-1} \Big\Vert \Big[\frac{1}{N} \sum_{i=1}^N \mb{g}_i^{(\tau)} \Big] \Big\Vert^2
	\end{align}
	where (a) follows by noting that $\sum_{l=\tau+1}^{T-1} \beta^{l-1-\tau} = \frac{1-\beta^{T-\tau-1}}{1-\beta} \leq \frac{1}{1-\beta}$.
\end{proof}

\begin{Lem} \label{lm:nesterov-diff-avg-per-node}
	Consider problem \eqref{eq:sto-opt} under Assumption \ref{ass:basic}.  Let $\{\bar{\mb{x}}^{(t)}\}_{t\geq 0}$ defined in \eqref{eq:def-x-bar} be the node averages of local solutions from Algorithm \ref{alg:prsgd-momentum} with Option II.  If $\gamma$ and $I$ are chosen to satisfy $\frac{12L^2\gamma^2I^2}{(1-\beta)^2} < 1$, then for all $T\geq 1$, we have 
	\begin{align*}
	\sum_{t=0}^{T-1}\frac{1}{N}\sum_{i=1}^{N}\mbb{E}[\norm{\bar{\mb{x}}^{(t)}  - \mb{x}_i^{(t)}}^2] \leq  \frac{1}{1-\frac{12L^2\gamma^2I^2}{(1-\beta)^2}} \frac{2 \gamma^2 I \sigma^2}{(1-\beta)^2} T+ \frac{1}{1- \frac{12L^2\gamma^2I^2}{(1-\beta)^2}} \frac{6\gamma^2I^2 \kappa^2}{(1-\beta)^2} T
	\end{align*}
\end{Lem}
\begin{proof}
	Since $\bar{\mb{x}}^{(t)}  - \mb{x}_i^{(t)}=\mb{0}$ when $t$ is a multiple of $I$, our focus is to develop upper bounds for $\mbb{E}[\norm{\bar{\mb{x}}^{(t)}  - \mb{x}_i^{(t)}}^2]$ when $t$ is not a multiple of $I$. Consider $t \geq 1$ that is not a multiple of $I$. Note that Algorithm \ref{alg:prsgd-momentum} restarts ``SGD with momentum" every $I$ iterations (by resetting $\mb{u}_i^{(t)}  =  \hat{\mb{u}}\defeq \frac{1}{N} \sum_{j=1}^{N} \mb{u}_j^{(t)}$ and $\mb{x}_i^{(t)} = \hat{\mathbf{x}} \defeq \frac{1}{N} \sum_{j=1}^{N} \mathbf{x}_{j}^{(t)}$).  Let $t_{0} < t$ be the largest iteration index that is a multiple of $I$, i.e., $t_0~\text{mod}~I = 0$. Note that we must have $t - t_0 < I$ and $\mb{u}_i^{(t_0)} = \hat{\mb{u}}, \mb{x}_i^{(t_0)} = \hat{\mb{x}}$.  For any $\tau\in\{t_0+1, \ldots, t\}$, by recursively applying the first equation in \eqref{eq:prsgd-nesterov}, we have 
	\begin{align}
	\mb{u}_i^{(\tau)} = \beta^{\tau-t_0} \hat{\mb{u}} + \sum_{k=t_0}^{\tau-1} \beta^{\tau-1-k} \mb{g}_i^{(k)} 
	\end{align}
	By the second equation in \eqref{eq:prsgd-nesterov}, this further implies 
	\begin{align}
	\mb{v}_i^{(\tau)} = \beta^{\tau+1-t_0} \hat{\mb{u}} + \sum_{k=t_0}^{\tau-1} \beta^{\tau-k} \mb{g}_i^{(k)}  +  \mb{g}_i^{(\tau-1)}, \quad \forall  \tau\in\{t_0+1, \ldots, t\}.
	\label{eq:pf-nesterov-diff-avg-per-node-eq0}
	\end{align}
	For any  $\tau\in\{t_0+1, \ldots, t\}$, by the third equation in \eqref{eq:prsgd-nesterov}, we have
	\begin{align}
	\mb{x}_i^{(\tau)} -\mb{x}_i^{(\tau-1)} = -\gamma \mb{v}_i^{(\tau)}
	\end{align}	 
	Summing over $\tau\in\{t_0+1, \ldots, t\}$ and noting that $\mb{x}_i^{(t_0)} = \hat{\mb{x}}$ yields
	\begin{align}
	\mb{x}_i^{(t)} &= \hat{\mb{x}} -\gamma \sum_{\tau=t_0+1}^{t} \mb{v}_i^{(\tau)} \nonumber\\
	&\overset{(a)}{=} \hat{\mb{x}} -\gamma \big(\sum_{\tau=t_0+1}^{t} \beta^{\tau+1-t_0} \big)\hat{\mb{u}} - \gamma \sum_{\tau = t_0+1}^{t} \sum_{k=t_0}^{\tau-1} \beta^{\tau-k} \mb{g}_i^{(k)} - \gamma \sum_{\tau=t_0+1}^{t} \mb{g}_i^{(\tau-1)} \nonumber \\
	&= \hat{\mb{x}} -\gamma\big(\sum_{j=2}^{t+1-t_0} \beta^{j} \big) \hat{\mb{u}} - \gamma \sum_{k = t_0}^{t-1} \sum_{j=1}^{t-k} \beta^{j} \mb{g}_i^{(k)} -  \gamma \sum_{\tau=t_0}^{t-1} \mb{g}_i^{(\tau)}\nonumber\\
	&= \hat{\mb{x}} -\gamma\big(\sum_{j=2}^{t+1-t_0} \beta^{j} \big) \hat{\mb{u}} - \gamma \sum_{\tau = t_0}^{t-1} \frac{1-\beta^{t+1-\tau}}{1-\beta} \mb{g}_i^{(\tau)} \label{eq:pf-nesterov-diff-avg-per-node-eq1}
	\end{align}
	where (a) follows by substituting \eqref{eq:pf-nesterov-diff-avg-per-node-eq0}.

	Using a argument similar to the above, we can further show 
	\begin{align}
	\bar{\mb{x}}^{(t)} = \hat{\mb{x}} -\gamma \big(\sum_{j=2}^{t+1-t_0} \beta^{j} \big)\hat{\mb{u}} - \gamma \sum_{\tau=t_0}^{t-1}\frac{1-\beta^{t+1-\tau}}{1-\beta} \frac{1}{N}\sum_{i=1}^N\mb{g}_i^{(\tau)} \label{eq:pf-nesterov-diff-avg-per-node-eq2} 
	\end{align}
	Combining \eqref{eq:pf-nesterov-diff-avg-per-node-eq1} and \eqref{eq:pf-nesterov-diff-avg-per-node-eq2} yields
	\begin{align}
	&\frac{1}{N}\sum_{i=1}^{N}\mbb{E}[\norm{\bar{\mb{x}}^{(t)}  - \mb{x}_i^{(t)}}^2] \nonumber\\
	=& \gamma^2 \frac{1}{N} \sum_{i=1}^N \mbb{E}\Big[\Big\Vert \sum_{\tau = t_0}^{t-1} \big[ \mb{g}_i^{(\tau)} - \frac{1}{N} \sum_{j=1}^N \mb{g}_j^{(\tau)}\big]\frac{1-\beta^{t+1-\tau}}{1-\beta} \Big\Vert^2\Big] \nonumber\\
	\overset{(a)}{\leq}& 2\gamma^2\frac{1}{N}\sum_{i=1}^N\mbb{E}\Big[\Big\Vert \sum_{\tau = t_0}^{t-1} \big[ [\mb{g}_i^{(\tau)} - \nabla f_i(\mb{x}_i^{(\tau)}) ] - \frac{1}{N} \sum_{j=1}^N [\mb{g}_j^{(\tau)} - \nabla f_j(\mb{x}_j^{(\tau)})] \big]\frac{1-\beta^{t+1-\tau}}{1-\beta} \Big\Vert^2\Big]  \nonumber \\&+ 2\gamma^2 \frac{1}{N}\sum_{i=1}^N \mbb{E}\Big[\Big\Vert   \sum_{\tau=t_0}^{t-1} \big[\nabla f_i(\mb{x}_i^{(\tau)})- \frac{1}{N}\sum_{j=1}^{N} \nabla f_j(\mb{x}_j^{(\tau)}) \big]\frac{1-\beta^{t+1-\tau}}{1-\beta} \Big\Vert^2\Big] \label{eq:nesterov-diff-avg-per-node-eq3}
	\end{align}
	where (a) follows from the basic inequality $\norm{\mb{a}_1 + \mb{a}_2}^2 \leq 2 \norm{\mb{a}_1}^2 + 2\norm{\mb{a}_2}^2$.

	Now we develop the respective upper bounds of the two terms on the right side of \eqref{eq:nesterov-diff-avg-per-node-eq3}. We note that 
	\begin{align}
	&\frac{1}{N}\sum_{i=1}^N\mbb{E}\Big[\Big\Vert \sum_{\tau = t_0}^{t-1} \big[ [\mb{g}_i^{(\tau)} - \nabla f_i(\mb{x}_i^{(\tau)}) ] - \frac{1}{N} \sum_{j=1}^N [\mb{g}_j^{(\tau)} - \nabla f_j(\mb{x}_j^{(\tau)})] \big]\frac{1-\beta^{t+1-\tau}}{1-\beta} \Big\Vert^2\Big]  \nonumber\\
	\overset{(a)}{\leq}& \frac{1}{N}\sum_{i=1}^N\mbb{E}\Big[\Big\Vert \sum_{\tau = t_0}^{t-1} \big[ \mb{g}_i^{(\tau)} - \nabla f_i(\mb{x}_i^{(\tau)})  \big]\frac{1-\beta^{t+1-\tau}}{1-\beta} \Big\Vert^2\Big]  \nonumber\\
	\overset{(b)}{=} & \frac{1}{N}\sum_{i=1}^N \sum_{\tau = t_0}^{t-1} \mbb{E}\Big[\Big\Vert  \big[ \mb{g}_i^{(\tau)} - \nabla f_i(\mb{x}_i^{(\tau)})  \big]\frac{1-\beta^{t+1-\tau}}{1-\beta} \Big\Vert^2\Big]  \nonumber \\
	\overset{(c)}{\leq} &  \frac{I\sigma^2}{(1-\beta)^2}  \label{eq:nesterov-diff-avg-per-node-eq4}
	\end{align}
	where (a) follows from the inequality $\frac{1}{N}\sum_{i=1}^N \norm{\mb{a}_i - [\frac{1}{N}\sum_{j=1}^N \mb{a}_j] }^2 = \frac{1}{N}\sum_{i=1}^{N} \norm{\mb{a}_i}^2 - \norm{\frac{1}{N}\sum_{i=1}^{N} \mb{a}_i}^2 \leq  \frac{1}{N}\sum_{i=1}^{N} \norm{\mb{a}_i}^2$ with $\mb{a}_i = \sum_{\tau = t_0}^{t-1}[ \mb{g}_i^{(\tau)} - \nabla f_i(\mb{x}_i^{(\tau)})]\frac{1-\beta^{t+1-\tau}}{1-\beta}$; (b) follows because $\mbb{E}[ \mb{g}_i^{(\tau_2)} - \nabla f_i(\mb{x}_i^{(\tau_2)}) \big\vert  \mb{g}_i^{(\tau_1)} - \nabla f_i(\mb{x}_i^{(\tau_1)}) ] = \mb{0}$ for any $\tau_2 > \tau_1$; (c) follows because $\abs{\frac{1-\beta^{t+1-\tau}}{1-\beta} } \leq \frac{1}{1-\beta}$, $t-t_0 < I$ and $\mbb{E}[\Vert \mb{g}_i^{(\tau)} - \nabla f_i(\mb{x}_i^{(\tau)})\Vert^2] \leq \sigma^2$ (by Assumption \ref{ass:basic}).
	
	We further note that the
	
	\begin{align}
	& \frac{1}{N}\sum_{i=1}^N \mbb{E}\Big[\Big\Vert   \sum_{\tau=t_0}^{t-1} \big[\nabla f_i(\mb{x}_i^{(\tau)})- \frac{1}{N}\sum_{j=1}^{N} \nabla f_j(\mb{x}_j^{(\tau)}) \big]\frac{1-\beta^{t+1-\tau}}{1-\beta} \Big\Vert^2\Big]\nonumber\\
	\overset{(a)}{\leq} & \frac{1}{N}\sum_{i=1}^N (t-t_0)\frac{1}{(1-\beta)^2}\sum_{\tau=t_0}^{t-1} \mbb{E}\Big[\Big\Vert  \nabla f_i(\mb{x}_i^{(\tau)})- \frac{1}{N}\sum_{j=1}^{N} \nabla f_j(\mb{x}_j^{(\tau)})  \Big\Vert^2\Big]\nonumber\\
	\overset{(b)}{\leq}&\frac{I}{(1-\beta)^2}\sum_{\tau=t_0}^{t-1}\frac{1}{N}\sum_{i=1}^N  \mbb{E}\Big[\Big\Vert    \big[\nabla f_i(\mb{x}_i^{(\tau)})- \frac{1}{N}\sum_{j=1}^{N} \nabla f_j(\mb{x}_j^{(\tau)}) \big] \Big\Vert^2\Big] \nonumber\\
	\overset{(c)}{\leq}& \frac{6L^2I}{(1-\beta)^2}\sum_{\tau=t_0}^{t-1} \frac{1}{N}\sum_{i=1}^N \mbb{E}\big[\Vert \mb{x}_i^{(\tau)} - \bar{\mb{x}}^{(\tau)}\Vert^2\big] +  \frac{3I}{(1-\beta)^2}\sum_{\tau=t_0}^{t-1} \frac{1}{N}\sum_{i=1}^N\mbb{E}\big[\Vert \nabla f_i(\bar{\mb{x}}^{(\tau)}) - \nabla f(\bar{\mb{x}}^{(\tau)})\Vert^2 \big] \nonumber\\
	\overset{(d)}{\leq}& \frac{6L^2I}{(1-\beta)^2}\sum_{\tau=t_0}^{t-1} \frac{1}{N}\sum_{i=1}^N \mbb{E}\big[\Vert \mb{x}_i^{(\tau)} - \bar{\mb{x}}^{(\tau)}\Vert^2\big] +  \frac{3I^2 \kappa^2}{(1-\beta)^2} \label{eq:nesterov-diff-avg-per-node-eq5}
	\end{align} 
	where (a) follows by applying the basic inequality $\norm{\sum_{i=1}^{n} \mb{a}_i}^2 \leq n \sum_{i=1}^n \norm{\mb{a}_i}^2$ for any vectors $\mb{a}_i$ and any integer $n$ and noting that $\abs{\frac{1-\beta^{t+1-\tau}}{1-\beta}} \leq \frac{1}{1-\beta}$; (b) follows because $0< t-t_0 < I$; (c) follows from Lemma \ref{lm:from-smooth-f};  and (d) follows because $\frac{1}{N}\sum_{i=1}^N\mbb{E}\big[\Vert \nabla f_i(\bar{\mb{x}}^{(\tau)}) - \nabla f(\bar{\mb{x}}^{(\tau)})\Vert^2 \big] \leq \kappa^2$ by Assumption \ref{ass:basic} and $0< t-t_0 <I$.

	Substituting \eqref{eq:nesterov-diff-avg-per-node-eq4}-\eqref{eq:nesterov-diff-avg-per-node-eq5} into \eqref{eq:nesterov-diff-avg-per-node-eq3} yields
	\begin{align}
	\frac{1}{N}\sum_{i=1}^{N}\mbb{E}[\norm{\bar{\mb{x}}^{(t)}  - \mb{x}_i^{(t)}}^2] \leq   \frac{2\gamma^2 I \sigma^2}{(1-\beta)^2} + \frac{12L^2\gamma^2I}{(1-\beta)^2}\sum_{\tau=t_0}^{t-1} \frac{1}{N}\sum_{i=1}^N \mbb{E}\big[\Vert \mb{x}_i^{(\tau)} - \bar{\mb{x}}^{(\tau)}\Vert^2\big] +  \frac{6\gamma^2I^2 \kappa^2}{(1-\beta)^2} \label{eq:nesterov-diff-avg-per-node-eq6}
	\end{align}
	Recall that for each $t$, the index $t_0$ in the above equation is the largest integer such that $t_0 < t$,  $t_0~\text{mod}~I = 0$ and $t-t_0 <  I$. Summing \eqref{eq:nesterov-diff-avg-per-node-eq6} over $t\in \{0,1, \ldots, T\}$ that are not a multiple of $I$ and noting that $\bar{\mb{x}}^{(t)}  - \mb{x}_i^{(t)}=\mb{0}$ if $t$ is a multiple of $I$ yields
	\begin{align}
	\sum_{t=0}^{T-1}\frac{1}{N}\sum_{i=1}^{N}\mbb{E}[\norm{\bar{\mb{x}}^{(t)}  - \mb{x}_i^{(t)}}^2] \leq   \frac{2 \gamma^2 I \sigma^2}{(1-\beta)^2}T+ \frac{12L^2\gamma^2I^2}{(1-\beta)^2}\sum_{t=0}^{T-1}\frac{1}{N}\sum_{i=1}^{N}\mbb{E}[\norm{\bar{\mb{x}}^{(t)}  - \mb{x}_i^{(t)}}^2] +  \frac{6\gamma^2I^2 \kappa^2}{(1-\beta)^2}T\nonumber
	\end{align}
	Collecting common terms and dividing both sides by $1-\frac{12L^2\gamma^2I^2}{(1-\beta)^2}$ yields
	\begin{align}
	\sum_{t=0}^{T-1}\frac{1}{N}\sum_{i=1}^{N}\mbb{E}[\norm{\bar{\mb{x}}^{(t)}  - \mb{x}_i^{(t)}}^2] \leq  \frac{1}{1-\frac{12L^2\gamma^2I^2}{(1-\beta)^2}} \frac{2 \gamma^2 I \sigma^2}{(1-\beta)^2}T + \frac{1}{1- \frac{12L^2\gamma^2I^2}{(1-\beta)^2}} \frac{6\gamma^2I^2 \kappa^2}{(1-\beta)^2}T \label{eq:diff-avg-per-node-eq7}
	\end{align}
\end{proof}

\subsubsection{Main Proof of Theorem \ref{thm:nesterov-rate}} \label{sec:main-pf-thm-nesterov-rate}

It is easy to realize that Lemmas \ref{lm:y-diff}, \ref{lm:y-x-diff} and \ref{lm:nesterov-diff-avg-per-node} developed above are the respective counterparts for Lemmas \ref{lm:z-diff}, \ref{lm:z-x-diff} and \ref{lm:diff-avg-per-node}.  The main proof of Theorem \ref{thm:nesterov-rate} follows similar steps as the proof of Theorem \ref{thm:polyak-rate} in Section \ref{sec:pf-thm-polyak-rate} with the minor changes that Lemmas \ref{lm:z-diff}, \ref{lm:z-x-diff} and \ref{lm:diff-avg-per-node} should be replaced by Lemmas \ref{lm:y-diff}, \ref{lm:y-x-diff} and \ref{lm:nesterov-diff-avg-per-node}, respectively.

Below is the main proof of Theorem \ref{thm:nesterov-rate}.

Fix $t\geq 0$. By the smoothness of function $f(\cdot)$ (in Assumption \ref{ass:basic}), we have 
\begin{align}
\mbb{E}[f(\bar{\mb{y}}^{(t+1)})] \leq \mbb{E}[f(\bar{\mb{y}}^{(t)})] + \mbb{E}[\langle \nabla f(\bar{\mb{y}}^{(t)}), \bar{\mb{y}}^{(t+1)} - \bar{\mb{y}}^{(t)}\rangle] + \frac{L}{2} \mbb{E}[\norm{\bar{\mb{y}}^{(t+1)} - \bar{\mb{y}}^{(t)}}^2]  \label{eq:pf-thm-nesterov-rate-eq1}
\end{align}

By Lemma \ref{lm:y-diff}, we have 
\begin{align}
&\mbb{E}[\langle \nabla f(\bar{\mb{y}}^{(t)}), \bar{\mb{y}}^{(t+1)} - \bar{\mb{y}}^{(t)}\rangle] \nonumber\\
=& -\frac{\gamma}{1-\beta} \mbb{E}[\langle \nabla f(\bar{\mb{y}}^{(t)}), \frac{1}{N}\sum_{i=1}^N\mb{g}_i^{(t)}\rangle]\nonumber\\
\overset{(a)}{=}& -\frac{\gamma}{1-\beta} \mbb{E}[\langle \nabla f(\bar{\mb{y}}^{(t)}), \frac{1}{N}\sum_{i=1}^N\nabla f_i(\mb{x}_i^{(t)})\rangle] \nonumber \\
=& -\frac{\gamma}{1-\beta} \mbb{E}[\langle \nabla f(\bar{\mb{y}}^{(t)}) - \nabla f(\bar{\mb{x}}^{(t)}), \frac{1}{N}\sum_{i=1}^N\nabla f_i(\mb{x}_i^{(t)})\rangle] -\frac{\gamma}{1-\beta} \mbb{E}[\langle \nabla f(\bar{\mb{x}}^{(t)}), \frac{1}{N}\sum_{i=1}^N\nabla f_i(\mb{x}_i^{(t)})\rangle] \label{eq:pf-thm-nesterov-rate-eq2}
\end{align}
where (a) follows because $\bar{\mb{y}}^{(t)}$ and $\mb{x}_i^{(t)}$ are determined by $\bs{\xi}^{[t-1]} = [\bs{\xi}^{(0)}, \ldots, \bs{\xi}^{(t-1)}]$, which is independent of $\bs{\xi}^{(t)}$, and $\mbb{E}[\mb{g}_i^{(t)} | \bs{\xi}^{[t-1]}] = \mbb{E}[\mb{g}_i^{(t)}] = \nabla f_i (\mb{x}_i^{(t)})$.

We note that 
\begin{align}
&-\frac{\gamma}{1-\beta} \langle \nabla f(\bar{\mb{y}}^{(t)})-\nabla f(\bar{\mb{x}}^{(t)}), \frac{1}{N}\sum_{i=1}^N\nabla f_i(\mb{x}_i^{(t)})\rangle  \nonumber\\
\overset{(a)}{\leq}& \frac{1-\beta}{2\beta^3 L} \norm{\nabla f(\bar{\mb{y}}^{(t)})-\nabla f(\bar{\mb{x}}^{(t)})}^2 + \frac{L \gamma^2 \beta^3}{2(1-\beta)^3} \norm{\frac{1}{N}\sum_{i=1}^N\nabla f_i (\mb{x}_i^{(t)})}^2 \nonumber \\
\overset{(b)}{\leq}& \frac{(1-\beta)L}{2\beta^3} \norm{ \bar{\mb{y}}^{(t)}- \bar{\mb{x}}^{(t)}}^2 +\frac{\beta^3 L \gamma^2}{2(1-\beta)^3} \norm{\frac{1}{N}\sum_{i=1}^N\nabla f_i (\mb{x}_i^{(t)})}^2 \label{eq:pf-thm-nesterov-rate-eq3}
\end{align}
where (a) follows by applying the basic inequality $\langle \mb{a}, \mb{b}\rangle \leq \frac{1}{2}\norm{\mb{a}}^2 + \frac{1}{2}\norm{\mb{b}}^2$ with $\mb{a} =- \frac{\sqrt{1-\beta}}{\sqrt{L}\beta^{3/2}} [\nabla f(\bar{\mb{y}}^{(t)})-\nabla f(\bar{\mb{x}}^{(t)})]$ and $\mb{b} = \frac{\gamma \sqrt{L}\beta^{3/2}}{(1-\beta)^{3/2}} \frac{1}{N}\sum_{i=1}^N\nabla f_i(\mb{x}_i^{(t)})$; and (b) follows from the smoothness of function $f(\cdot)$.

Applying the basic identity $\langle \mb{a}, \mb{b}\rangle = \frac{1}{2}[\norm{\mb{a}}^2 + \norm{\mb{b}}^2 - \norm{\mb{a}-\mb{b}}^2]$ with $\mb{a} = \nabla f(\bar{\mb{x}}^{(t)})$ and $\mb{b} =\frac{1}{N}\sum_{i=1}^N\nabla f_i (\mb{x}_i^{(t)})$ yields
\begin{align}
&\langle \nabla f(\bar{\mb{x}}^{(t)}), \frac{1}{N}\sum_{i=1}^N\nabla f_i(\mb{x}_i^{(t)})\rangle \nonumber\\
=&  \frac{1}{2} \Big(\norm{\nabla f(\bar{\mb{x}}^{(t)})}^2 + \norm{\frac{1}{N}\sum_{i=1}^N\nabla f_i(\mb{x}_i^{(t)})}^2 - \norm{\nabla f(\bar{\mb{x}}^{(t)})-\frac{1}{N}\sum_{i=1}^N\nabla f_i(\mb{x}_i^{(t)})}^2 \Big) \nonumber \\
\overset{(a)}{\geq} & \frac{1}{2} \Big(\norm{\nabla f(\bar{\mb{x}}^{(t)})}^2 + \norm{\frac{1}{N}\sum_{i=1}^N\nabla f_i(\mb{x}_i^{(t)})}^2 -  L^2 \frac{1}{N} \sum_{i=1}^N \Vert \bar{\mb{x}}^{(t)} - \mb{x}_i^{(t)}\Vert^2 \Big) \label{eq:pf-thm-nesterov-rate-eq4}
\end{align}
where (a) follows because $\Vert\nabla f(\bar{\mb{x}}^{(t)}) - \frac{1}{N}\sum_{j=1}^N \nabla f_j (\mb{x}_j^{(t)})\Vert^2 = \Vert \frac{1}{N} \sum_{i=1}^N \nabla f_i(\bar{\mb{x}}^{(t)}) - \frac{1}{N}\sum_{i=1}^N \nabla f_i (\mb{x}_i^{(t)})\Vert^2 \leq \frac{1}{N} \sum_{i=1}^{N} \Vert \nabla f_i(\bar{\mb{x}}^{(t)})  - \nabla f_i (\mb{x}_i^{(t)})\Vert^2 \leq \frac{1}{N} \sum_{i=1}^N L^2 \Vert \bar{\mb{x}}^{(t)} - \mb{x}_i^{(t)}\Vert^2$, where the first inequality follows from the convexity of $\Vert \cdot \Vert^2$ and Jensen's inequality and the second inequality follows from the smoothness of each $f_i(\cdot)$.

Substituting \eqref{eq:pf-thm-nesterov-rate-eq3}-\eqref{eq:pf-thm-nesterov-rate-eq4} into \eqref{eq:pf-thm-nesterov-rate-eq2} yields
\begin{align}
&\mbb{E}[\langle \nabla f(\bar{\mb{y}}^{(t)}), \bar{\mb{y}}^{(t+1)} - \bar{\mb{y}}^{(t)}\rangle] \nonumber\\
\leq &\frac{(1-\beta)L}{2\beta^3} \mbb{E}[\norm{\bar{\mb{y}}^{(t)} - \bar{\mb{x}}^{(t)} }^2] + \big( \frac{L \gamma^2\beta^3}{2(1-\beta)^3} - \frac{\gamma}{2(1-\beta)}\big) \mbb{E}[\norm{\frac{1}{N}\sum_{i=1}^N\nabla f_i (\mb{x}_i^{(t)})}^2]  - \frac{\gamma}{2(1-\beta)}\mbb{E}[\norm{\nabla f(\bar{\mb{x}}^{(t)})}^2] \nonumber \\ &+ \frac{\gamma L^2}{2(1-\beta)} \frac{1}{N}\sum_{i=1}^N\mbb{E}[ \Vert \bar{\mb{x}}^{(t)} - \mb{x}_i^{(t)}\Vert^2]  \label{eq:pf-thm-nesterov-rate-eq5}
\end{align}

By Lemma \ref{lm:y-diff}, we have 
\begin{align}
\mbb{E}[\norm{\bar{\mb{z}}^{(t+1)} - \bar{\mb{z}}^{(t)}}^2] = \frac{\gamma^2}{(1-\beta)^2} \mbb{E}[\norm{\frac{1}{N}\sum_{i=1}^N \mb{g}_i^{(t)}}^2]  \label{eq:pf-thm-nesterov-rate-eq6}
\end{align}

Substituting \eqref{eq:pf-thm-nesterov-rate-eq5}-\eqref{eq:pf-thm-nesterov-rate-eq6} into \eqref{eq:pf-thm-nesterov-rate-eq1} yields

\begin{align}
\mbb{E}[f(\bar{\mb{y}}^{(t+1)})] \leq &\mbb{E}[f(\bar{\mb{y}}^{(t)})] + \frac{(1-\beta)L}{2\beta^3}  \mbb{E}[\norm{\bar{\mb{y}}^{(t)} - \bar{\mb{x}}^{(t)} }^2] + \big( \frac{L \gamma^2\beta^3}{2(1-\beta)^3} - \frac{\gamma}{2(1-\beta)}\big) \mbb{E}[\norm{\frac{1}{N}\sum_{i=1}^N\nabla f_i (\mb{x}_i^{(t)})}^2] \nonumber\\ &- \frac{\gamma}{2(1-\beta)}\mbb{E}[\norm{\nabla f(\bar{\mb{x}}^{(t)})}^2] + \frac{\gamma L^2}{2(1-\beta)} \frac{1}{N}\sum_{i=1}^N\mbb{E}[ \Vert \bar{\mb{x}}^{(t)} - \mb{x}_i^{(t)}\Vert^2] + \frac{\gamma^2L}{2(1-\beta)^2} \mbb{E}[\norm{\frac{1}{N}\sum_{i=1}^N \mb{g}_i^{(t)}}^2]
\end{align}
Dividing both sides by $\frac{\gamma}{2(1-\beta)}$ and rearranging terms yields
\begin{align}
\mbb{E}[\norm{\nabla f(\bar{\mb{x}}^{(t)})}^2] \leq &\frac{2(1-\beta)}{\gamma} \big(\mbb{E}[f(\bar{\mb{y}}^{(t)})] -\mbb{E}[f(\bar{\mb{y}}^{(t+1)})]  \big) - (1- \frac{L \gamma\beta^3}{(1-\beta)^2})\mbb{E}[\norm{\frac{1}{N}\sum_{i=1}^N\nabla f_i (\mb{x}_i^{(t)})}^2]  \nonumber\\
&+ \frac{(1-\beta)^2 L}{ \gamma\beta^3} \mbb{E}[\norm{ \bar{\mb{y}}^{(t)} - \bar{\mb{x}}^{(t)} }^2] +  \frac{\gamma L}{(1-\beta)} \mbb{E}[\norm{\frac{1}{N}\sum_{i=1}^N \mb{g}_i^{(t)}}^2] + L^2 \frac{1}{N}\sum_{i=1}^N\mbb{E}[ \Vert \bar{\mb{x}}^{(t)} - \mb{x}_i^{(t)}\Vert^2]
\end{align}
Summing over $t\in\{0,1,\ldots, T-1\}$ 
\begin{align}
&\sum_{t=0}^{T-1}\mbb{E}[\norm{\nabla f(\bar{\mb{x}}^{(t)})}^2] \nonumber \\
\leq &\frac{2(1-\beta)}{\gamma } \big(\mbb{E}[f(\bar{\mb{y}}^{(0)})] -\mbb{E}[f(\bar{\mb{y}}^{(T)})]  \big) - \Big(1- \frac{L\gamma\beta^3}{(1-\beta)^2}\Big) \sum_{t=0}^{T-1}\mbb{E}[\norm{\frac{1}{N}\sum_{i=1}^N\nabla f_i (\mb{x}_i^{(t)})}^2] \nonumber\\
&+ \frac{(1-\beta)^2 L}{\gamma\beta^3} \sum_{t=0}^{T-1}\mbb{E}[\norm{ \bar{\mb{y}}^{(t)} - \bar{\mb{x}}^{(t)} }^2] + \frac{L\gamma}{(1-\beta)} \sum_{t=0}^{T-1}\mbb{E}[\norm{\frac{1}{N}\sum_{i=1}^N \mb{g}_i^{(t)}}^2] + L^2 \sum_{t=0}^{T-1} \frac{1}{N}\sum_{i=1}^N\mbb{E}[ \Vert \bar{\mb{x}}^{(t)} - \mb{x}_i^{(t)}\Vert^2] \nonumber \\
\overset{(a)}{\leq} &\frac{2(1-\beta)}{\gamma } \big(\mbb{E}[f(\bar{\mb{y}}^{(0)})] -\mbb{E}[f(\bar{\mb{y}}^{(T)})]  \big) - \Big(1- \frac{L\gamma\beta^3}{(1-\beta)^2}\Big) \sum_{t=0}^{T-1}\mbb{E}[\norm{\frac{1}{N}\sum_{i=1}^N\nabla f_i (\mb{x}_i^{(t)})}^2]  \nonumber \\ & + \frac{L\gamma}{(1-\beta)^2} \sum_{t=0}^{T-1}\mbb{E}[\norm{\frac{1}{N}\sum_{i=1}^N \mb{g}_i^{(t)}}^2] + L^2 \sum_{t=0}^{T-1} \frac{1}{N}\sum_{i=1}^N\mbb{E}[ \Vert \bar{\mb{x}}^{(t)} - \mb{x}_i^{(t)}\Vert^2] \nonumber\\ 
\overset{(b)}{\leq} &\frac{2(1-\beta)}{\gamma } \big(\mbb{E}[f(\bar{\mb{y}}^{(0)})] -\mbb{E}[f(\bar{\mb{y}}^{(T)})]  \big) -  \Big(1- \frac{ (1+\beta^3)L\gamma}{(1-\beta)^2}\Big) \sum_{t=0}^{T-1}\mbb{E}[\norm{\frac{1}{N}\sum_{i=1}^N\nabla f_i (\mb{x}_i^{(t)})}^2]  +  \frac{L\gamma }{(1-\beta)^2}  \frac{\sigma^2}{N}T \nonumber \\ &+ \frac{1}{1-\frac{12L^2\gamma^2I^2}{(1-\beta)^2}} \frac{4 L^2\gamma^2 I \sigma^2}{(1-\beta)^2} T+ \frac{1}{1- \frac{12L^2\gamma^2I^2}{(1-\beta)^2}} \frac{6L^2\gamma^2I^2 \kappa^2}{(1-\beta)^2}T  \nonumber\\
\overset{(c)}{\leq} &\frac{2(1-\beta)}{\gamma } \big(f(\bar{\mb{x}}^{(0)}) - f^\ast \big) +  \frac{L\gamma }{(1-\beta)^2}  \frac{\sigma^2}{N}T+ \frac{1}{1-\frac{12L^2\gamma^2I^2}{(1-\beta)^2}} \frac{4 L^2\gamma^2 I \sigma^2}{(1-\beta)^2} T+ \frac{1}{1- \frac{12L^2\gamma^2I^2}{(1-\beta)^2}} \frac{6L^2\gamma^2I^2 \kappa^2}{(1-\beta)^2}T
\end{align}
where (a) follows by using Lemma \ref{lm:y-x-diff} and $\frac{\beta L\gamma}{(1-\beta)^2} + \frac{L\gamma}{1-\beta} = \frac{L\gamma}{(1-\beta)^2}$; (b) follows by applying Lemma \ref{lm:nesterov-diff-avg-per-node} and by noting that $\mbb{E}[\norm{\frac{1}{N}\sum_{i=1}^N \mb{g}_i^{(t)}}^2] \leq  \frac{1}{N}\sigma^{2}  +  \mathbb{E} [ \Vert \frac{1}{N} \sum_{i=1}^{N} \nabla f_{i}(\mathbf{x}_{i}^{(t)})\Vert^{2}]$ by Lemma \ref{lm:expected-gradient} ; and (c) follows because $\gamma $ is chosen to ensure $1- \frac{ (1+\beta^3)L\gamma}{(1-\beta)^2} \geq 0$, $\bar{\mb{y}}^{(0)} = \bar{\mb{x}}^{(0)}$ by the definition in \eqref{eq:def-y}, and $f^\ast$ is the minimum value of problem \eqref{eq:sto-opt}.

Dividing both sides by $T$ yields
\begin{align}
&\frac{1}{T} \sum_{t=0}^{T-1}\mbb{E}[\norm{\nabla f(\bar{\mb{x}}^{(t)})}^2] \nonumber \\
\leq&\frac{2(1-\beta)}{\gamma T} \big(f(\bar{\mb{x}}^{(0)}) - f^\ast \big) +  \frac{L\gamma }{(1-\beta)^2}  \frac{\sigma^2}{N}+ \frac{1}{1-\frac{12L^2\gamma^2I^2}{(1-\beta)^2}} \frac{4 L^2\gamma^2 I \sigma^2}{(1-\beta)^2} + \frac{1}{1- \frac{12L^2\gamma^2I^2}{(1-\beta)^2}} \frac{6L^2\gamma^2I^2 \kappa^2}{(1-\beta)^2} \nonumber\\
\overset{(a)}{\leq}&\frac{2(1-\beta)}{\gamma T} \big(f(\bar{\mb{x}}^{(0)}) - f^\ast \big) +  \frac{L\gamma }{(1-\beta)^2}  \frac{\sigma^2}{N}+  \frac{3 L^2\gamma^2 I \sigma^2}{(1-\beta)^2} +  \frac{9L^2\gamma^2I^2 \kappa^2}{(1-\beta)^2} \nonumber\\
=& O(\frac{1}{\gamma T})  + O(\frac{\gamma}{N} \sigma^2) + O(\gamma^2 I \sigma^2)+ O(\gamma^2 I^2 \kappa^2)
\end{align}
where (a) follows because $I \leq \frac{1-\beta}{6L\gamma}$ is chosen to ensure $\frac{1}{1-\frac{12L^2\gamma^2I^2}{(1-\beta)^2}} \leq \frac{3}{2}$.

\subsection{Proof of Theorem \ref{thm:decentralized-rate}} \label{sec:pf-thm-decentralized-rate}
This section provides the complete proof for Algorithm \ref{alg:prsgd-decentralized} with Option I (Polyak's momentum).  An extension to Algorithm \ref{alg:prsgd-decentralized} with Option II can be similarly done as our extension in Section \ref{sec:pf-thm-nesterov-rate} for Algorithm \ref{alg:prsgd-momentum} from Option I to Option II.

Let $\bar{\mb{u}}^{(t)}$ and $\bar{\mb{x}}^{(t)}$ (using the forms in \eqref{eq:def-u-bar} and \eqref{eq:def-x-bar}, respectively) be the $N$ node averages of local variables $\mb{u}_i^{(t)}$ and $\mb{x}_i^{(t)}$ from Algorithm \ref{alg:prsgd-decentralized} with Option I.  It is not difficult to show that $\bar{\mb{u}}^{(t)}$ and $\bar{\mb{x}}^{(t)}$ satisfiy the following fact.

\begin{Fact}\label{fact:decentralized-bar-dynamic}
Let $\bar{\mb{u}}^{(t)} \defeq \frac{1}{N} \sum_{i=1}^{N} \mb{u}_{i}^{(t)}$ and $\bar{\mb{x}}^{(t)} \defeq \frac{1}{N}\sum_{i=1}^N \mb{x}_{i}^{(t)}$ be node averages of local variables from Algorithm \ref{alg:prsgd-decentralized} with Option I. For all $ t \geq 1$, we have 
\begin{align}
\begin{cases}
\bar{\mb{u}}^{(t)} &= \beta \bar{\mb{u}}^{(t-1)} + \frac{1}{N}\sum_{i=1}^N \mb{g}_i^{(t-1)} \\
\bar{\mb{x}}^{(t)}  &= \bar{\mb{x}}^{(t-1)}  - \gamma  \bar{\mb{u}}^{(t)}
\end{cases} \label{eq:bar-polyak-decentralized}
\end{align}
\end{Fact}
\begin{proof}
By the definition of $\bar{\mb{u}}^{(t)}$, we have 
\begin{align}
\bar{\mb{u}}^{(t)} =& \frac{1}{N} \sum_{i=1}^{N} \mb{u}_{i}^{(t)} \nonumber \\
\overset{(a)}{=}& \frac{1}{N} \sum_{i=1}^N \sum_{j=1}^N \tilde{\mb{u}}_{j}^{(t)} W_{ji} \nonumber\\
=&  \frac{1}{N} \sum_{j=1}^N \tilde{\mb{u}}_{j}^{(t)} [\sum_{j=1}^N W_{ji}] \nonumber\\
\overset{(b)}{=} &\frac{1}{N} \sum_{i=1}^N \tilde{\mb{u}}_{i}^{(t)} \label{eq:pf-fact-decentralized-bar-eq1}\\
\overset{(c)}{=} & \beta \frac{1}{N}\sum_{i=1}^{N}\mb{u}_{i}^{(t-1)} + \frac{1}{N} \sum_{i=1}^N \mb{g}_i^{(t-1)} \nonumber\\
\overset{(d)}{=}& \beta \bar{\mb{u}}^{(t-1)} + \frac{1}{N} \sum_{i=1}^N \mb{g}_i^{(t-1)} \nonumber
\end{align}
where (a) follows by substituting the first equation in \eqref{eq:prsgd-avg-decentralized}; (b) follows by recalling that $\sum_{j=1}^N W_{ji}=1$ for doubly stochastic matrix $\mb{W}$; (c) follows by substituting the first equation in \eqref{eq:decentralized-polyak}; and (d) follows from the definition of $\bar{\mb{u}}^{(t)}$.

Using a similar argument, we can show
\begin{align*}
\bar{\mb{x}}^{(t)} \overset{(a)}{=}& \frac{1}{N} \sum_{i=1}^{N} \mb{x}_{i}^{(t)} \nonumber \\
\overset{(b)}{=}& \frac{1}{N} \sum_{i=1}^N \sum_{j=1}^N \tilde{\mb{x}}_{j}^{(t)} W_{ji} \nonumber\\
=& \frac{1}{N} \sum_{j=1}^N  \tilde{\mb{x}}_{j}^{(t)} [\sum_{i=1}^NW_{ji}] \nonumber\\
\overset{(c)}{=}& \frac{1}{N}\sum_{i=1}^{N}\tilde{\mb{x}}_{i}^{(t)} \nonumber\\
\overset{(d)}{=} &  \frac{1}{N}\sum_{i=1}^{N} [\mb{x}_{i}^{(t-1)} - \gamma \tilde{\mb{u}}_i^{(t-1)} ]\nonumber\\
=& \bar{\mb{x}}^{(t-1)} -\gamma  \bar{\mb{u}}^{(t-1)} \nonumber
\end{align*}
where (a) follows from the definition of $\bar{\mb{x}}^{(t)}$; (b) follows by substituting the second equation in \eqref{eq:prsgd-avg-decentralized}; (c) follows byrecalling that $\sum_{j=1}^N W_{ji}=1$ for doubly stochastic matrix $\mb{W}$; and (d) follows by using the definition of $\bar{\mb{x}}^{(t)}$ and equation \eqref{eq:pf-fact-decentralized-bar-eq1}.

\end{proof}

It is remarkable that Fact \ref{fact:decentralized-bar-dynamic} implies that even if we decentralized local averaging in Algorithm \ref{alg:prsgd-decentralized}, the yielded global averages $\{\bar{\mb{u}}^{(t)}\}_{t\geq 0}$ and $\{\bar{\mb{x}}^{(t)}\}_{t\geq 0}$ follow the same dynamics as the global averages in Algorithm \ref{alg:prsgd-momentum}. 

Let $\bar{\mb{x}}^{(t)}= \frac{1}{N}\sum_{i=1}^N \mb{x}_{i}^{(t)}$ be node averages of local variables $\mb{x}_{i}^{(t)}$ from Algorithm \ref{alg:prsgd-decentralized} with Option I.  We again define the auxiliary sequence $\{\bar{\mb{z}}^{(t)}\}_{t\geq 0}$ via 
\begin{align}
\bar{\mb{z}}^{(t)} \defeq \left \{  \begin{array}{ll} \bar{\mb{x}}^{(t)}, & \quad t=0 \\
\frac{1}{1-\beta} \bar{\mb{x}}^{(t)} - \frac{\beta}{1-\beta}\bar{\mb{x}}^{(t-1)}, &\quad t\geq 1 \end{array} \right. \label{eq:def-z-decentralized}
\end{align}

Since \eqref{eq:bar-polyak-decentralized} and \eqref{eq:def-z-decentralized} are respectively identical to \eqref{eq:bar-polyak} and \eqref{eq:def-z} for Algorithm \ref{alg:prsgd-momentum} with Option I, the following two lemmas can be proven using exactly the same steps in the proofs for Lemmas \ref{lm:z-diff} and \ref{lm:z-x-diff}.

\begin{Lem}\label{lm:z-diff-decentralized}
Consider the sequence $\{\bar{\mb{z}}^{(t)}\}_{t\geq 0}$ defined in \eqref{eq:def-z-decentralized}. Algorithm \ref{alg:prsgd-decentralized} with Option I ensures that for all $t\geq 0$, we have 
\begin{align}
\bar{\mb{z}}^{(t+1)} - \bar{\mb{z}}^{(t)} = -\frac{\gamma}{1-\beta}\frac{1}{N}\sum_{i=1}^N \mb{g}_i^{(t)}.
\end{align}
\end{Lem}

\begin{Lem} \label{lm:z-x-diff-decentralized}
	Let $\{\bar{\mb{x}}^{(t)}\}_{t\geq 0}$ defined in \eqref{eq:def-x-bar} be the global node averages of local solutions from Algorithm \ref{alg:prsgd-decentralized} with Option I.  Let $\{\bar{\mb{z}}^{(t)}\}_{t\geq 0}$ be defined in \eqref{eq:def-z-decentralized}. For all $T\geq 1$, Algorithm \ref{alg:prsgd-decentralized} with Option I ensures that 
	\begin{align}
		\sum_{t=0}^{T-1} \norm{ \bar{\mb{z}}^{(t)} - \bar{\mb{x}}^{(t)} }^2 \leq \frac{\gamma^2 \beta^2}{(1-\beta)^4}  \sum_{t=0}^{T-1} \Big\Vert \Big[\frac{1}{N} \sum_{i=1}^N \mb{g}_i^{(t)} \Big] \Big\Vert^2
	\end{align}
\end{Lem}

To extend the proof in Section \ref{sec:main-pf-thm-polyak-rate} (for Theorem \ref{alg:prsgd-momentum}) to analyze the convergence rate for Algorithm \ref{alg:prsgd-decentralized}, the outstanding challenge is to provide a tight upper bound for quantity $\sum_{t=0}^{T-1} \frac{1}{N} \sum_{i=1}^N \mbb{E}[\norm{\bar{\mb{x}}^{(t)} - \mb{x}_i^{(t)}}^2]$, i.e., to develop an counterpart of Lemma \ref{lm:diff-avg-per-node}.

For each $t\geq 0$, let 
\begin{align}
\begin{cases}
\mb{G}^{(t)} \defeq \big[\mb{g}_1^{(t)}, \mb{g}_2^{(t)}, \ldots, \mb{g}_N^{(t)}\big] \\
\mb{U}^{(t)} \defeq \big[\mb{u}_1^{(t)}, \mb{u}_2^{(t)}, \ldots, \mb{u}_N^{(t)}\big]\\
\widetilde{\mb{U}}^{(t)} \defeq \big[\tilde{\mb{u}}_1^{(t)}, \tilde{\mb{u}}_2^{(t)}, \ldots, \tilde{\mb{u}}_N^{(t)}\big]\\
\mb{X}^{(t)} \defeq \big[\mb{x}_1^{(t)}, \mb{x}_2^{(t)}, \ldots, \mb{x}_N^{(t)}\big]   
\end{cases}\label{eq:def-matrix-var}
\end{align}
be  $m \times N$ matrices that concatenate local variables $\mb{g}_i^{(t)}$, $\mb{u}_i^{(t)}$, $\tilde{\mb{u}}_i^{(t)}$ and $\mb{x}_i^{(t)}$ for all $N$ nodes.  Recall that the Frobenius norm for any $m\times N$ matrix $\mb{Z}$ satisfies $\norm{\mb{Z}}_F^2 = \sum_{i=1}^N \norm{\mb{z}_i}^2$ where $\norm{\mb{z}_i}^2$ is the Euclidean norm of the $i$-th column of matrix $\mb{Z}$. It can be easily verified that  $\sum_{i=1}^N \norm{ \mb{x}_i^{(t)} - \bar{\mb{x}}^{(t)} }^2 = \norm{\mb{X}^{(t)}(\mb{I} - \mb{Q})}_F^2 $ where 
\begin{align}
\mb{Q} \defeq \frac{1}{N}\mb{1}_N \mb{1}_N\tran \label{eq:def-Q-matrix}
\end{align}
is an $N\times N$ matrix where all entries are $1/N$ and $\mb{I}$ is the identity matrix for which the dimensions are obvious in its context.  

Due to the asymmetry in computing local averages for different nodes in \eqref{eq:prsgd-avg-decentralized},  it is easier to provide the desired counterpart of Lemma \ref{lm:diff-avg-per-node} by studying the equivalent matrix form $\norm{\mb{X}^{(t)}(\mb{I}- \mb{Q})}_F^2$ .  The technique of introducing matrix forms to analyze the convergence rate has been previously used in \cite{Lian17NIPS, Tang18ArXiv, WangJoshi18ArXiv} for the analysis of decentralized SGD without momentum.

The following useful facts are simple in matrix analysis \cite{book_MatrixAnalysis}. Similar facts are used in \cite{Lian17NIPS, WangJoshi18ArXiv}.

\begin{Fact}\label{fact:frobenius-norm-inequality}
Let $\mb{A}_i, i \in\{1,2,\ldots, n\}$ be $n$ arbitrary real square matrices. It follows that $\norm{\sum_{i=1}^n \mb{A}_i}_F^2 \leq \sum_{i=1}^{n} \sum_{j=1}^n \norm{\mb{A}_i}_F  \norm{\mb{A}_i}_F$.
\end{Fact}
\begin{proof}
Recall by the definition of Frobenius norm, we have $\norm{\mb{A}}_F^2 = \text{tr}(\mb{A}\tran \mb{A})$ where $\text{tr}(\cdot)$ denote the trace operator of a square matrix. It follows that 
\begin{align*}
\norm{\sum_{i=1}^n \mb{A}_i}_F^2 = ~& \text{tr}((\sum_{i=1}^n \mb{A}_i)\tran (\sum_{j=1}^n \mb{A}_j)) \\
=~& \sum_{i=1}^n \sum_{j=1}^n \text{tr}(\mb{A}_i\tran \mb{A}_j) \\
\overset{(a)}{\leq}& \sum_{i=1}^n \sum_{j=1}^n  \norm{\mb{A}_i}_F \norm{\mb{A}_j}_F
\end{align*}
where (a) follows from the simple fact that $\abs{\text{tr}(\mb{A}\tran \mb{B})} \leq  \norm{\mb{A}}_F \norm{\mb{B}}_F$ for any two square matrices.
\end{proof}

\begin{Fact}\label{fact:matrix-identity}
Let $\mb{Q}$ be defined in \eqref{eq:def-Q-matrix}. For any symmetric doubly stochastic matrix $\mb{W}$ under Assumption \ref{ass:mix-matrix}, we have 
\begin{enumerate}
\item $\mb{Q}\mb{W} = \mb{W}\mb{Q}$
\item $(\mb{I} - \mb{Q}) \mb{W} = \mb{W}(\mb{I} - \mb{Q})$
\item For any integer $k\geq 1$, we have $\norm{(\mb{I} - \mb{Q})\mb{W}^{k}} \leq \rho^{k/2}$ where $\rho <1$ is the constant in Assumption \ref{ass:mix-matrix} and $\norm{\cdot}$ denotes the spectrum norm for a matrix.
\end{enumerate}
\end{Fact}
\begin{proof}
Recall that $\mb{W}$ is a symmetric matrix.  Let $\mb{W} = \mb{P}\tran\bs{\Lambda}\mb{P}$ be the eigenvalue decomposition for $\mb{W}$ where $\bs{\Lambda} = \text{Diag}\{\lambda_1(\mb{W}), \ldots, \lambda_N(\mb{W})\}$ and $\mb{P}$ is a unitary matrix where each column is an eigenvector of $\mb{W}$.  Since $\frac{1}{\sqrt{N}}\mb{1}_N$ is the eigenvector corresponding to eigenvalue $\lambda_1(\mb{W})=1$ for matrix $\mb{W}$, thus the first column of $\mb{P}$ is $\frac{1}{
\sqrt{N}}\mb{1}_N$. By \eqref{eq:def-Q-matrix}, we further have an eigenvalue decomposition of $\mb{Q}$ given by $\mb{Q} = \mb{P}\bs{\Gamma}\mb{P}\tran$ where $\bs{\Gamma} = \text{Diag}\{1, 0, \ldots, 0\}$. Thus, we have 

\begin{enumerate}
\item $\mb{Q}\mb{W}  = \mb{P}\tran\bs{\Gamma}\mb{P}  \mb{P}\tran\bs{\Lambda}\mb{P} = \mb{P}\tran\bs{\Gamma}\bs{\Lambda}\mb{P} = \mb{P}\tran\bs{\Lambda}\bs{\Gamma}\mb{P} = \mb{P}\tran\bs{\Lambda}\mb{P}\mb{P}\tran\bs{\Gamma}\mb{P} = \mb{W}\mb{Q}$
\item $(\mb{I} - \mb{Q}) \mb{W} = \mb{W} - \mb{Q}\mb{W} \overset{(a)}{=} \mb{W} - \mb{W}\mb{Q} =\mb{W}(\mb{I} - \mb{Q})$ where (a) follows from part (1) of this fact.
\item Note that $(\mb{I} - \mb{Q})\mb{W}^{k} =  \mb{P}\tran(\mb{I}-\bs{\Gamma})\mb{P} (\mb{P}\tran\bs{\Lambda}\mb{P})^k= \mb{P}\tran (\mb{I}-\bs{\Gamma})\bs{\Lambda}^k \mb{P}$ where $(\mb{I}-\bs{\Gamma}) \bs{\Lambda}^k= \text{Diag}\{0, (\lambda_2(\mb{W}))^k, \ldots, (\lambda_N(\mb{W}))^k\}$. By Assumption \ref{ass:mix-matrix}, we have $\max\{\abs{(\lambda_2(\mb{W}))^k}, \abs{(\lambda_N(\mb{W}))^k}\} \leq (\sqrt{\rho})^k$. 
\end{enumerate}
\end{proof}

\begin{Lem} \label{lm:x-bar-x-diff-decentralized}
	Consider problem \eqref{eq:sto-opt} under Assumptions \ref{ass:basic} and \ref{ass:mix-matrix}.  Let $\{\bar{\mb{x}}^{(t)}\}_{t\geq 0}$ defined in \eqref{eq:def-x-bar} be the node averages of local solutions from Algorithm \ref{alg:prsgd-decentralized} with Option I.  If $\gamma$ is chosen to satisfy $\frac{16L^2\gamma^2}{(1-\beta)^2(1-\sqrt{\rho})^2} < 1$, then for all $T\geq 1$, we have 
	\begin{align*}
	&\sum_{t=0}^{T-1}\frac{1}{N}\sum_{i=1}^{N}\mbb{E}[\norm{\bar{\mb{x}}^{(t)}  - \mb{x}_i^{(t)}}^2] \\
	\leq &\frac{1}{1- \frac{16L^2\gamma^2}{(1-\beta)^2(1-\sqrt{\rho})^2}}\frac{2\gamma^2\sigma^2}{(1-\beta)^2(1-\rho)}T + \frac{1}{1- \frac{16L^2\gamma^2}{(1-\beta)^2(1-\sqrt{\rho})^2}} \frac{8\gamma^2}{(1-\beta)^2(1-\sqrt{\rho})^2}\sum_{t=0}^{T-1}\mbb{E}[\norm{\frac{1}{N} \sum_{i=1}^N \nabla f_i (\mb{x}_i^{(t)})}^2] \nonumber\\ &+ \frac{1}{1- \frac{16L^2\gamma^2}{(1-\beta)^2(1-\sqrt{\rho})^2}}  \frac{2\gamma^2\kappa^2}{(1-\beta)^2(1-\sqrt{\rho})^2}T
	\end{align*}
\end{Lem}

\begin{proof}
Recall the definitions of $\mb{U}^{(t)}$ and $\widetilde{\mb{U}}^{(t)}$ in \eqref{eq:def-matrix-var}, for all $t\geq 1$, we have 
\begin{align}
\mb{U}^{(t)} \overset{(a)}{=}~& \widetilde{\mb{U}}^{(t)} \mb{W}  \nonumber\\
 \overset{(b)}{=}~& [\beta \mb{U}^{(t-1)} + \mb{G}^{(t-1)}]\mb{W} \label{eq:pf-lm-decentralized-node-diff-eq1}
\end{align}
where (a) follows from the first equation in \eqref{eq:prsgd-avg-decentralized} and (b) follows from the first equation in \eqref{eq:decentralized-polyak}.

Recursively applying the above equation for $t$ times yields
\begin{align}
\mb{U}^{(t)} =~& \beta^t \mb{U}^{(0)} \mb{W}^t + \sum_{\tau=0}^{t-1}  \mb{G}^{(\tau)} (\beta\mb{W})^{t-1-\tau}  \nonumber \\
\overset{(a)}{=}~& \sum_{\tau=0}^{t-1}  \mb{G}^{(\tau)} (\beta\mb{W})^{t-1-\tau}  \label{eq:pf-lm-decentralized-node-diff-eq2}
\end{align}
where (a) follows from $\mb{U}^{(0)}  = \mb{0}$.

Recall the definitions of $\mb{X}^{(t)}$ and  $\widetilde{\mb{U}}^{(t)}$ in \eqref{eq:def-matrix-var}, for all $t\geq 1$, we have
\begin{align}
\mb{X}^{(t)} \overset{(a)}{=}~& [\mb{X}^{(t-1)} - \gamma \widetilde{\mb{U}}^{(t)}] \mb{W}  \nonumber\\
\overset{(b)}{=}~& [\mb{X}^{(t-1)} - \gamma [\beta \mb{U}^{(t-1)}+\mb{G}^{(t-1)}]] \mb{W} \nonumber\\
=~&  \mb{X}^{(t-1)}\mb{W} - \gamma [\beta \mb{U}^{(t-1)}+\mb{G}^{(t-1)}]\mb{W} \nonumber\\
 \overset{(c)}{=}~& \mb{X}^{(t-1)}\mb{W} - \gamma \mb{U}^{(t)}
\end{align}
where (a) follows from the second equation in \eqref{eq:decentralized-polyak}; (b) follows from the first equation in \eqref{eq:decentralized-polyak}; and (c) follows from \eqref{eq:pf-lm-decentralized-node-diff-eq1}. 

Multiplying both sides by $\mb{I} - \mb{Q}$ with $\mb{Q}$ defined in \eqref{eq:def-Q-matrix}  yields
\begin{align}
\mb{X}^{(t)}(\mb{I} - \mb{Q})  =~& \mb{X}^{(t-1)}\mb{W}(\mb{I} - \mb{Q}) - \gamma \mb{U}^{(t)}(\mb{I} - \mb{Q})  \nonumber \\
\overset{(a)}{=}~& \mb{X}^{(t-1)}(\mb{I} - \mb{Q}) \mb{W}- \gamma \mb{U}^{(t)}(\mb{I} - \mb{Q}) \nonumber
\end{align}
where (a) follows because $\mb{W}(\mb{I} - \mb{Q}) = (\mb{I} - \mb{Q}) \mb{W}$ by Fact \ref{fact:matrix-identity}.

For all $t\geq 1$, recursively applying the above equation for $t$ times (using $\mb{W}(\mb{I} - \mb{Q}) = (\mb{I} - \mb{Q})\mb{W}$ when needed) yields
\begin{align}
\mb{X}^{(t)}(\mb{I} - \mb{Q}) =~& \mb{X}^{(0)}(\mb{I} - \mb{Q}) \mb{W}^t -\gamma \sum_{\tau=1}^{t} \mb{U}^{(\tau)}(\mb{I} - \mb{Q}) \mb{W}^{t-\tau} \nonumber\\
\overset{(a)}{=}~&-\gamma \sum_{\tau=1}^{t} \mb{U}^{(\tau)}(\mb{I} - \mb{Q}) \mb{W}^{t-\tau} \nonumber\\
\overset{(b)}{=}~&-\gamma  \sum_{\tau=1}^{t} \sum_{j=0}^{\tau-1}  \mb{G}^{(j)} (\beta\mb{W})^{\tau-1-j} (\mb{I} - \mb{Q}) \mb{W}^{t-\tau} \nonumber\\
\overset{(c)}{=}~& -\gamma  \sum_{\tau=1}^{t} \sum_{j=0}^{\tau-1}  \mb{G}^{(j)} \beta^{\tau-1-j}\mb{W}^{t-1-j} (\mb{I} - \mb{Q}) \nonumber\\
=~& -\gamma \sum_{k=0}^{t-1} \mb{G}^{(k)} [\sum_{j=k+1}^{t} \beta^{j-1-k}] \mb{W}^{t-1-k} (\mb{I}-\mb{Q}) \nonumber\\
=~& -\gamma \sum_{\tau=0}^{t-1} \frac{1-\beta^{t-\tau}}{1-\beta}\mb{G}^{(\tau)} (\mb{I}-\mb{Q}) \mb{W}^{t-1-\tau}  \label{eq:pf-lm-decentralized-node-diff-eq3}
\end{align}
where (a) follows because $\mb{X}^{(0)}(\mb{I} - \mb{Q})  = \mb{0}$ due to all columns of $\mb{X}^{(0)}$ are identical; (b) follows by substituting \eqref{eq:pf-lm-decentralized-node-diff-eq2}; (c) follows because $\mb{W}(\mb{I} - \mb{Q})= (\mb{I} - \mb{Q}) \mb{W}$ by Fact \ref{fact:matrix-identity}.

For all $t\geq 0$, denote
\begin{align}
\mb{H}^{(t)} \defeq \big[\nabla f_1(\mb{x}_1^{(t)}), \nabla f_2(\mb{x}_2^{(t)}), \ldots, \nabla f_N(\mb{x}_N^{(t)})\big] 
\end{align} 

From \eqref{eq:pf-lm-decentralized-node-diff-eq3}, for all $t\geq1$, we have
\begin{align}
&\mbb{E}[\norm{\mb{X}^{(t)}(\mb{I} - \mb{Q})}_F^2] \nonumber \\
=~& \gamma^2 \mbb{E}[\norm{\sum_{\tau=0}^{t-1} \frac{1-\beta^{t-\tau}}{1-\beta}\mb{G}^{(\tau)}(\mb{I}-\mb{Q})\mb{W}^{t-1-\tau} }_F^2] \nonumber\\
\overset{(a)}{\leq} & 2\gamma^2 \mbb{E}[\norm{\sum_{\tau=0}^{t-1} \frac{1-\beta^{t-\tau}}{1-\beta}(\mb{G}^{(\tau)}-\mb{H}^{(\tau)})(\mb{I}-\mb{Q})\mb{W}^{t-1-\tau} }_F^2] + 2\gamma^2 \mbb{E}[\norm{\sum_{\tau=0}^{t-1} \frac{1-\beta^{t-\tau}}{1-\beta}\mb{H}^{(\tau)}(\mb{I}-\mb{Q})\mb{W}^{t-1-\tau} }_F^2] \label{eq:pf-lm-decentralized-node-diff-eq4}
\end{align}
where (a) follows from the basic inequality $\norm{\mb{A}_1+\mb{A}_2}_F^2 \leq 2 \norm{\mb{A}_1}_F^2 + 2 \norm{\mb{A}_2}_F^2$.

Now we develop the respective upper bounds of the two terms on the right side of \eqref{eq:pf-lm-decentralized-node-diff-eq4}. We note that
\begin{align}
&\mbb{E}[\norm{\sum_{\tau=0}^{t-1} \frac{1-\beta^{t-\tau}}{1-\beta}(\mb{G}^{(\tau)}-\mb{H}^{(\tau)}) (\mb{I}-\mb{Q})\mb{W}^{t-1-\tau}}_F^2] \nonumber\\
\overset{(a)}{=}& \sum_{\tau=0}^{t-1} \mbb{E}[\norm{ \frac{1-\beta^{t-\tau}}{1-\beta}(\mb{G}^{(\tau)}-\mb{H}^{(\tau)})(\mb{I}-\mb{Q})\mb{W}^{t-1-\tau} }_F^2] \nonumber \\
\overset{(b)}{\leq}& \frac{1}{(1-\beta)^2} \sum_{\tau=0}^{t-1} \rho^{t-1-\tau}\mbb{E}[\norm{\mb{G}^{(\tau)}-\mb{H}^{(\tau)}}_F^2] \nonumber\\
\overset{(c)}{\leq}& \frac{1}{(1-\beta)^2} \sum_{\tau=0}^{t-1} \rho^{t-1-\tau} N\sigma^2  \nonumber \\
\overset{(d)}{\leq}& \frac{N\sigma^2}{(1-\beta)^2(1-\rho)}   \label{eq:pf-lm-decentralized-node-diff-eq5}
\end{align}
where (a) follows because $\mbb{E}[\mb{G}^{(\tau_2)}-\mb{H}^{(\tau_2)} | \mb{G}^{(\tau_1)}-\mb{H}^{(\tau_1)}] = \mb{0}$ for any $\tau_2 > \tau_1$; (b) follows because $\abs{ \frac{1-\beta^{t-\tau}}{1-\beta}} \leq \frac{1}{1-\beta}$, $\norm{\mb{A}\mb{B}}_F \leq \norm{\mb{A}}_F \norm{\mb{B}}$ for any two matrices $\mb{A}, \mb{B}$ and $\norm{(\mb{I}-\mb{Q})\mb{W}^{t-1-\tau} } \leq \rho^{(t-1-\tau)/2}$ by Fact  \ref{fact:matrix-identity}; and (c) follows by noting $\mbb{E}[\norm{\mb{G}^{(\tau)}-\mb{H}^{(\tau)}}_F^2] = \sum_{i=1}^N\mbb{E}[\norm{\mb{g}_i^{(\tau)} - \nabla f_i (\mb{x}_i^{(\tau)})}^2] \leq N\sigma^2$ where the last inequality follows by Assumption \ref{ass:basic}; and (d) follows because $\sum_{\tau=0}^{t-1}\rho^{t-1-\tau} =\frac{1-\rho^t}{1-\rho} \leq \frac{1}{1-\rho}$.

We also note that  
\begin{align}
&\mbb{E}[\norm{\sum_{\tau=0}^{t-1} \frac{1-\beta^{t-\tau}}{1-\beta}\mb{H}^{(\tau)}(\mb{I}-\mb{Q})\mb{W}^{t-1-\tau} }_F^2] \nonumber\\
\overset{(a)}{\leq}~ & \sum_{\tau=0}^{t-1} \sum_{\tau^\prime = 0}^{t-1} \mbb{E}[\norm{ \frac{1-\beta^{t-\tau}}{1-\beta}\mb{H}^{(\tau)}(\mb{I}-\mb{Q})\mb{W}^{t-1-\tau}}_F \norm{ \frac{1-\beta^{t-\tau}}{1-\beta}\mb{H}^{(\tau^\prime)}(\mb{I}-\mb{Q})\mb{W}^{t-1-\tau^\prime} }_F] \nonumber\\
\overset{(b)}{\leq}~& \frac{1}{(1-\beta)^2} \sum_{\tau=0}^{t-1} \sum_{\tau^\prime = 0}^{t-1} \rho^{t-1-(\tau + \tau^\prime)/2} \mbb{E}[\norm{\mb{H}^{(\tau)}}_F \norm{\mb{H}^{(\tau^\prime)}}_F] \nonumber\\
\overset{(c)}{\leq}~& \frac{1}{(1-\beta)^2} \sum_{\tau=0}^{t-1} \sum_{\tau^\prime = 0}^{t-1} \rho^{t-1-(\tau + \tau^\prime)/2} \big(\frac{1}{2}\mbb{E}[\norm{\mb{H}^{(\tau)}}_F^2] + \frac{1}{2} \mbb{E}[\norm{\mb{H}^{(\tau^\prime)}}_F^2]\big) \nonumber\\
\overset{(d)}{=}~& \frac{1}{(1-\beta)^2} \sum_{\tau=0}^{t-1} \sum_{\tau^\prime = 0}^{t-1} \rho^{t-1-(\tau + \tau^\prime)/2}\mbb{E}[\norm{\mb{H}^{(\tau)}}_F^2] \nonumber\\
\overset{(e)}{\leq}~& \frac{1}{(1-\beta)^2(1-\sqrt{\rho})} \sum_{\tau=0}^{t-1}  \rho^{(t-1-\tau)/2} \mbb{E}[\norm{\mb{H}^{(\tau)}}_F^2]   \label{eq:pf-lm-decentralized-node-diff-eq6}
\end{align}
where (a) follows from Fact \ref{fact:frobenius-norm-inequality} ; (b) follows because $\abs{ \frac{1-\beta^{t-\tau}}{1-\beta}} \leq \frac{1}{1-\beta}$, $\norm{\mb{A}\mb{B}}_F \leq \norm{\mb{A}}_F \norm{\mb{B}}$ for any two matrices $\mb{A}, \mb{B}$ and $\norm{(\mb{I}-\mb{Q})\mb{W}^{t-1-\tau} } \leq \rho^{(t-1-\tau)/2}$ by Fact  \ref{fact:matrix-identity}; (c) follows from the basic identity $xy \leq \frac{1}{2}x^2 + \frac{1}{2}y^2$ for any two real number $x,y$; (d) follows by noting the symmetry between $\tau$ and $\tau^\prime$ in the expression from last line; (e) follows because $\sum_{\tau^\prime = 0}^{t-1}\rho^{t-1-(\tau + \tau^\prime)/2} \leq \frac{\rho^{(t-1-\tau)/2}}{1-\sqrt{\rho}}$.

For all $t\geq 0$, we denote
\begin{align}
\mb{J}^{(t)} \defeq \big[\nabla f_1(\bar{\mb{x}}^{(t)}), \nabla f_2(\bar{\mb{x}}^{(t)}), \ldots, \nabla f_N(\bar{\mb{x}}^{(t)})\big] 
\end{align} 

Then, for all $\tau\geq 0$, we have  
\begin{align}
&\mbb{E}[\norm{\mb{H}^{(\tau)}}_F^2]  \nonumber \\
=~&\mbb{E}[\norm{\mb{H}^{(\tau)} -\mb{J}^{(\tau)} +\mb{J}^{(\tau)} -\mb{J}^{(\tau)}\mb{Q} +\mb{J}^{(\tau)}\mb{Q} -\mb{H}^{(\tau)}\mb{Q} +\mb{H}^{(\tau)}\mb{Q}}_F^2] \nonumber\\
\overset{(a)}{\leq}~& 4\mbb{E}[\norm{ \mb{H}^{(\tau)} -\mb{J}^{(\tau)}}_F^2] + 
 4 \mbb{E}[\norm{ \mb{J}^{(\tau)} -\mb{J}^{(\tau)}\mb{Q}}_F^2] +4\mbb{E}[\norm{(\mb{J}^{(\tau)} -\mb{H}^{(\tau)})\mb{Q}}_F^2]  + 4 \mbb{E}[\norm{ \mb{H}^{(\tau)}\mb{Q}}_F^2]  \nonumber \\
 \overset{(b)}{\leq}~& 4L^2\mbb{E}[\norm{ \mb{X}^{(\tau)}(\mb{I} - \mb{Q})}_F^2]  + 4N\kappa^2 + 4L^2\mbb{E}[\norm{ \mb{X}^{(\tau)}(\mb{I} - \mb{Q})}_F^2] + 4N  \mbb{E}[\norm{\frac{1}{N} \sum_{i=1}^N \nabla f_i (\mb{x}_i^{(\tau)})}^2] \nonumber \\
 =~& 8L^2\mbb{E}[\norm{ \mb{X}^{(\tau)}(\mb{I} - \mb{Q})}_F^2]  + 4N\kappa^2 + 4N \mbb{E}[\norm{\frac{1}{N} \sum_{i=1}^N \nabla f_i (\mb{x}_i^{(\tau)})}^2]   \label{eq:pf-lm-decentralized-node-diff-eq7}
\end{align}
where (a) follows by applying the basic inequality $\norm{\sum_{i=1}^n\mb{A}_i}_F^2 \leq n\sum_{i=1}^n \norm{\mb{A}_i}_F^2$ for matrices of the same dimension with $n=4$; (b) follows because $\norm{ \mb{H}^{(\tau)} -\mb{J}^{(\tau)}}_F^2 = \sum_{i=1}^N \norm{\nabla f_i(\mb{x}_i^{(\tau)}) - \nabla f_i(\bar{\mb{x}}^{(\tau)})}^2 \leq L^2 \sum_{i=1}^N \norm{\mb{x}_i^{(\tau)} -\bar{\mb{x}}^{(\tau)} }^2 = L^2 \norm{ \mb{X}^{(\tau)}(\mb{I} - \mb{Q})}_F^2$ where the inequality step follows from the smoothness of each $f_i(\cdot)$ by Assumption \ref{ass:basic}, $\norm{\mb{J}^{(\tau)} -\mb{J}^{(\tau)}\mb{Q}}_F^2 = \sum_{i=1}^N \norm{\nabla f_i(\bar{\mb{x}}^{(\tau)}) - \nabla f(\bar{\mb{x}}^{(\tau)}) }^2 \leq N\kappa^2$ where the second step follows from Assumption \ref{ass:basic}, $\norm{(\mb{J}^{(\tau)} -\mb{H}^{(\tau)})\mb{Q}}_F^2 \leq \norm{ \mb{H}^{(\tau)} -\mb{J}^{(\tau)}}_F^2$.

Substituting \eqref{eq:pf-lm-decentralized-node-diff-eq7} into \eqref{eq:pf-lm-decentralized-node-diff-eq6} yields
{\small
\begin{align}
&\mbb{E}[\norm{\sum_{\tau=0}^{t-1} \frac{1-\beta^{t-\tau}}{1-\beta}\mb{H}^{(\tau)}(\mb{I}-\mb{Q})\mb{W}^{t-1-\tau} }_F^2] \nonumber\\
\leq &  \frac{1}{(1-\beta)^2(1-\sqrt{\rho})} \Big[ \Big(8L^2\sum_{\tau=0}^{t-1}  \rho^{(t-1-\tau)/2}  \mbb{E}[\norm{ \mb{X}^{(\tau)}(\mb{I} - \mb{Q})}_F^2]\Big)  +  \Big(4N\sum_{\tau=0}^{t-1}  \rho^{(t-1-\tau)/2} \mbb{E}[\norm{\frac{1}{N} \sum_{i=1}^N \nabla f_i (\mb{x}_i^{(\tau)})}^2] \Big) \nonumber\\ &\qquad \qquad \qquad \qquad+ N\kappa^2 \sum_{\tau=0}^{t-1}  \rho^{(t-1-\tau)/2}  \Big]  \nonumber\\
\leq & \frac{1}{(1-\beta)^2(1-\sqrt{\rho})} \Big[ \Big(8L^2\sum_{\tau=0}^{t-1}  \rho^{(t-1-\tau)/2}  \mbb{E}[\norm{ \mb{X}^{(\tau)}(\mb{I} - \mb{Q})}_F^2]\Big)  +  \Big(4N\sum_{\tau=0}^{t-1}  \rho^{(t-1-\tau)/2} \mbb{E}[\norm{\frac{1}{N} \sum_{i=1}^N \nabla f_i (\mb{x}_i^{(\tau)})}^2] \Big) \nonumber\\ &\qquad \qquad \qquad \qquad+ N\kappa^2 \frac{1}{1-\sqrt{\rho}}  \Big] \label{eq:pf-lm-decentralized-node-diff-eq8}
\end{align}
}%

Substituting \eqref{eq:pf-lm-decentralized-node-diff-eq5} and \eqref{eq:pf-lm-decentralized-node-diff-eq8} into \eqref{eq:pf-lm-decentralized-node-diff-eq4} yields
{\small
\begin{align*}
&\mbb{E}[\norm{\mb{X}^{(t)}(\mb{I} - \mb{Q})}_F^2] \nonumber \\
\leq & \frac{2N\gamma^2\sigma^2}{(1-\beta)^2(1-\rho)}  +  \frac{2\gamma^2}{(1-\beta)^2(1-\sqrt{\rho})} \Big[ \Big(8L^2\sum_{\tau=0}^{t-1}  \rho^{(t-1-\tau)/2}  \mbb{E}[\norm{ \mb{X}^{(\tau)}(\mb{I} - \mb{Q})}_F^2]\Big)  +  \Big(4N\sum_{\tau=0}^{t-1}  \rho^{(t-1-\tau)/2} \mbb{E}[\norm{\frac{1}{N} \sum_{i=1}^N \nabla f_i (\mb{x}_i^{(\tau)})}^2] \Big) \Big]  \nonumber\\
& +    \frac{2N\gamma^2\kappa^2}{(1-\beta)^2(1-\sqrt{\rho})^2} 
\end{align*}
}%

Summing over $t\in\{1,2,\ldots, T-1\}$ and noting that $\norm{\mb{X}^{(0)}(\mb{I} - \mb{Q})}_F^2 = 0$ yields
\begin{align}
&\sum_{t=0}^{T-1}\mbb{E}[\norm{\mb{X}^{(t)}(\mb{I} - \mb{Q})}_F^2] \nonumber \\
\leq & \frac{2N\gamma^2\sigma^2}{(1-\beta)^2(1-\rho)}T  +  \frac{16L^2\gamma^2}{(1-\beta)^2(1-\sqrt{\rho})} \sum_{t=1}^{T-1}\sum_{\tau=0}^{t-1}  \rho^{(t-1-\tau)/2}  \mbb{E}[\norm{ \mb{X}^{(\tau)}(\mb{I} - \mb{Q})}_F^2]  \nonumber \\ &+  \frac{8N\gamma^2}{(1-\beta)^2(1-\sqrt{\rho})} \sum_{t=1}^{T-1}  \sum_{\tau=0}^{t-1}  \rho^{(t-1-\tau)/2} \mbb{E}[\norm{\frac{1}{N} \sum_{i=1}^N \nabla f_i (\mb{x}_i^{(\tau)})}^2]  +    \frac{2N\gamma^2\kappa^2}{(1-\beta)^2(1-\sqrt{\rho})^2}T \nonumber\\
\overset{(a)}{\leq}&\frac{2N\gamma^2\sigma^2}{(1-\beta)^2(1-\rho)}T + \frac{16L^2\gamma^2}{(1-\beta)^2(1-\sqrt{\rho})} \sum_{t=0}^{T-1}\frac{1-\rho^{(T-1-t)/2}}{1-\sqrt{\rho}} \mbb{E}[\norm{ \mb{X}^{(t)}(\mb{I} - \mb{Q})}_F^2]  \nonumber\\ &+  \frac{8N\gamma^2}{(1-\beta)^2(1-\sqrt{\rho})}\sum_{t=0}^{T-1}\frac{1-\rho^{(T-1-t)/2}}{1-\sqrt{\rho}} \mbb{E}[\norm{\frac{1}{N} \sum_{i=1}^N \nabla f_i (\mb{x}_i^{(t)})}^2]  +   \frac{2N\gamma^2\kappa^2}{(1-\beta)^2(1-\sqrt{\rho})^2}T \nonumber\\
\leq &\frac{2N\gamma^2\sigma^2}{(1-\beta)^2(1-\rho)}T + \frac{16L^2\gamma^2}{(1-\beta)^2(1-\sqrt{\rho})^2} \sum_{t=0}^{T-1} \mbb{E}[\norm{ \mb{X}^{(t)}(\mb{I} - \mb{Q})}_F^2]  \nonumber\\ &+  \frac{8N\gamma^2}{(1-\beta)^2(1-\sqrt{\rho})^2}\sum_{t=0}^{T-1}\mbb{E}[\norm{\frac{1}{N} \sum_{i=1}^N \nabla f_i (\mb{x}_i^{(t)})}^2]  +   \frac{2N\gamma^2\kappa^2}{(1-\beta)^2(1-\sqrt{\rho})^2}T \nonumber
\end{align}
where (a) follows after simplifying the double summations (by swapping the order of two summations and collecting coefficients for common terms)

Rearranging terms and dividing both sides by $N(1- \frac{16L^2\gamma^2}{(1-\beta)^2(1-\sqrt{\rho})^2})$ yields

\begin{align*}
&\sum_{t=0}^{T-1}\frac{1}{N}\mbb{E}[\norm{\mb{X}^{(t)}(\mb{I} - \mb{Q})}_F^2] \nonumber \\
\leq &\frac{1}{1- \frac{16L^2\gamma^2}{(1-\beta)^2(1-\sqrt{\rho})^2}}\frac{2\gamma^2\sigma^2}{(1-\beta)^2(1-\rho)}T + \frac{1}{1- \frac{16L^2\gamma^2}{(1-\beta)^2(1-\sqrt{\rho})^2}} \frac{8\gamma^2}{(1-\beta)^2(1-\sqrt{\rho})^2}\sum_{t=0}^{T-1}\mbb{E}[\norm{\frac{1}{N} \sum_{i=1}^N \nabla f_i (\mb{x}_i^{(t)})}^2] \nonumber\\ &+ \frac{1}{1- \frac{16L^2\gamma^2}{(1-\beta)^2(1-\sqrt{\rho})^2}}  \frac{2\gamma^2\kappa^2}{(1-\beta)^2(1-\sqrt{\rho})^2}T
\end{align*}

Finally, this lemma follows by recalling that $\norm{\mb{X}^{(t)}(\mb{I} - \mb{Q})}_F^2 = \sum_{i=1}^{N} \norm{\bar{\mb{x}}^{(t)}  - \mb{x}_i^{(t)}}^2$.
\end{proof}

\begin{Rem}
Comparing Lemma \ref{lm:x-bar-x-diff-decentralized} with Lemma \ref{lm:diff-avg-per-node}, we note that the upper bound in Lemma \ref{lm:x-bar-x-diff-decentralized} includes an additional term involving $\mbb{E}[\norm{\frac{1}{N} \sum_{i=1}^N \nabla f_i (\mb{x}_i^{(t)})}^2]$, which is the cost of asymmetry in local averages.
\end{Rem}

To simplify the lengthy coefficient $\frac{1}{1- \frac{16L^2\gamma^2}{(1-\beta)^2(1-\sqrt{\rho})^2}}$ in Lemma \ref{lm:x-bar-x-diff-decentralized}, we have the following trivial corollary.

\begin{Cor}\label{cor:x-bar-x-diff-decentralized}
	Consider problem \eqref{eq:sto-opt} under Assumptions \ref{ass:basic} and \ref{ass:mix-matrix}.  Let $\{\bar{\mb{x}}^{(t)}\}_{t\geq 0}$ defined in \eqref{eq:def-x-bar} be the node averages of local solutions from Algorithm \ref{alg:prsgd-decentralized} with Option I.  If $\gamma \leq \frac{(1-\beta)(1-\sqrt{\rho})}{4\sqrt{2}L}$, then for all $T\geq 1$, we have 
	\begin{align*}
	&\sum_{t=0}^{T-1}\frac{1}{N}\sum_{i=1}^{N}\mbb{E}[\norm{\bar{\mb{x}}^{(t)}  - \mb{x}_i^{(t)}}^2] \\
	\leq &\frac{4\gamma^2\sigma^2}{(1-\beta)^2(1-\rho)}T +  \frac{16\gamma^2}{(1-\beta)^2(1-\sqrt{\rho})^2}\sum_{t=0}^{T-1}\mbb{E}[\norm{\frac{1}{N} \sum_{i=1}^N \nabla f_i (\mb{x}_i^{(t)})}^2] +   \frac{4\gamma^2\kappa^2}{(1-\beta)^2(1-\sqrt{\rho})^2}T
	\end{align*}
\end{Cor}
\begin{proof}
It follows because $\frac{1}{1- \frac{16L^2\gamma^2}{(1-\beta)^2(1-\sqrt{\rho})^2}} < 2$ under our choice of $\gamma$.
\end{proof}

\subsubsection{Main Proof of Theorem \ref{thm:decentralized-rate}}

Now, we are ready to present the main proof of Theorem \ref{thm:decentralized-rate}. Since Lemmas \ref{lm:z-diff-decentralized} and \ref{lm:z-x-diff-decentralized} under Algorithm \ref{alg:prsgd-decentralized} are respectively identical to Lemmas \ref{lm:z-diff} and \ref{lm:z-x-diff}, repeating the steps from \eqref{eq:pf-thm-rate-eq1} to \eqref{eq:pf-thm-rate-eq6} in Section \ref{sec:main-pf-thm-polyak-rate} yields:

\begin{align}
\mbb{E}[\norm{\nabla f(\bar{\mb{x}}^{(t)})}^2] \leq &\frac{2(1-\beta)}{\gamma} \big(\mbb{E}[f(\bar{\mb{z}}^{(t)})] -\mbb{E}[f(\bar{\mb{z}}^{(t+1)})]  \big) - (1- \frac{L \gamma\beta}{(1-\beta)^2})\mbb{E}[\norm{\frac{1}{N}\sum_{i=1}^N\nabla f_i (\mb{x}_i^{(t)})}^2]  \nonumber\\
&+ \frac{(1-\beta)^2 L}{\beta \gamma} \mbb{E}[\norm{ \bar{\mb{z}}^{(t)} - \bar{\mb{x}}^{(t)} }^2] +  \frac{L\gamma}{(1-\beta)} \mbb{E}[\norm{\frac{1}{N}\sum_{i=1}^N \mb{g}_i^{(t)}}^2] + L^2 \frac{1}{N}\sum_{i=1}^N\mbb{E}[ \Vert \bar{\mb{x}}^{(t)} - \mb{x}_i^{(t)}\Vert^2]  
\end{align}

Summing over $t\in \{0,1,\ldots, T-1\}$  yields

\begin{align}
&\sum_{t=0}^{T-1}\mbb{E}[\norm{\nabla f(\bar{\mb{x}}^{(t)})}^2] \nonumber \\
\leq &\frac{2(1-\beta)}{\gamma } \big(\mbb{E}[f(\bar{\mb{z}}^{(0)})] -\mbb{E}[f(\bar{\mb{z}}^{(T)})]  \big) - \Big(1- \frac{L\gamma \beta}{(1-\beta)^2}\Big) \sum_{t=0}^{T-1}\mbb{E}[\norm{\frac{1}{N}\sum_{i=1}^N\nabla f_i (\mb{x}_i^{(t)})}^2] \nonumber\\
&+ \frac{(1-\beta)^2 L}{\beta \gamma} \sum_{t=0}^{T-1}\mbb{E}[\norm{ \bar{\mb{z}}^{(t)} - \bar{\mb{x}}^{(t)} }^2] + \frac{L\gamma}{(1-\beta)} \sum_{t=0}^{T-1}\mbb{E}[\norm{\frac{1}{N}\sum_{i=1}^N \mb{g}_i^{(t)}}^2] + L^2 \sum_{t=0}^{T-1} \frac{1}{N}\sum_{i=1}^N\mbb{E}[ \Vert \bar{\mb{x}}^{(t)} - \mb{x}_i^{(t)}\Vert^2] \nonumber\\
\overset{(a)}{\leq} &\frac{2(1-\beta)}{\gamma } \big(\mbb{E}[f(\bar{\mb{z}}^{(0)})] -\mbb{E}[f(\bar{\mb{z}}^{(T)})]  \big) - \Big(1- \frac{L\gamma\beta}{(1-\beta)^2}-  \frac{16L^2\gamma^2}{(1-\beta)^2(1-\sqrt{\rho})^2}\Big) \sum_{t=0}^{T-1}\mbb{E}[\norm{\frac{1}{N}\sum_{i=1}^N\nabla f_i (\mb{x}_i^{(t)})}^2]  \nonumber \\ & + \frac{L\gamma}{(1-\beta)^2} \sum_{t=0}^{T-1}\mbb{E}[\norm{\frac{1}{N}\sum_{i=1}^N \mb{g}_i^{(t)}}^2] +\frac{4L^2\gamma^2\sigma^2}{(1-\beta)^2(1-\rho)}T +   \frac{4L^2\gamma^2\kappa^2}{(1-\beta)^2(1-\sqrt{\rho})^2}T\nonumber \\
\overset{(b)}{\leq} &\frac{2(1-\beta)}{\gamma } \big(\mbb{E}[f(\bar{\mb{z}}^{(0)})] -\mbb{E}[f(\bar{\mb{z}}^{(T)})]  \big) - \Big(1- \frac{(1+\beta)L\gamma}{(1-\beta)^2}-  \frac{16L^2\gamma^2}{(1-\beta)^2(1-\sqrt{\rho})^2}\Big) \sum_{t=0}^{T-1}\mbb{E}[\norm{\frac{1}{N}\sum_{i=1}^N\nabla f_i (\mb{x}_i^{(t)})}^2]  \nonumber \\ & + \frac{L\gamma}{(1-\beta)^2}\frac{\sigma^2}{N}T +\frac{4L^2\gamma^2\sigma^2}{(1-\beta)^2(1-\rho)}T +   \frac{4L^2\gamma^2\kappa^2}{(1-\beta)^2(1-\sqrt{\rho})^2}T\nonumber \\
\overset{(c)}{\leq} &\frac{2(1-\beta)}{\gamma } \big(f(\bar{\mb{x}}^{(0)}) -f^\ast]  \big)  + \frac{L\gamma}{(1-\beta)^2}\frac{\sigma^2}{N}T +\frac{4L^2\gamma^2\sigma^2}{(1-\beta)^2(1-\rho)}T +   \frac{4L^2\gamma^2\kappa^2}{(1-\beta)^2(1-\sqrt{\rho})^2}T\nonumber 
\end{align}

where (a) follows by applying Lemma \ref{lm:z-x-diff-decentralized} (noting that $\frac{\beta L\gamma}{(1-\beta)^2} + \frac{L\gamma}{1-\beta} = \frac{L\gamma}{(1-\beta)^2}$) and Corollary \ref{cor:x-bar-x-diff-decentralized} (noting that $\gamma$ satisfies the condition in Corollary \ref{cor:x-bar-x-diff-decentralized});  (b) follows by noting that $\mbb{E}[\norm{\frac{1}{N}\sum_{i=1}^N \mb{g}_i^{(t)}}^2] \leq  \frac{1}{N}\sigma^{2}  +  \mathbb{E} [ \Vert \frac{1}{N} \sum_{i=1}^{N} \nabla f_{i}(\mathbf{x}_{i}^{(t)})\Vert^{2}]$ by Lemma \ref{lm:expected-gradient} ; and (c) follows because $\gamma $ is chosen to ensure $1- \frac{(1+\beta)L\gamma}{(1-\beta)^2}-  \frac{16L^2\gamma^2}{(1-\beta)^2(1-\sqrt{\rho})^2} \geq 0$, $\bar{\mb{z}}^{(0)} = \bar{\mb{x}}^{(0)}$ by definition in \eqref{eq:def-z-decentralized}, and $f^\ast$ is the minimum value of problem \eqref{eq:sto-opt}.

Dividing both sides by $T$ yields

\begin{align}
&\frac{1}{T} \sum_{t=0}^{T-1}\mbb{E}[\norm{\nabla f(\bar{\mb{x}}^{(t)})}^2] \nonumber \\
\leq&\frac{2(1-\beta)}{\gamma T} \big(f(\bar{\mb{x}}^{(0)}) - f^\ast \big) + \frac{L\gamma}{(1-\beta)^2}\frac{\sigma^2}{N} +\frac{4L^2\gamma^2\sigma^2}{(1-\beta)^2(1-\rho)} +   \frac{4L^2\gamma^2\kappa^2}{(1-\beta)^2(1-\sqrt{\rho})^2} \nonumber\\
=& O(\frac{1}{\gamma T})  + O(\frac{\gamma}{N} \sigma^2) + O(\gamma^2 \sigma^2)+ O(\gamma^2 \kappa^2)
\end{align}

\subsection{More Experiments} \label{sec:more-exp}

\subsubsection{Experiments on Algorithm \ref{alg:prsgd-momentum} with Constant Learning Rates}\label{sec:exp-clr}
\begin{figure*}[t!]
	\begin{subfigure}[t]{0.5\textwidth}
		\includegraphics[height=2.3in]{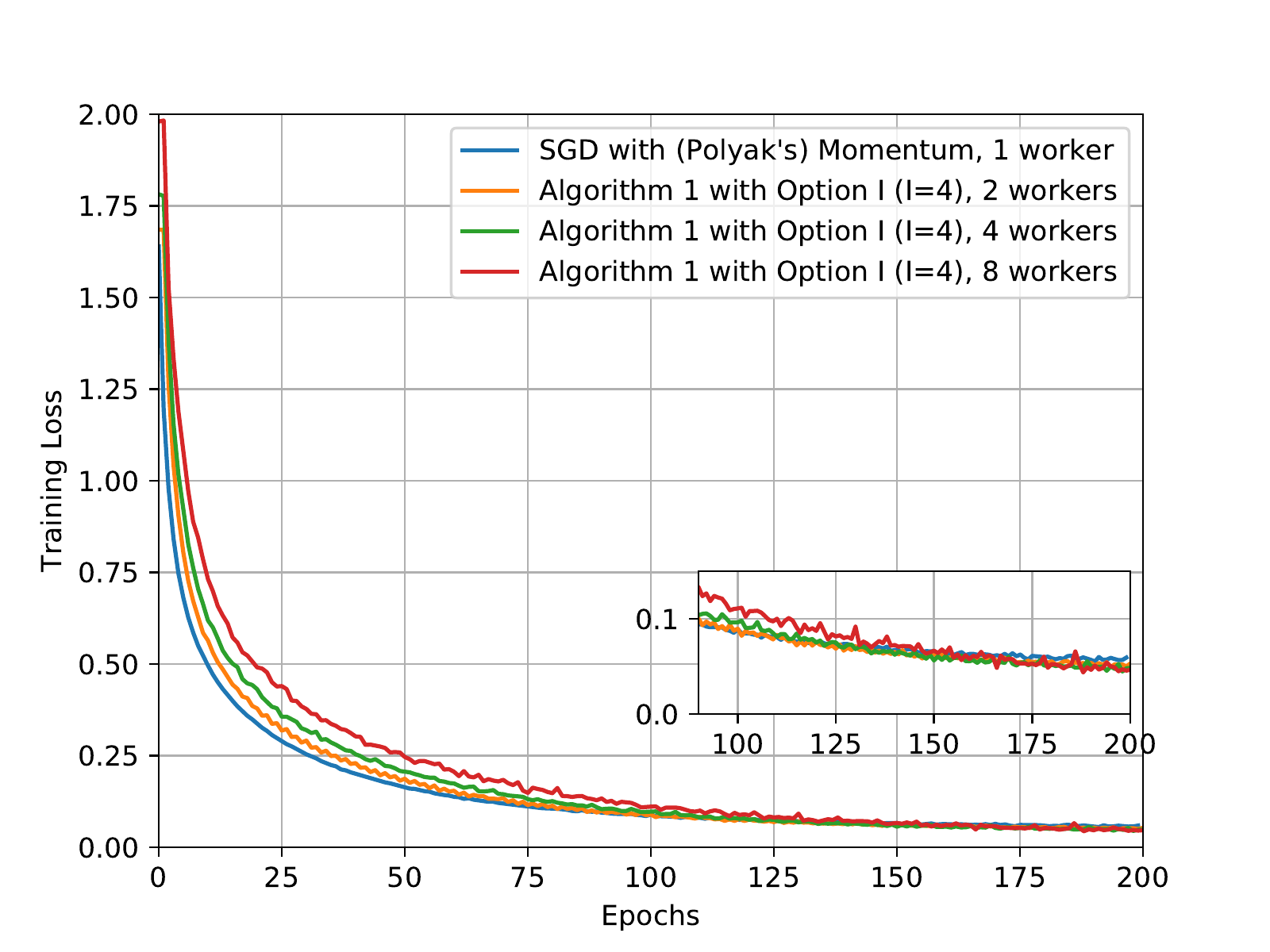}
		\caption{Training loss v.s. epochs. }
	\end{subfigure}
	~
	\begin{subfigure}[t]{0.5\textwidth}
		\includegraphics[height=2.3in]{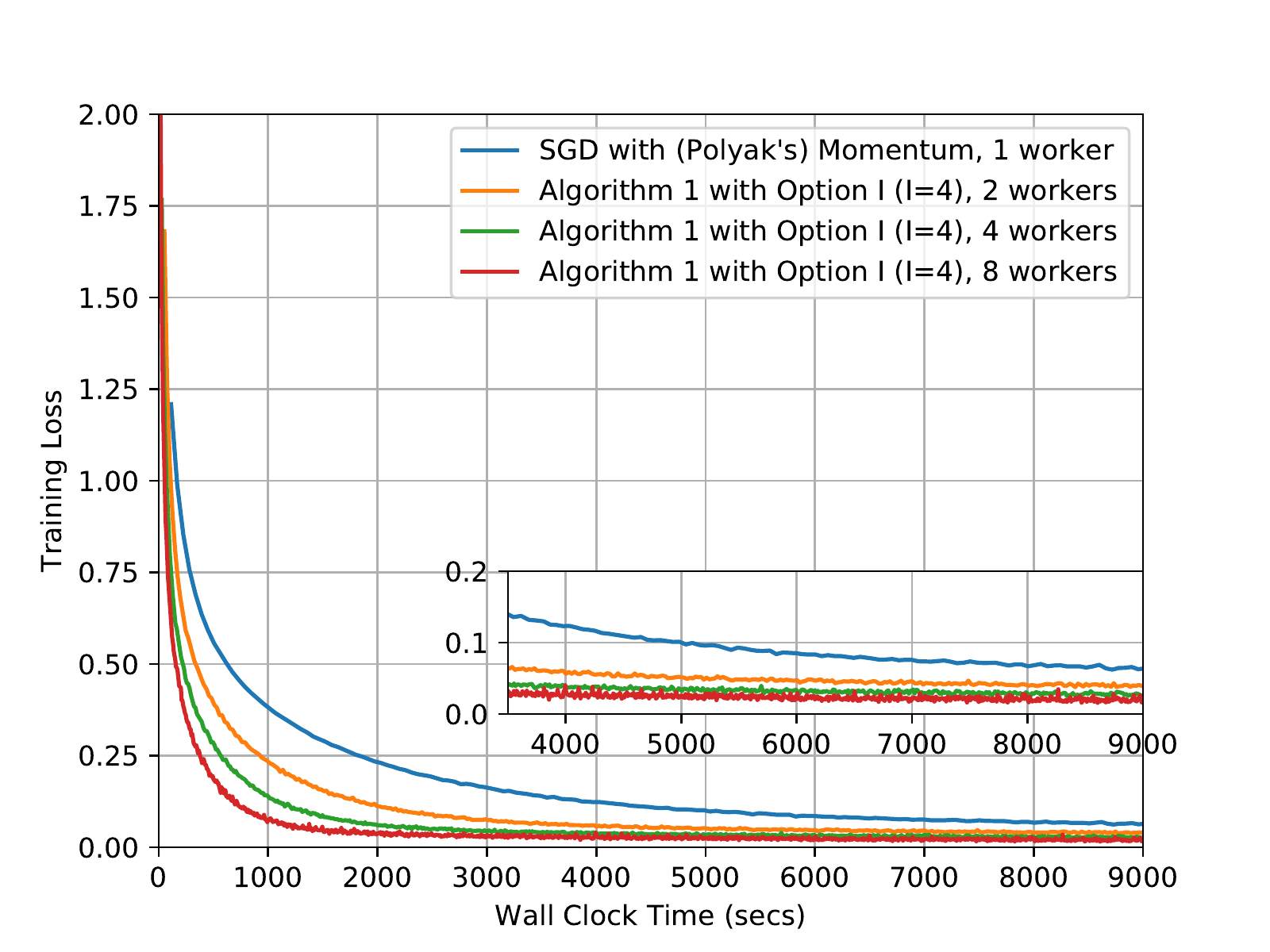}
		\caption{Training loss v.s. wall clock time. }
	\end{subfigure}
	\caption{Algorithm \ref{alg:prsgd-momentum} with constant learning rates: linear speedup verification with ResNet56 over CIFAR10. }
	\label{fig:polyak_clr}
\end{figure*}

Recall our theory, e.g., Corollary \ref{cor:prsgd-momentum-rate}, establishes the linear speedup property of Algorithm \ref{alg:prsgd-momentum} (with $I > 0$) using constant learning rates.  We now verify our theory by training ResNet56 over CIFAR10 using learning rates faithful to our theory. We run SGD with Polyak's momentum using momentum coefficient $0.9$, batch size $64$ and constant learning rate $0.01$.  We also run Algorithm \ref{alg:prsgd-momentum} with Option I and $I=4$ over $N\in\{2,4,8\}$ GPUs. The local batch size at each GPU is $64$ and the learning rate is set to $0.01\sqrt{N}$. Figure \ref{fig:polyak_clr}(a) plots the convergence of training loss in terms of the number of epochs that are jointly accessed by all used GPUs. Figure \ref{fig:polyak_clr}(a) verifies the linear speedup property of our Algorithm \ref{alg:prsgd-momentum} since each GPU in Algorithm \ref{alg:prsgd-momentum} only uses $1/N$ of the number of epochs, which are used by the single worker momentum SGD, to converge to the same loss value.  A more straightforward verification is given in Figure \ref{fig:polyak_clr}(b) where the x-axis is the wall clock time.

\subsubsection{Comparisons between Algorithm \ref{alg:prsgd-momentum} and Existing Model Averaging with Momentum SGD} \label{sec:naive-ma}

Our Algorithm \ref{alg:prsgd-momentum} is quite similar to a common practice, known as ``model averaging", used for training deep neural networks with multiple workers. When the deep neural works are trained with momentum SGD,  Microsoft's CNTK framework \cite{CNTK} and recent work \cite{WangJoshi18ArXiv2} suggest that each worker should additionally clear its local momentum buffer by reseting it to $0$ when they average their local models.  In contrast, our Algorithm \ref{alg:prsgd-momentum} averages both local momentum buffers and local models.  There is no convergence guarantees shown in the literature for the model averaging strategy suggested in \cite{CNTK,WangJoshi18ArXiv2}.  We run both Algorithm \ref{alg:prsgd-momentum} and the model averaging with Polyak's momentum SGD suggested in \cite{CNTK,WangJoshi18ArXiv2}, referred as ``model averaging with cleared momentum" to train ResNet56 over CIFAR10.  In the experiment, both Algorithm \ref{alg:prsgd-momentum} and ``model averaging with cleared momentum" perform one synchronization step, e.g., \eqref{eq:prsgd-restart} in Algorithm \ref{alg:prsgd-momentum},  every $I=16$ iterations.  Figure \ref{fig:naive_ma} shows that Algorithm \ref{alg:prsgd-momentum} converges slightly faster. More impressively, the test accuracy attained by Algorithm \ref{alg:prsgd-momentum} is roughly $1\%$ better.

\begin{figure*}[t!]
\begin{subfigure}[t]{0.5\textwidth}
\includegraphics[height=2.5in]{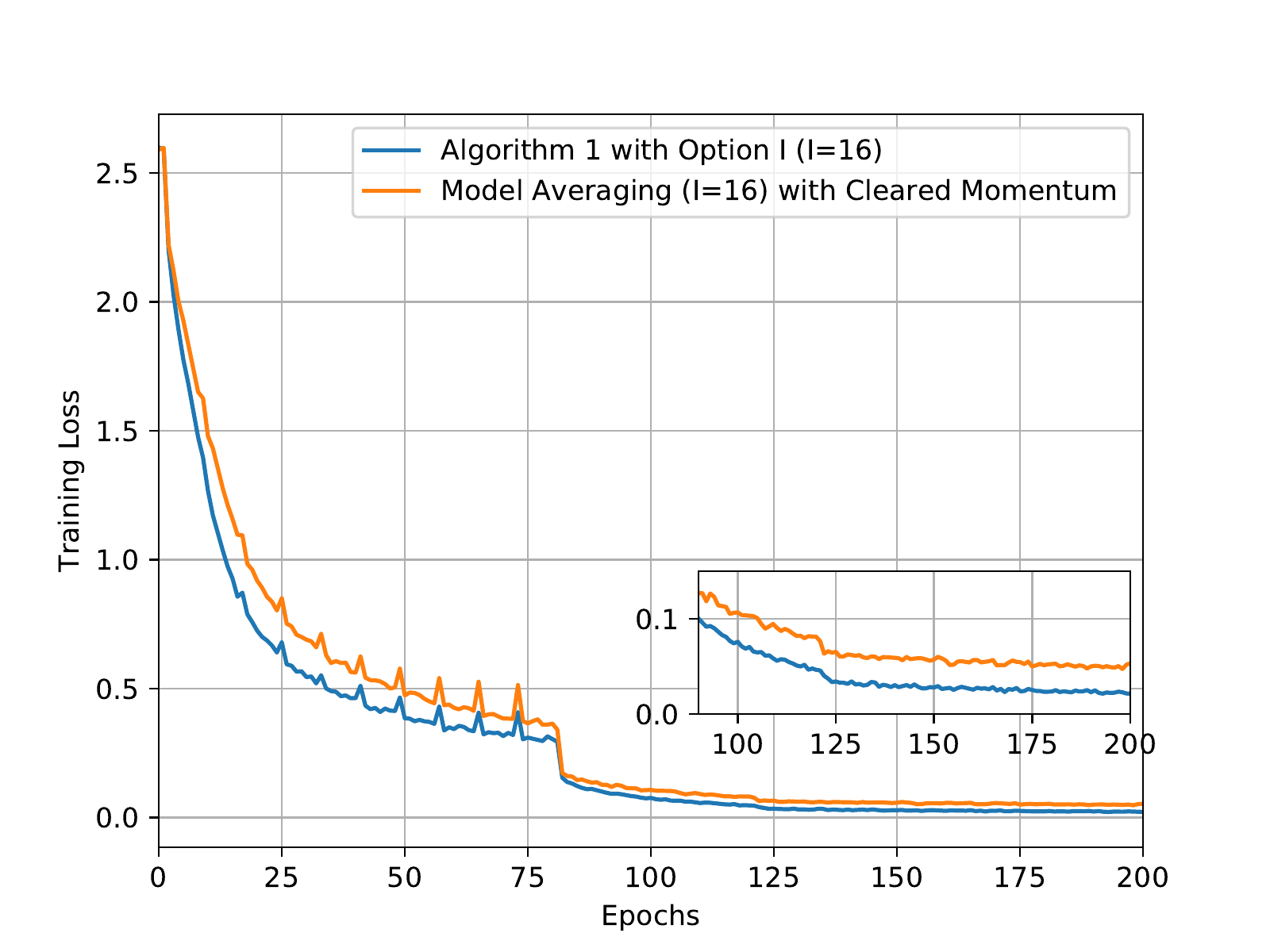}
\caption{Training loss v.s. epochs }
 \end{subfigure}
 ~
 \begin{subfigure}[t]{0.5\textwidth}
\includegraphics[height=2.5in]{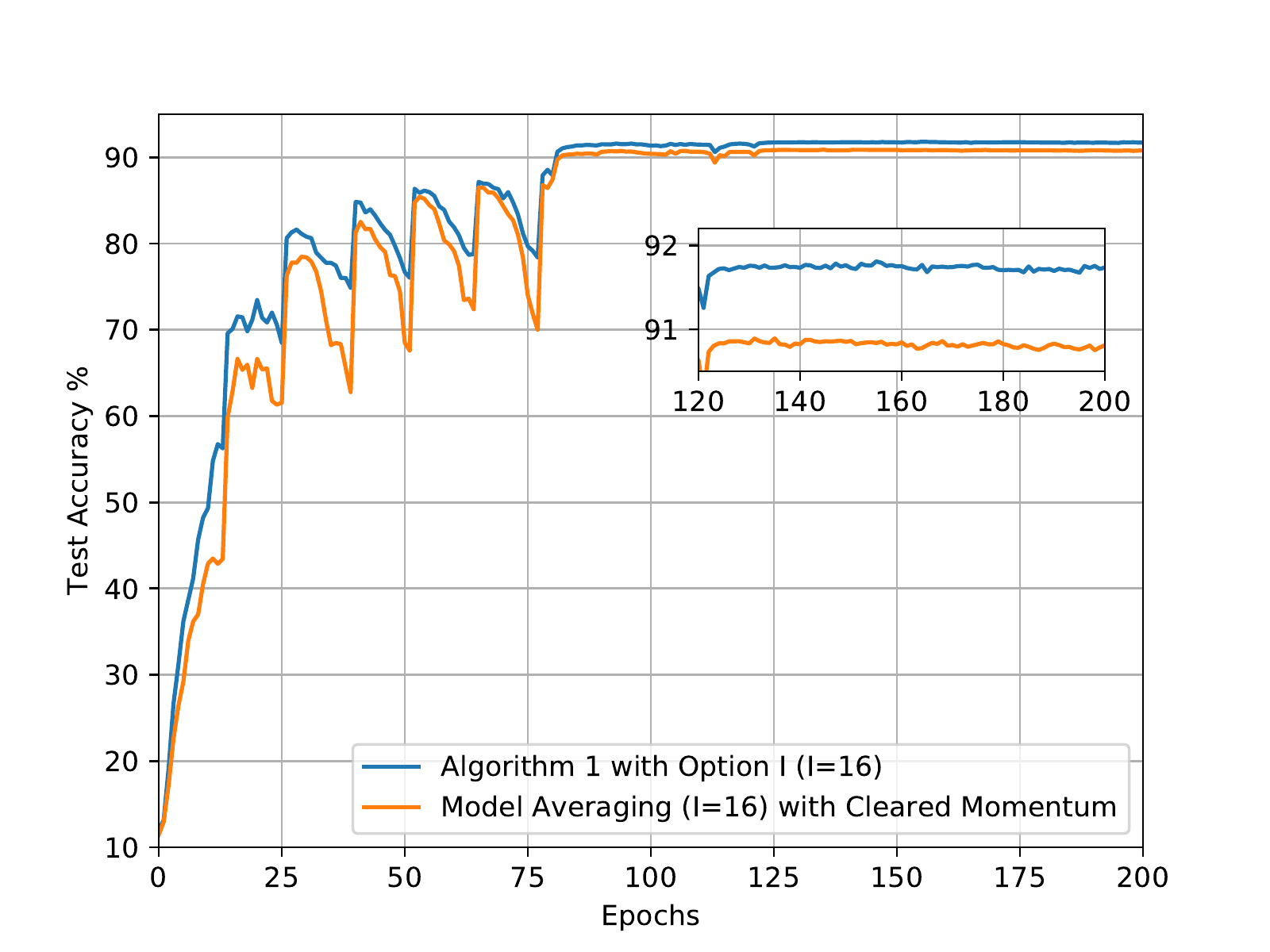}
\caption{Test accuracy v.s. epochs. }
 \end{subfigure}
  \caption{Comparisons between Algorithm \ref{alg:prsgd-momentum} and existing model averaging with momentum SGD.}
   \label{fig:naive_ma}
\end{figure*}

\subsubsection{Experiments with ImageNet} \label{sec:exp-imagenet}

\begin{figure*}[t!]
\begin{subfigure}[t]{0.5\textwidth}
\includegraphics[height=2.5in]{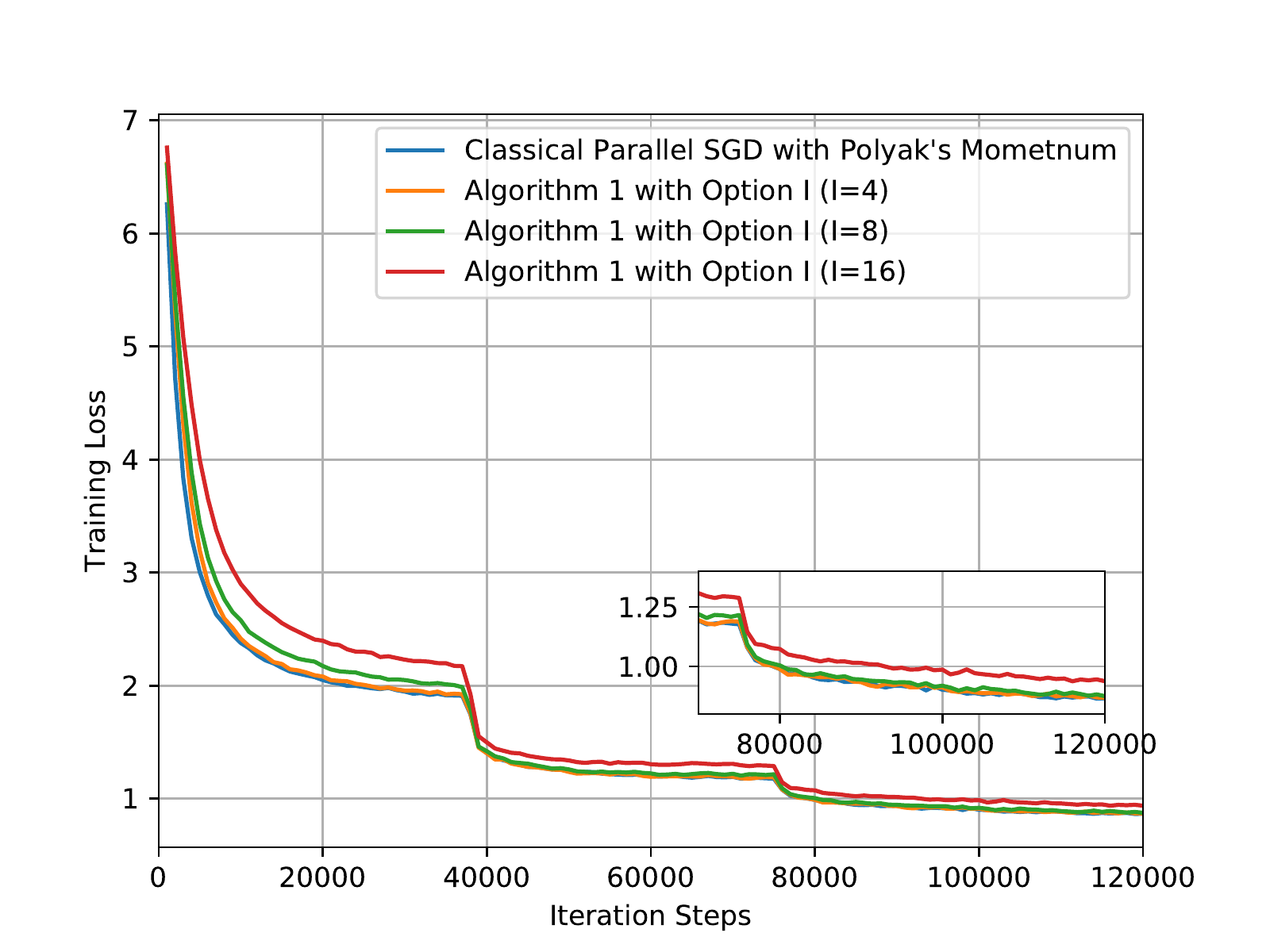}
\caption{Training loss v.s. iteration steps. }
 \end{subfigure}
 ~
 \begin{subfigure}[t]{0.5\textwidth}
\includegraphics[height=2.5in]{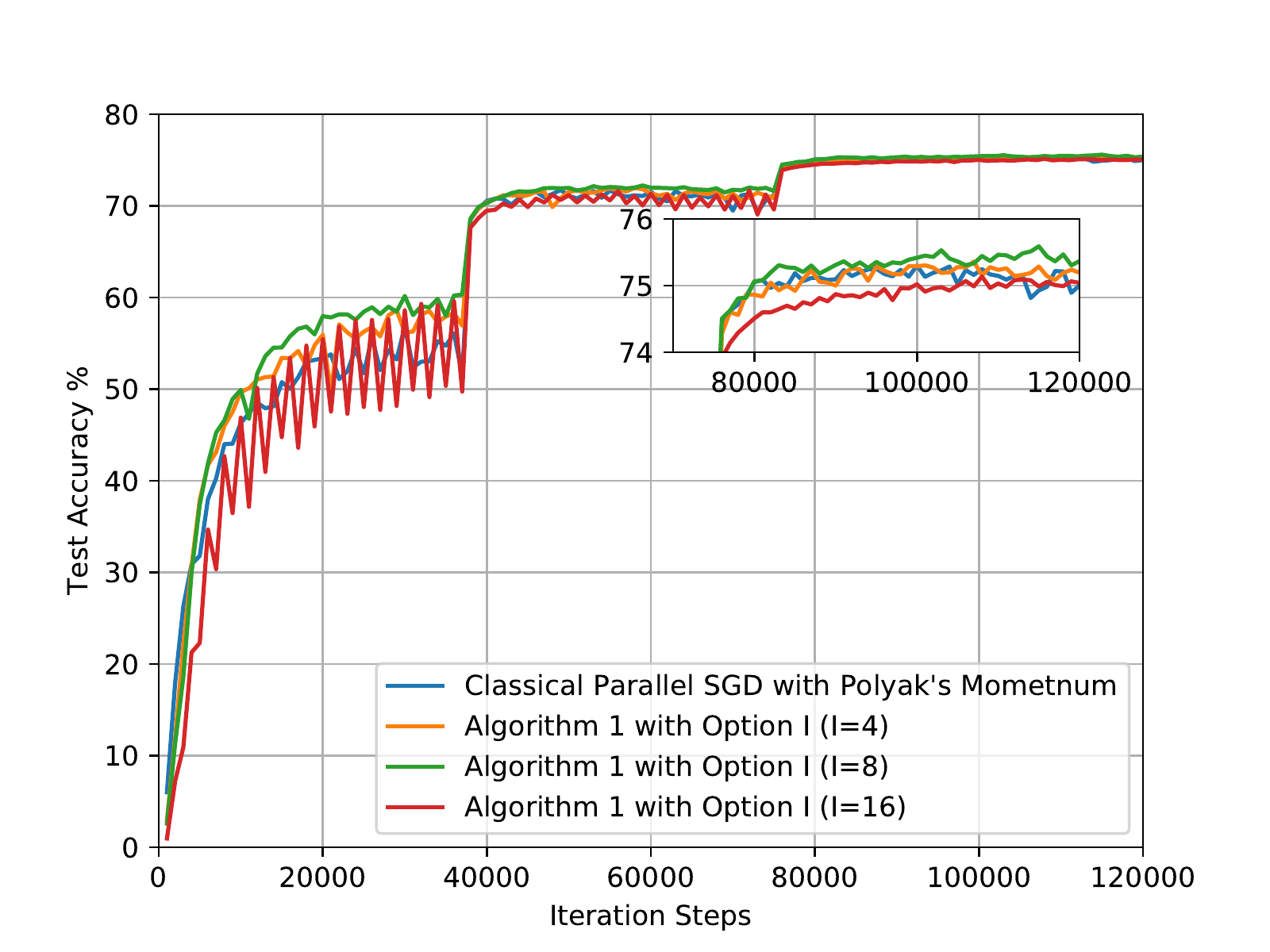}
\caption{Test accuracy v.s. iteration steps. }
 \end{subfigure}
  \caption{Algorithm \ref{alg:prsgd-momentum} with Option I: convergence v.s. iteration steps for ResNet50 over ImageNet.}
   \label{fig:imagenet-step}
\end{figure*} 

\begin{figure*}[t!]
\begin{subfigure}[t]{0.5\textwidth}
\includegraphics[height=2.5in]{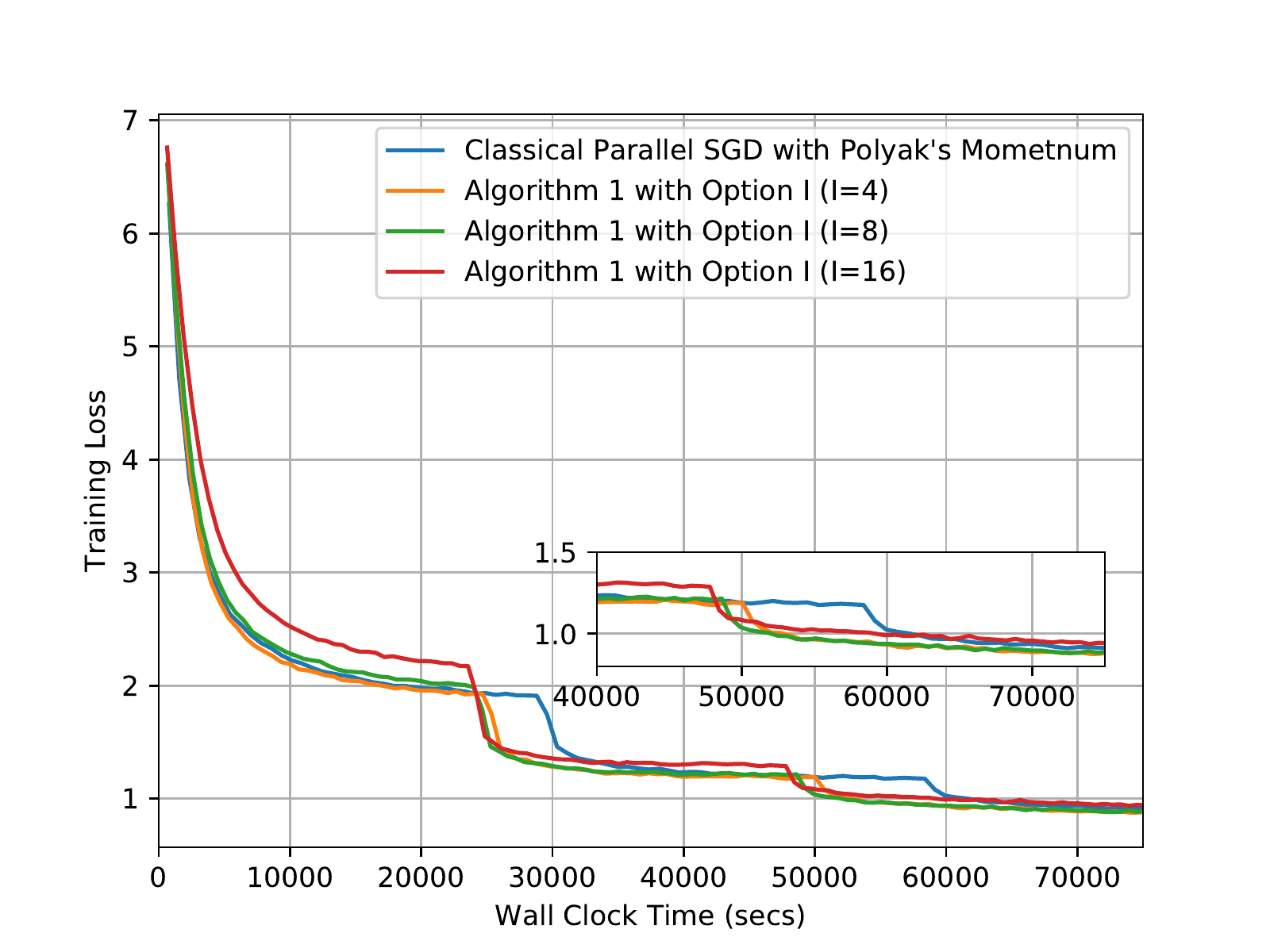}
\caption{Training loss v.s. wall clock time. }
 \end{subfigure}
 ~
 \begin{subfigure}[t]{0.5\textwidth}
\includegraphics[height=2.5in]{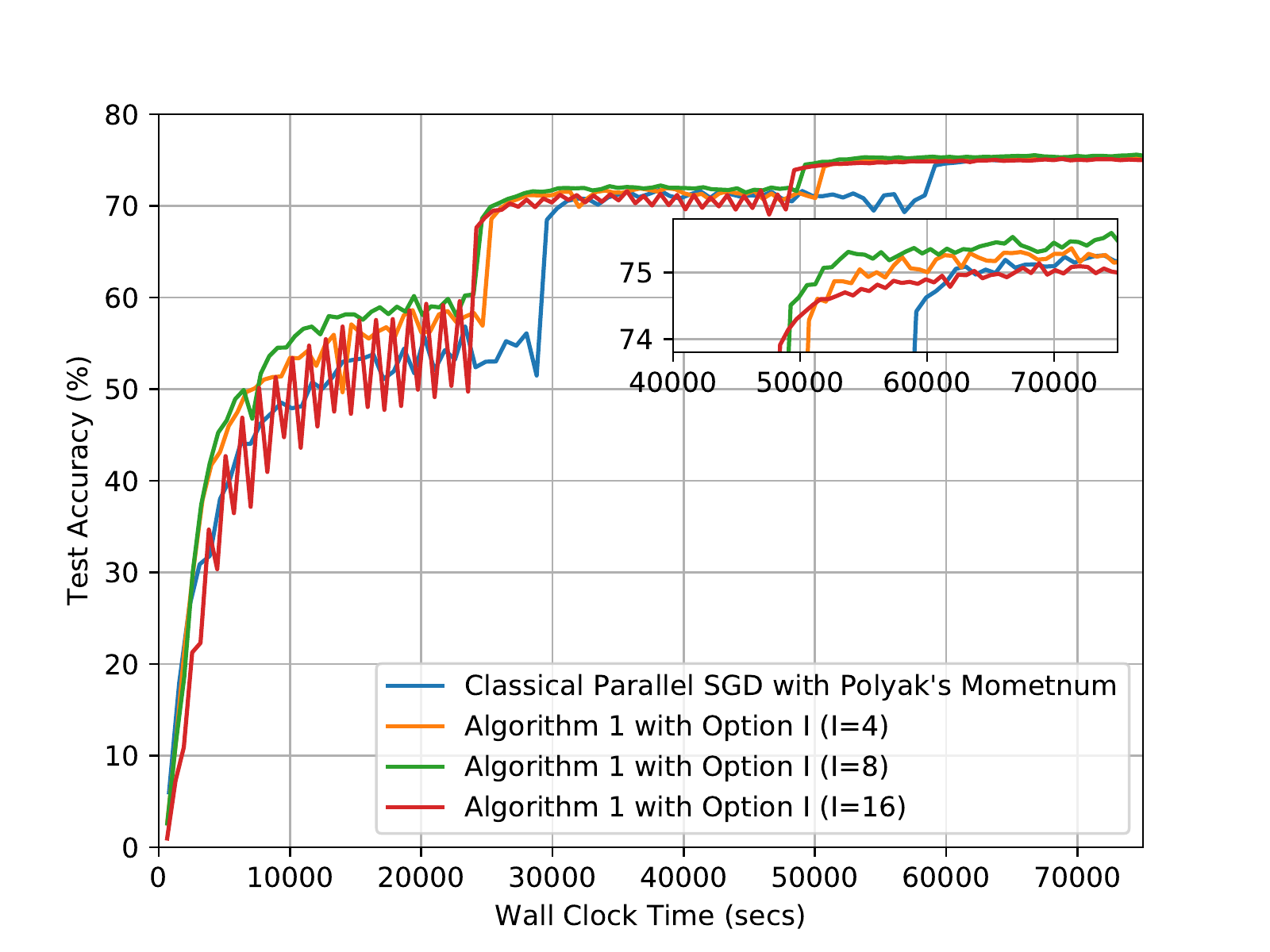}
\caption{Test accuracy v.s. wall clock time. }
 \end{subfigure}
  \caption{Algorithm \ref{alg:prsgd-momentum} with Option I: convergence v.s. wall clock time for ResNet50 over ImageNet.}
   \label{fig:imagenet-wall}
\end{figure*} 

We further verify the performance of Algorithm \ref{alg:prsgd-momentum} when used to train ResNet50 over ImageNet \cite{ImageNet}, which is an image classification task significantly harder than CIFAR10.  We run Algorithm \ref{alg:prsgd-momentum} with $I\in \{4,8,16\}$ and the classical parallel SGD with momentum on a machine with $8$ NVIDIA P100 GPUs. The local batch size at each GPU is $128$.  The learning rate is initialized to $0.2$ and is divided by $10$ when all GPUs jointly access $30$ and $60$ epochs.  The momentum coefficient is set to $0.9$. Figure \ref{fig:imagenet-step} plots the convergence of training loss and test accuracy in terms of the number of iteration steps.  This figure again verifies that Algorithm \ref{alg:prsgd-momentum} has the same convergence rate as the classical parallel momentum SGD.  To verify the benefit of skipping communication,  Figure \ref{fig:imagenet-wall} plots the convergence of training loss and test accuracy in terms of the wall clock time. 

\subsubsection{Experiments on Algorithm \ref{alg:prsgd-momentum} with Option II} \label{sec:exp-nesterov}

\begin{figure*}[t!]
\begin{subfigure}[t]{0.5\textwidth}
\includegraphics[height=2.5in]{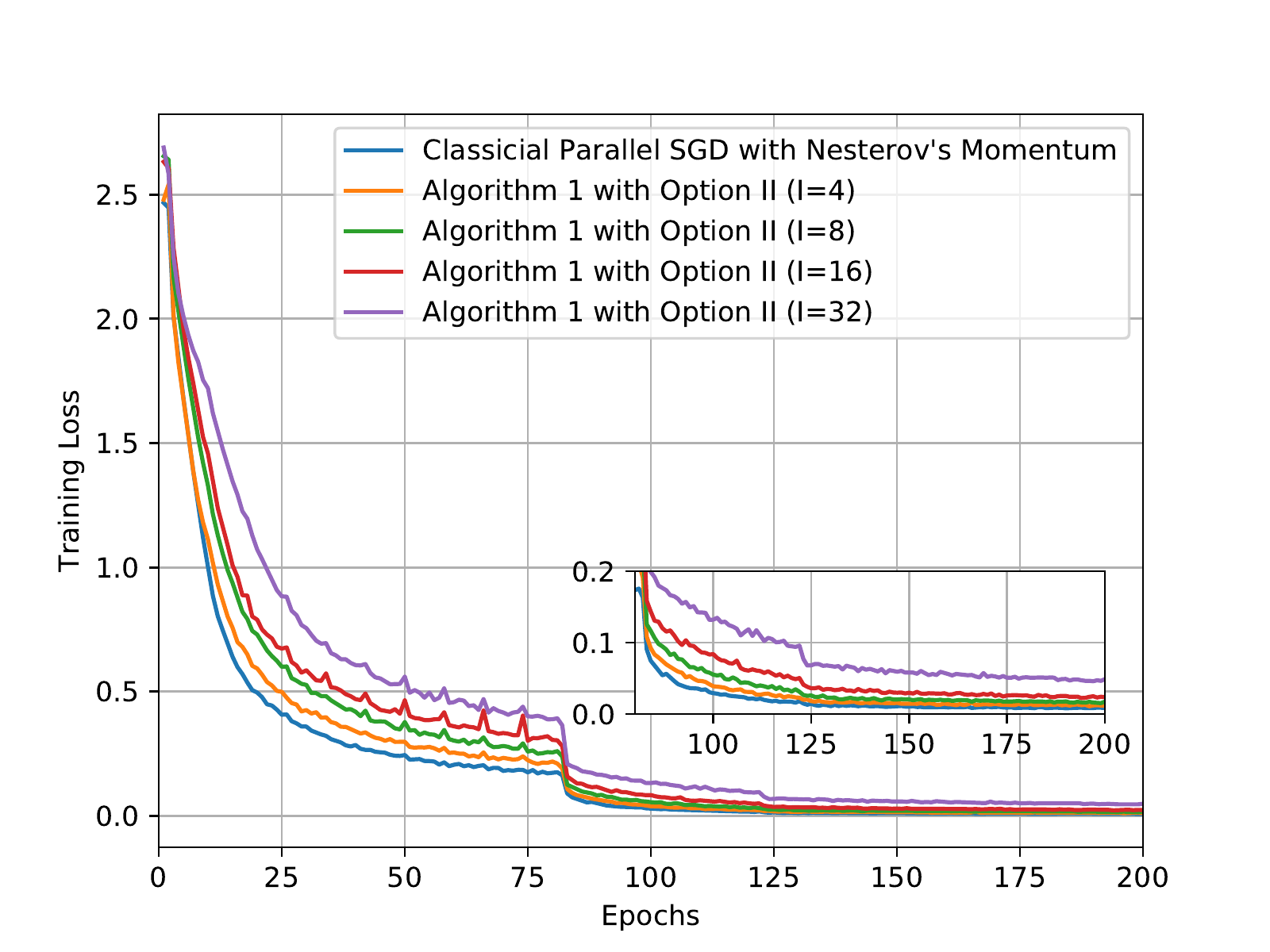}
\caption{Training loss v.s. epochs }
 \end{subfigure}
 ~
 \begin{subfigure}[t]{0.5\textwidth}
\includegraphics[height=2.5in]{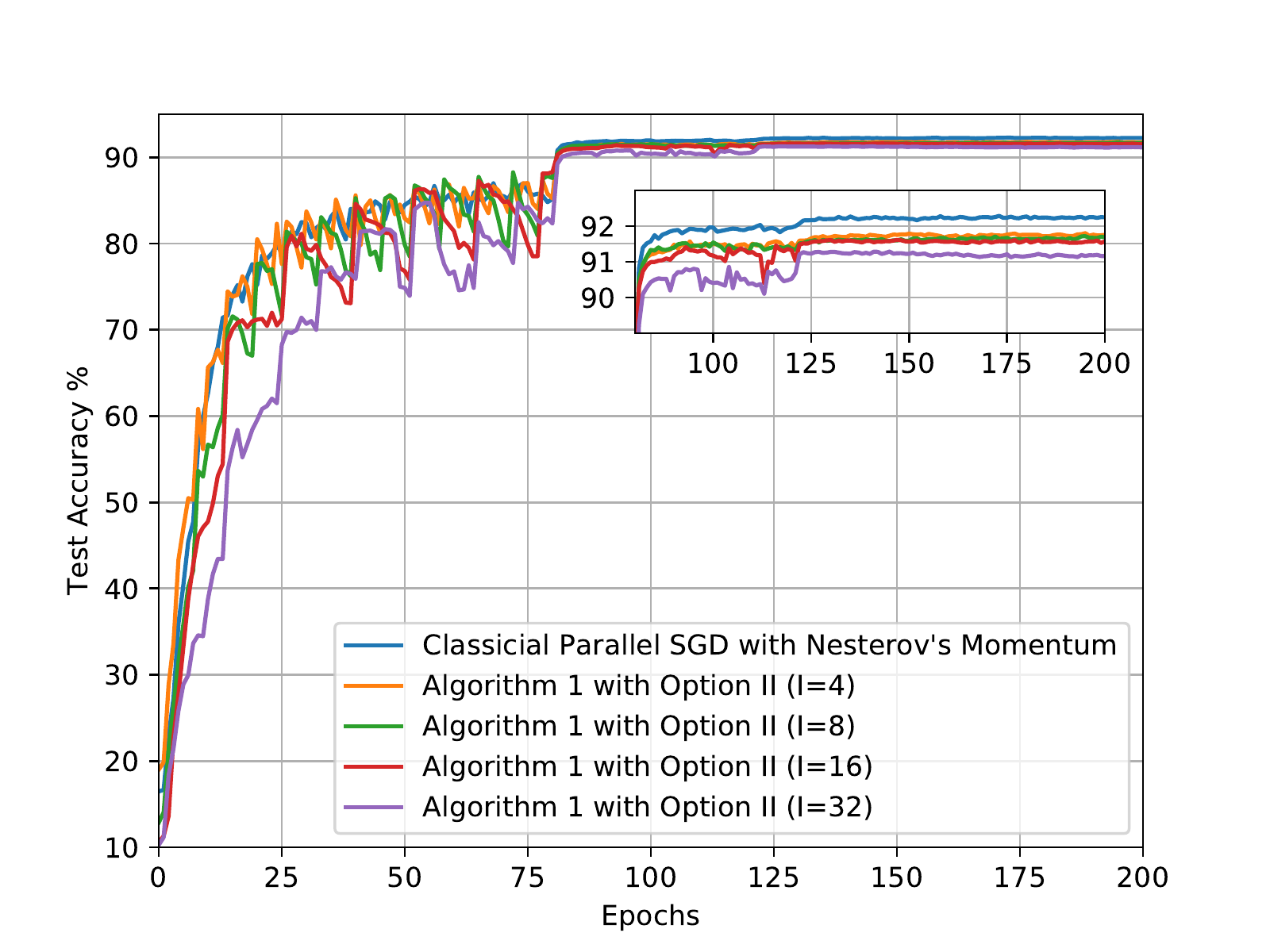}
\caption{Test accuracy v.s. epochs. }
 \end{subfigure}
  \caption{Algorithm \ref{alg:prsgd-momentum} with Option II: convergence v.s. epochs for ResNet56 over CIFAR10. }
   \label{fig:nesterov_epoch}
\end{figure*} 

\begin{figure*}[t!]
\begin{subfigure}[t]{0.5\textwidth}
\includegraphics[height=2.5in]{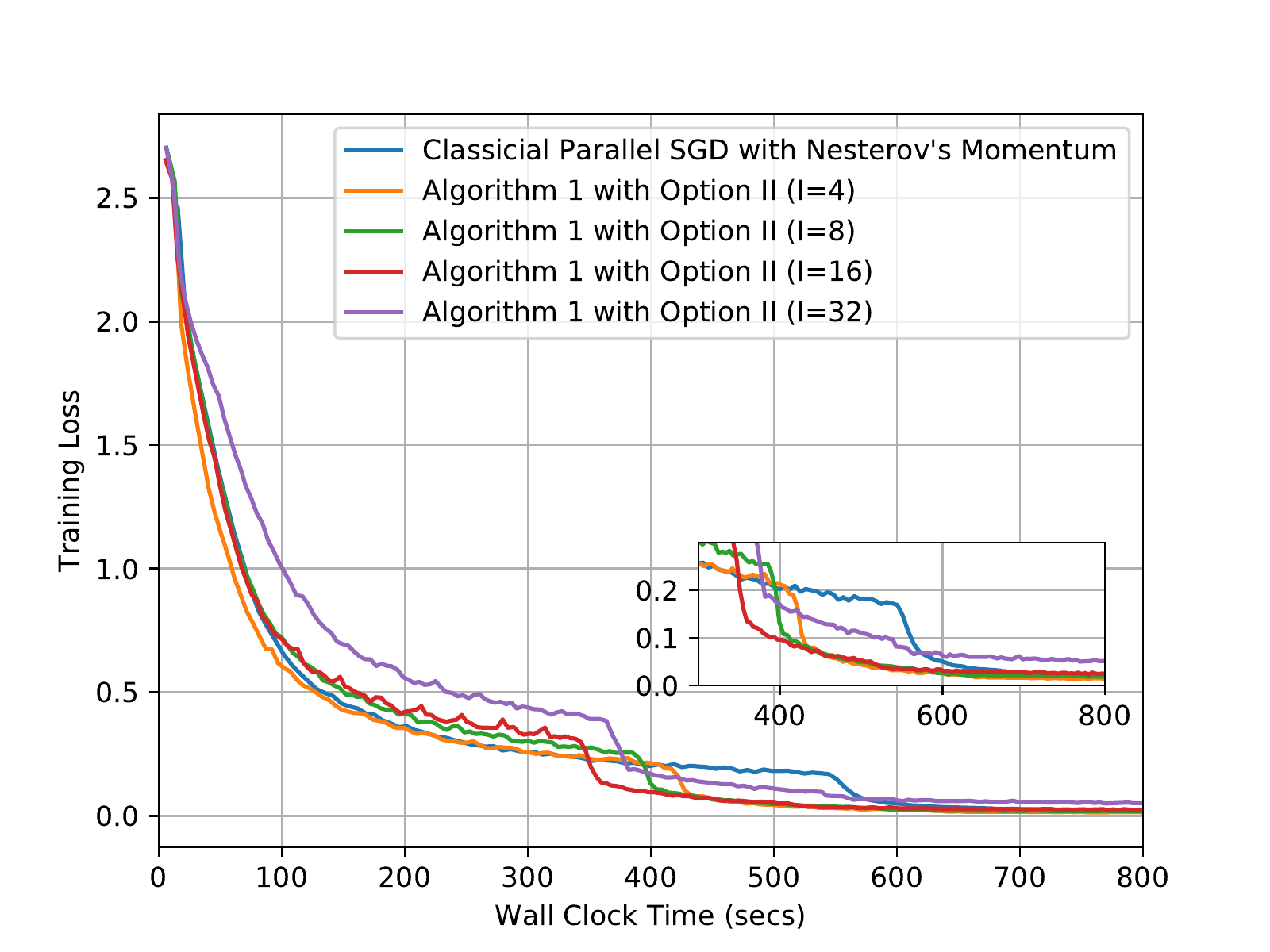}
\caption{Training loss v.s. wall clock time.}
 \end{subfigure}
 ~
 \begin{subfigure}[t]{0.5\textwidth}
\includegraphics[height=2.5in]{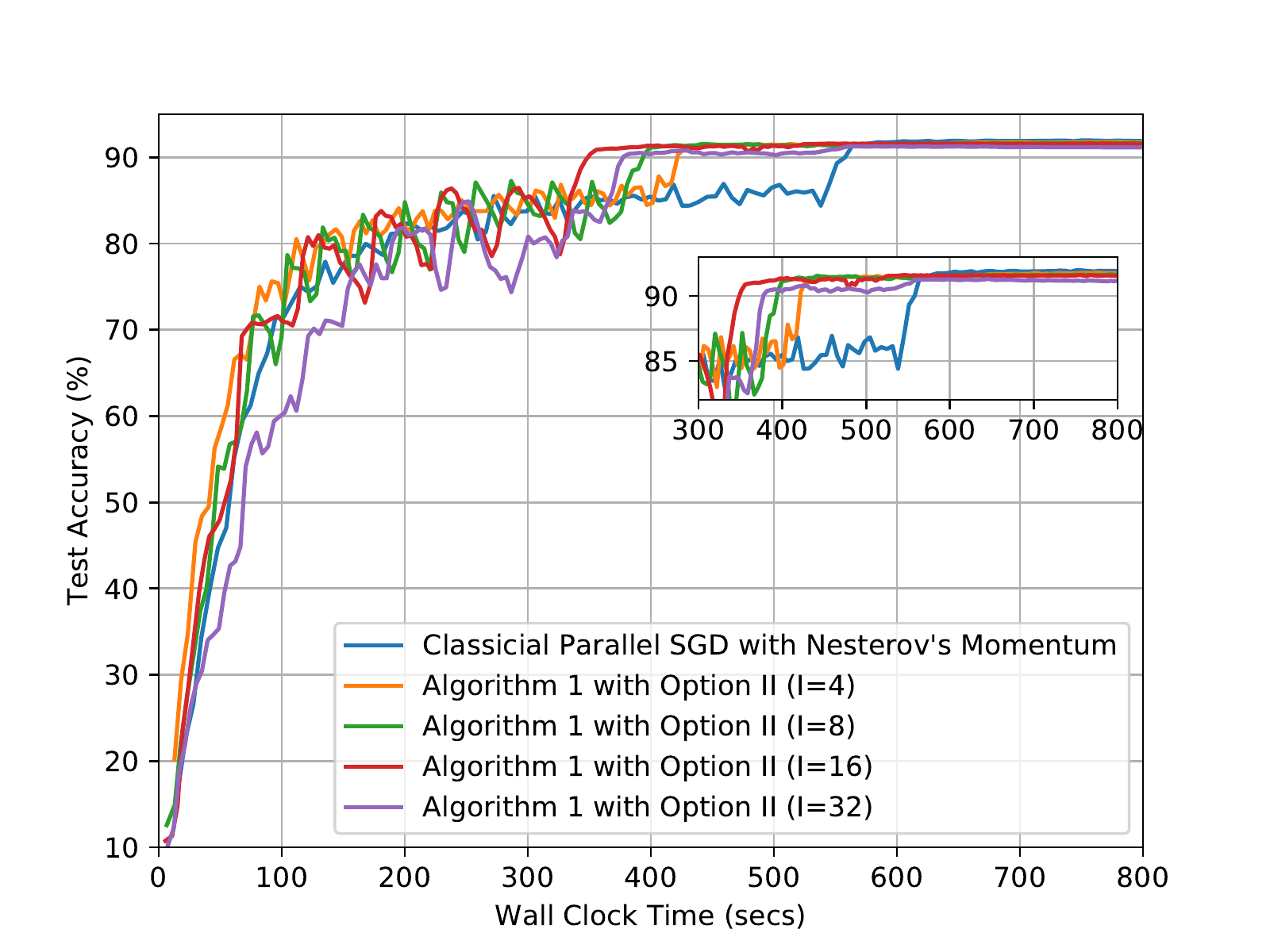}
\caption{Test accuracy v.s. wall clock time. }
 \end{subfigure}
  \caption{Algorithm \ref{alg:prsgd-momentum} with Option II: convergence v.s. wall clock time for ResNet56 over CIFAR10. }
   \label{fig:nesterov_wall}
\end{figure*}

Figures \ref{fig:nesterov_epoch} and \ref{fig:nesterov_wall} verify the convergence of Algorithm \ref{alg:prsgd-momentum} with Option II.  Their experiment configuration is identical to Figures \ref{fig:polyak_epoch} and \ref{fig:polyak_wall} except that Polyak's momentum in all algorithms is replaced with Nesterov's momentum. While Figures \ref{fig:nesterov_epoch} and \ref{fig:nesterov_wall} verify the convergence rate analysis proven in Theorem \ref{thm:nesterov-rate}, our preliminary experiments seem to suggest Nesterov's momentum is less robust to communication skipping since even a small $I$ leads to performance degradation of test accuracy.

\end{document}